\documentclass[12pt,reqno,english]{amsart}
\usepackage{fullpage}
\usepackage{amsfonts}
\usepackage{amssymb}
\usepackage[T1]{fontenc}
\usepackage[utf8]{inputenc}
\usepackage{babel}
\usepackage{stmaryrd}
\usepackage{xcolor}
\usepackage[shortlabels]{enumitem}

\usepackage{csquotes}
\usepackage{newtx} 
\usepackage{amsmath}
\usepackage{url}
\usepackage{blindtext}
\usepackage{tikz-cd}
\usepackage{mathtools}
\usepackage{graphicx}
\usepackage{amsmath,calligra,mathrsfs}
\usepackage%
 [pdfusetitle,
 bookmarksdepth=subsubsection,
 unicode]%
 {hyperref}

\usepackage{cleveref}

\newtheorem{Theorem}{Theorem}[section]
\newtheorem{Corollary}[Theorem]{Corollary}
\newtheorem{Lemma}[Theorem]{Lemma}
\newtheorem{Proposition}[Theorem]{Proposition}
\newtheorem{Conjecture}[Theorem]{Conjecture}

\theoremstyle{definition}

\newtheorem{Definition}[Theorem]{Definition}

\newtheorem{Remark}[Theorem]{Remark}

\newcommand{\id}{\textnormal{id}}
\newcommand{\aug}{\textnormal{aug}}

\newcommand{\A}{\mathbb{A}}

\newcommand{\ov}{\overline}
\newcommand\restr[2]{{
  \left.\kern-\nulldelimiterspace %
  #1 %
  \vphantom{\big|} 
  \right|_{#2} 
  }}

\newcommand{\Hom}{\textnormal{Hom}}
\newcommand*{\sheafhom}{\textnormal{H}\kern -.5pt om}
\newcommand{\Spec}{\textnormal{Spec}}

\DeclareMathOperator{\sHom}{\mathscr{H}\text{\kern -3pt {\calligra\large om}}\,}
\newcommand{\Z}{\mathbb{Z}}
\newcommand{\Q}{\mathbb{Q}}

\newcommand{\Kal}{\tn{Kal}_{F}}

\newcommand\tn{\textnormal}

\newcommand{\N}{\mathbb{N}}

\begin{document}

\title[GlobalKF]{Moduli of $G$-bundles on rigid gerbes over affine curves}

\author{Peter Dillery}
\address{University of Bonn, Mathematics Institute}
\email{dillery@math.uni-bonn.de}

\maketitle

\begin{abstract}
We geometrize the basic cohomology set $H^{1}(\tn{Kal}_{F}, G)_{\tn{basic}}$ for a global function field $F$. We do this by constructing a v-stack $\tn{Bun}_{G,F}^{e}$ which has localization maps to Fargues' analogous stack $\tn{Bun}_{G,F_{v}}^{e}$ for all places $v$ of $F$ and whose semistable locus is the disjoint union of $\tn{Bun}_{G_{b},F}$ for all $b \in H^{1}(\tn{Kott}_{F} \times_{F} \tn{Kal}_{F},G)_{\tn{basic}}$. We also prove a version of Tate-Nakayama duality for $H^{1}(\tn{Kott}_{F} \times_{F} \tn{Kal}_{F},G)_{\tn{basic}}$, which lets us state a conjectural multiplicity formula for discrete automorphic representations of $G(\mathbb{A}_{F})$ adapted to this new cohomology set. 
\end{abstract}

\tableofcontents

\section{Introduction}

\subsection{Motivation and overview}\hfill\\
The goal of this paper is to realize rigid inner forms of a reductive group $G$ over a function field $F = \mathbb{F}_{q}(C)$ for a projective curve $C$ over $\mathbb{F}_{q}$ using the moduli of $G$-bundles on gerbes over affine open subsets of $C$.  We achieve this by constructing a gerbe $\tn{Kal}_{F,\Sigma} \to C \setminus |\Sigma|$ for each finite set of places $\Sigma$ of $F$, generalizing the construction of \cite{Dillery23b}, which defines a gerbe over the generic fiber of $C$. We then synthesize the new gerbes $\tn{Kal}_{F,\Sigma}$ with the geometric framework of \cite{DLH26} in order to give an approach to the Langlands correspondence over $F$ compatible with the local geometrization program \cite{FS21} of Fargues and Scholze which is adapted to the cohomology set $H^{1}(\tn{Kal}_{F},G)_{\tn{basic}}$ from \cite{Dillery23b}, laying the foundations for studying the multiplicity of discrete automorphic representations using the aforementioned geometrization program. 

We remark that the global Kottwitz set $B(G)$ and its local analogues $B(G_{F_{v}})$ are the cohomological objects that appear naturally in the geometric approaches of \cite{FS21} and \cite{DLH26} in order to organize the inner forms of $G$. However, there is no known conjectural multiplicity formula using the global Langlands correspondence for $L$-packets constructed using an element $b \in B(G)_{\tn{basic}}$, even when $G$ has connected center. One can give such a formula \cite[\S 4.6]{KT} when $G$ satisfies the Hasse principle, but there are serious technical obstructions to a generalization beyond this case. As such, it is reasonable to believe that any attempt at studying these formulas in full generality using the geometrization program requires the set $H^{1}(\tn{Kal}_{F}, G)_{\tn{basic}}$---the key ingredient in \cite{Dillery23b} for giving the general multiplicity formula---to appear in the structure of the moduli stack of $G$-bundles. For example, \cite{HKW} uses a geometric trace formula argument to describe the multiplicities appearing in the cohomology of certain local shtuka spaces (in the Grothendieck group) corresponding to supercuspidal parameters. They use $B(G_{F_{v}})$, relying on the comparison map $B(G_{F_{v}})_{\tn{basic}} \to H^{1}(\tn{Kal}_{v}, G)_{\tn{basic}}$ from \cite{Kal18}; the relation between the analogous cohomology sets in the global function field case is much complicated. We hope that working directly with geometric objects from this paper---rather than using a comparison map---opens the door for applying analogous methods in this global situation for broad classes of groups.

We begin by reviewing the global multiplicity formula, which studies the decomposition of the Hilbert space $L_{\tn{disc}}^{2}(Z(\A_{F})G(F) \backslash G(\A_{F}))$ as a $G(\A_{F})$-representation using the global Langlands correspondence. There is a conjectural (\cite{Art02}) Langlands group $L_{F}$ which has the Weil group $W_{F}$ as a quotient such that admissible, tempered, discrete homomorphisms $\phi \colon L_{F} \to \prescript{L}{}G$ correspond to sets $\Pi_{\phi}$ of isomorphism classes of discrete, tempered automorphic irreducible representations of $G(\A_{F})$. Kottwitz \cite{Kott84} conjectured the existence of a $\mathbb{C}$-valued pairing 
$\langle - , - \rangle \colon \mathcal{S}_{\phi} \times \Pi_{\phi} \to \mathbb{C}$, where $\mathcal{S}_{\phi}$ is a finite group related to the centralizer of $\phi$ in $\widehat{G}$, such that the multiplicity $m(\pi)$ of such a fixed $G(\A_{F})$-representation $\pi$ in $L_{\tn{disc}}^{2}(Z(\A_{F})G(F) \backslash G(\A_{F}))$ is given by
\begin{equation}\label{introMF}
m(\pi) = \sum_{\phi} |\mathcal{S}_{\phi}|^{-1} \sum_{s \in \mathcal{S}_{\phi}} \langle s, \pi \rangle,
\end{equation}
where the outer sum runs over all discrete, tempered $\phi$ such that $\pi \in \Pi_{\phi}$.

Assuming a strong version of the local Langlands correspondence (the so-called \textit{refined} correspondence, see e.g. \cite[Conjecture 6.13]{Taibi22}), this pairing has been constructed in many cases, for example \cite{Kaletha18}, \cite{KT}, and \cite{Dillery23b}. All of these approaches rely on decomposing $\pi = \bigotimes'_{v \in V_{F}} \pi_{v}$ into a restricted tensor product of representations of $G(F_{v})$. Each such representation $\pi_{v}$ comes from the local $L$-packet $\Pi_{\phi_{v}}$, where $\phi_{v}$ is the localization of $\phi$ at $v$, and can be conjecturally associated to an irreducible representation $\rho_{v}$ of a finite group $\mathcal{S}_{\phi_{v}}$\footnote{There are a few different possibilities for the group $\mathcal{S}_{\phi_{v}}$ depending on one's approach, cf. \cite[\S 6.3]{Taibi22} for a summary of the approaches. We use this notation to be deliberately vague.} obtained from the centralizer of $\phi_{v}$. The pairing is then a weighted sum over all $v$ of the trace of $\rho_{v}$ averaged over $\mathcal{S}_{\phi_{v}}$.

There is a fundamental subtlety in using the local correspondences at each $v$ to define this global pairing rooted in the notion of \textit{rigid inner twists}. The local $L$-packet $\Pi_{\phi_{v}}$ does not consist of just representations $\pi_{v}$ of $G(F_{v})$, but rather of a certain enrichment of $\pi_{v}$ called a rigid inner twist obtained by adding the data of a $G^{*}$-torsor $\mathscr{T}_{v}$ on a local gerbe $\tn{Kal}_{v}$ which encodes the structure of $G$ as an inner form of a quasi-split group $G^{*}$. The papers \cite{Kaletha18} and \cite{Dillery23b} construct a global gerbe $\tn{Kal}_{F}$ and then use its $G^{*}$-torsors to create families of local torsors $\{\mathscr{T}_{v}\}_{v \in V_{F}}$ which enrich the family $\{\pi_{v}\}_{v}$ to an element of $\prod'_{v \in V_{F}} \Pi_{\phi_{v}}$ (see \eqref{eq:globalpi} for the precise restriction condition) which, using the recipe of the previous paragraph, gives the desired multiplicity formula. 

There are constructions of candidate versions of the global (\cite{Lafforgue18}) and local (\cite{GL18}, \cite{FS21}) Langlands correspondences over function fields using algebraic and adic geometry, but they have not yet been used to prove the multiplicity formula \eqref{introMF}. We expect, in view of the local-to-global definition of the pairing $\langle -, - \rangle$ given above, that, roughly speaking, doing so involves two key steps:
\begin{enumerate}[(a)] 
\item{Geometrizing the rigid refined local Langlands correspondence of \cite{Dil23a} in either the language of adic geometry or algebraic geometry;}
\item{Re-formulating the geometrization of the global Langlands correspondence so that it is compatible with the local geometrization from the first step.}
\end{enumerate}

Recent progress has been made for both steps. For (a), \cite{Fargues} replaces the v-stack of $G$-bundles on the Fargues--Fontaine curve $X_{\tn{FF}}$ over $F_{v}$ with $\tn{Bun}_{G,F_{v}}^{e}$, the v-stack of $G$-bundles on the gerbe $X_{\tn{FF}} \times_{F_{v}} \tn{Kal}_{v} \to X_{\tn{FF}}$ whose restriction to the band of $\tn{Kal}_{v}$ is valued in the center of $G$. This enables a version \cite[Conjecture 12.8]{Fargues} of the Fargues--Scholze categorical conjecture that includes $\tn{Kal}_{v}$--inner forms. On the other hand, progress towards step (b) is made by \cite{DLH26}, which defines and studies a global analogue of Hartl--Pink's \cite{HP04} quotient of the punctured unit disk---the Fargues--Fontaine curve over a local function field. A key consequence of Li-Huerta's work is to give a conjectural statement \cite[Conjecture F]{DLH26} of global Langlands duality in the language of adic geometry using the categorical local conjecture from \cite{FS21} at all places of $F$.

This paper's goal is to synthesize these two different directions of progress, by proving the following main results:

\begin{Theorem}\label{introthm1}
Let $G$ be any connected reductive group over $F$. Define $B_{e}(G)$ to be all elements of $H^{1}_{\tn{\'{e}t}}(\tn{Kott}_{F} \times_{F} \tn{Kal}_{F}, G)$ whose restriction to the band of $\tn{Kal}_{F}$ is central. Then:
\begin{enumerate}
\item{Every inner form of $G$ can be lifted to an element of $B_{e}(G)_{\tn{basic}}$.}
\item{There is a linear algebraic object $\pi_{1}(G)^{e}_{\Gamma_{F^{s}/F}}$ generalizing $(\pi_{1}(G)[V_{F^{s}}]_{0})_{\Gamma_{F^{s}/F}}$ such that one has a Tate-Nakayama duality isomorphism 
\begin{equation*}
B_{e}(G)_{\tn{basic}} \xrightarrow{\sim} \pi_{1}(G)^{e}_{\Gamma_{F^{s}/F}}
\end{equation*}
which recovers, via pullback along projections, the Tate-Nakayama duality isomorphisms for $B(G)_{\tn{basic}}$ and $H^{1}(\tn{Kal}_{F}, G)_{\tn{basic}}$ and is compatible (via the local duality isomorphism from \cite{Fargues}) with localization at each place $v$.}
\item{By localizing certain elements $\mathscr{T} \in Z^{1}_{\tn{\'{e}t}}(\tn{Kott}_{F} \times_{F} \tn{Kal}_{F}, G)$ at each place $v$, we can, using the $\tn{Kott}^{1/\infty}_{v}$-refined correspondence at all $v$, define a candidate for the multiplicity pairing $\langle - , - \rangle \colon \mathcal{S}_{\phi} \times \Pi_{\phi} \to \mathbb{C}$ discussed above for $G$.} 
\end{enumerate}
\end{Theorem}

The above result gives a multiplicity formula which can be stated for an arbitrary connected reductive group $G$ over $F$. The reason to work with the gerbe $\tn{Kott}_{F} \times_{F} \tn{Kal}_{F}$ rather than $\tn{Kal}_{F}$ is that we can geometrize its basic $G$-cohomology, in the following sense:

\begin{Theorem}\label{introthm2}
For $G$ a split connected reductive group over a global function field $F = \mathbb{F}_{q}(C)$, there is a small v-stack $\tn{Bun}_{G,F}^{e}$ over $\ov{\mathbb{F}_{q}}$ such that:
\begin{enumerate}
\item{There is a canonical bijection $\tn{Bun}_{G,F}^{e}(\tn{Spd}(\overline{\mathbb{F}_{q}})) \xrightarrow{\sim} B_{e}(G)$.}
\item{The semistable locus of $\tn{Bun}_{G,F}^{e}$ equals $\bigsqcup_{b \in B_{e}(G)_{\tn{basic}}} \tn{Bun}_{G,F}^{e,b}$.}
\item{There is a canonical localization map $\tn{Bun}_{G,F}^{e} \to \prod'_{v \in V_{F}} \tn{Bun}_{G,F_{v}}^{e}$, where $\tn{Bun}_{G,F_{v}}^{e}$ is as in \cite{Fargues}.}
\end{enumerate}
\end{Theorem}
The definition of $\tn{Bun}_{G,F}^{e}$ is given as the direct limit of small Artin v-stacks $\tn{Bun}_{G,U}^{e}$ over $\overline{\mathbb{F}_{q}}$ for each $U \subseteq C$ affine open. We can give the statement of the theorem for more general $G$ than split groups, but we assume split in the introduction for clarity of exposition. We remark that some of the parts of Theorem \ref{introthm2} are proved in \cite[Appendix C]{DLH26}, but are consequences of combining the work done here with the main body of \cite{DLH26}.

The proof of Theorem \ref{introthm1} relies on developing the theory of the cohomology set $B_{e}(G)$ defined above. The most difficult part is the Tate-Nakayama duality result, which is obtained in part by giving (Proposition \ref{prop:tannakian}) an explicit Tannakian description of the gerbe $\tn{Kott}_{F} \times_{F} \tn{Kal}_{F}$. We remark that, by taking the inclusion map $H^{1}(\tn{Kal}_{F}, G)_{\tn{basic}} \to B_{e}(G)_{\tn{basic}}$, this result gives a substantially shorter proof of the analogous duality result for $\tn{Kal}_{F}$, which is \cite[Theorem 4.11]{Dillery23b}. 

Theorem \ref{introthm2} is proved in two steps: First, since the v-stack $\tn{Bun}_{G,F}$ from \cite{DLH26} is defined as the direct limit of the stacks $\tn{Bun}_{G,U}$ for affine open subsets $U = C \setminus |\Sigma|$ of the curve $C$, we construct an analogue of the gerbe $\tn{Kal}_{F}$ which is defined over $U$, denoted by $\tn{Kal}_{F,\Sigma}$. The existence of such a gerbe is nontrivial, and relies on combining a projective limit of global classes on $U$ with a localization argument involving complexes of tori over $U$. We then use the gerbe $\tn{Kal}_{F,\Sigma}$ to define the stack $\tn{Bun}_{G,U}^{e}$ by sending a perfectoid space $S$ over $\mathbb{F}_{q}$ to the set of $G$-torsors on $U_{S}^{\tn{an}} \times_{\tn{Spa}(O_{F,\Sigma})} \tn{Kal}_{F,\Sigma}^{\tn{ad}}$ equipped with an isomorphism to their $(\tn{Frob}_{S} \times \text{id})$-twist (and whose restriction to the band of $\tn{Kal}_{F,\Sigma}^{\tn{ad}}$ is central), where $U_{S}^{\tn{an}}$ is as in \cite[Proposition 1.4]{DLH26}. The key result for understanding the geometry of this new stack is Proposition \ref{Nclopenprop}, which shows (Corollary \ref{prop:newtondecomp}) that $\tn{Bun}_{G,F}^{e}$ is isomorphic to a disjoint union of stacks $\tn{Bun}_{G',F}$ for varying inner forms $G'$ of $G$. 

It is natural to ask why one does not try to carry out an analogue of Theorem \ref{introthm2} for the gerbe $\tn{Kal}_{F}$ by itself, rather than appending $\tn{Kott}_{F}$, since all that one needs for a general statement of the multiplicity formula is $\tn{Kal}_{F}$. While this is possible in theory, we do not know how to do this, even in the local case at a place $v$. For example, the work \cite{DK25} shows that isomorphism classes of rank-$n$ vector bundles on the canonical $\mathbb{G}_{m}$-torsor over the Fargues--Fontaine curve corresponding to $\mathcal{O}(-1)$ are in bijection with $H^{1}_{\tn{\'{e}t}}(\tn{Kal}_{v}^{(F_{v})},\mathrm{GL}_{n})$, where $\tn{Kal}_{v}^{(F_{v})}$ is the gerbe\footnote{This gerbe can be thought of as a ``level-$F_{v}$'' approximation of $\tn{Kal}_{v}$.} corresponding to the canonical class $-1 \in \widehat{\Z} = H^{2}_{\tn{fppf}}(F_{v}, \varprojlim_{n} \mu_{n})$. The geometry is thus substantially different---if we write the schematic Fargues--Fontaine curve as $\tn{Proj}(A)$ for Fontaine's ring $A$, then the space on which these vector bundles are defined is $\tn{Spec}(A) \setminus \{I_{+}\}$, where $I_{+}$ denotes the ideal of all positive-degree elements. We also do not yet know how to obtain an analogue of this bijection for $H^{1}_{\tn{\'{e}t}}(\tn{Kal}_{F},\mathrm{GL}_{n})$, the cohomology of the actual Kaletha gerbe.

We will initiate the study of an analogue of shtukas adapted to the stack $\tn{Bun}_{G,F}^{e}$ in future work. It would also be interesting to see if one can define the ``restricted ramification'' gerbe $\tn{Kal}_{F,\Sigma}$ in the number field setting.

\subsection{Structure of the paper}\hfill\\
In \S \ref{sec:2} we begin by defining the gerbes $\tn{Kal}_{F,\Sigma}$ by gluing together two gerbes defined over an open cover of $U = C \setminus |\Sigma|$ and then using the local canonical classes from \cite{Dil23a} at each $v \in \Sigma$ to uniquely characterize the resulting class in the projective limit (over all $\Sigma$-unramified $E/F$). This last step requires developing a basic theory of complexes of tori over the rings of $\Sigma$-integers. We conclude this section by defining the gerbe $\tn{Kott}_{F}^{1/\infty}$ and proving some vanishing results about its cohomology in degree two. We then study the set of $G$-torsors on these gerbes in \S \ref{sec:H1} by constructing an explicit Tannakian category over $F$ corresponding to $\tn{Kott}_{F}^{1/\infty}$ and then, using this Tannakian description, we prove an analogue of the Tate-Nakayama isomorphism for $H^{1}_{\tn{\'{e}t}}(\tn{Kott}_{F}^{1/\infty},T)$ with $T$ a torus. The torus case allows us to construct an analogue of the Kottwitz map for $B_{e}(G)$ that restricts to an isomorphism on $B_{e}(G)_{\tn{basic}}$, giving a linear-algebraic description of this cohomology set when $G$ is connected and reductive.

We use the gerbes $\tn{Kal}_{F,\Sigma}$ in \S \ref{sec:geom} to define the v-stacks $\tn{Bun}_{G,U}^{e}$ and $\tn{Bun}_{G,F}^{e}$, construct a localization map $\tn{Bun}_{G,F}^{e} \to \prod_{v}' \tn{Bun}^{e}_{G,F_{v}}$, and prove that $\tn{Bun}_{G,U}^{e}$ fits into the Cartesian square where all the local contributions are integral away from $\Sigma$. We then prove that $\tn{Bun}_{G,U}^{e}$ is isomorphic to a disjoint union of Li-Huerta's v-stacks $\tn{Bun}_{G',U}$ for varying $G'$; this decomposition lets us deduce the descriptions of $\tn{Bun}_{G,F}^{e}(\tn{Spd}(\ov{\mathbb{F}_{q}}))$ and $\tn{Bun}_{G,F}^{e,\tn{ss}}$ from Theorem \ref{introthm2}. The last part, \S \ref{sec:mult}, gives a conjectural statement of the $\tn{Kott}^{1/\infty}_{v}$-refined correspondence at each place $v$ and then localizes an element $B_{e}(G)_{\tn{basic}}$ at all places to state the $\tn{Kott}^{1/\infty}_{F}$-adapted conjectural global multiplicity formula using the local correspondences and the global duality map for $B_{e}(G)_{\tn{basic}}$. 

\subsection{Notation and conventions}\hfill\\
We denote by $C$ a smooth, projective curve over $\mathbb{F}_{q}$ corresponding to a global function field $F$. We denote by $F^{s}$ a separable closure of $F$ inside a fixed algebraic closure $\ov{F}$. For $E/F$ a Galois extension, we denote by $\Gamma_{E/F}$ the Galois group of $E$ over $F$. For a ring homomorphism $R \to S$ and commutative $R$-group scheme $A$, one has the \v{C}ech differential maps $A(S^{\bigotimes_{R} i+1}) \to A(S^{\bigotimes_{R} i+2})$, and we denote by $\check{H}^{i}(S/R, A)$ the $i$th cohomology group of the corresponding complex (cf. \cite[Definition 2.11]{Dil23a}). When $A$ is a multiplicative group scheme of finite type over $R$ and $\Spec(S) \to \Spec(R)$ is a pro-fppf cover satisfying the hypotheses of \cite[Proposition 2.6]{Dillery23b} (as will always be the case in this paper) there are canonical isomorphisms $\check{H}^{i}(S/R, A) \xrightarrow{\sim} H^{i}_{\text{fppf}}(R, A)$ for all $i$. When $R=F$ is a field and $G$ is a smooth group scheme over $F$, we write $H^{1}(F,G)$ to denote $H^{1}_{\tn{\'{e}t}}(F, G)$, which agrees with analogous cohomology set for the fppf and fpqc topologies. When we say a ``gerbe over $F$'', we always mean a gerbe $\mathcal{E} \to \tn{Spec}(F)_{\tn{fppf}}$, where for an object $X$ we denote by $X_{?}$ the $?$-site over $X$ for $? \in \{\tn{fpqc}, \tn{fppf}, \tn{v}\}$.

We denote by $V_{F}$ the set of all places of $F$. For $\Sigma \subseteq V_{F}$ and $E/F$ an algebraic field extension, we denote by $\Sigma_{E}$ all places of $E$ above $\Sigma$. For $\Sigma'$ a set of places of $E$, we denote by $\Sigma'_{F}$ all places of $F$ below $\Sigma'$. The notation $\mathbb{A}_{F}$ denotes the ring of ad\`{e}les of $F$, and $\mathbb{A}_{F,\Sigma}$ denotes the subring of elements which are integral away from $\Sigma$. Moreover, $O_{F,\Sigma}$ denotes the ring of $\Sigma$-integers, and for $E/F$ we sometimes write $O_{E,\Sigma}$ rather than $O_{E,\Sigma_{E}}$. Finally, for a topological space $X$ and a topological abelian group $A$, $\mathcal{C}(X,A)$ denotes the group of continuous functions from $X$ to $A$.

\subsection{Acknowledgements}\hfill\\
We thank Siyan Daniel Li-Huerta for many helpful discussions, including one which led to the definition of $\tn{Kal}_{F,\Sigma}$, and for sending an early draft of \cite{DLH26}. We also thank Arnaud Eteve for helping to clarify several geometric details. In addition, we thank Alexander Bertoloni Meli for useful conversations about \cite{Fargues} and the associated refined local correspondence.

\section{Gerbes and their cohomology}\label{sec:2} 
This section recalls the definition of some gerbes over $F$ from the literature and then uses them to construct new global gerbes.
\subsection{A review of global gerbes}

\subsubsection{The Kottwitz gerbe}
Fix $E/F$ a finite Galois extension and $\Sigma^{(E)}$ a set of places  of $F$ with preimage $\Sigma^{(E)}_{E}$ in $V_{E}$ such that $E$ is unramified outside of $\Sigma$ and $\tn{Cl}(O_{E,\Sigma^{(E)}_{E}}) = \{1\}$; denote by $T_{2,E,\Sigma^{(E)}}$ the torus with character group $\mathbb{Z}[\Sigma^{(E)}_{E}]$, by $T_{3,E,\Sigma^{(E)}}$ the pro-torus with character group $\mathbb{Z}[\Sigma^{(E)}_{E}]_{0}$. Kottwitz constructed in \cite{Kot14} a canonical class $[\alpha^{\tn{Kott}}_{E/F,\Sigma^{(E)}}] \in H^{2}(\Gamma_{E/F}, T_{3,E,\Sigma^{(E)}}(E))$ for each $E$, along with transition maps $T_{3,K,\Sigma^{(K)}} \to T_{3,E,\Sigma^{(E)}}$.

Kottwitz then shows that we obtain a canonical class
\begin{equation*}
\varprojlim_{\Sigma^{(E)},E} [\alpha^{\tn{Kott}}_{E/F,\Sigma^{(E)}}] \in \varprojlim_{\Sigma^{(E)},E} H^{2}(\Gamma_{E/F}, T_{3,E,\Sigma^{(E)}}(E)),
\end{equation*}
where the limit is over all finite Galois $E/F$ and an exhaustive nested system of places $\Sigma^{(E)}$; we choose a group extension representing this class, denoted by $\tn{Kott}_{F}$, and call it a \textit{global Kottwitz gerbe}, which is an extension of $\Gamma_{\ov{F}/F}$ by $\mathbb{D}(\ov{F})$, where $\mathbb{D} := \varprojlim T_{3,E,\Sigma^{(E)}}$. 

\subsubsection{The Kaletha gerbe}\label{Kalgerbedef}
There is a parallel construction of a different gerbe: Fix an exhaustive system of finite Galois extensions $\{E_{k}/F\}_{k \in \mathbb{N}}$ along with a cofinal multiplicative system of natural numbers $\{n_{k}\}$, for each $E_{k}$ a set of places $\Sigma^{(E_{k})}$ of $V_{F}$ containing all places which ramify in $E_{k}$ and such that $\tn{Cl}(O_{E_{k},V^{(E_{k})}})$ is trivial, and a set of lifts $\dot{\Sigma}^{(E_{k})} \subseteq V_{E_{k}}$ of $\Sigma^{(E_{k})}$ such that, for $\ell \geq k$, we have $\dot{\Sigma}^{(E_{k})} \subseteq (\dot{S}^{(E_{\ell})})_{E_{k}}$ and $\dot{v}_{l}|_{E_{k}} = \dot{v}_{k}$---this implies that there is a unique $\dot{v} \in V_{F^{s}}$ lifting each $\dot{v}_{k}$, and denote the set of all such $\dot{v}$ by $\dot{V}$. We also insist that for any $w \in V_{F^{s}}$ there is some $\dot{v} \in \dot{V}$ with $(\Gamma_{F^{s}/F})_{w} = (\Gamma_{F^{s}/F})_{\dot{v}}$.

One then defines $\dot{P}_{k}$ to be the multiplicative $O_{E,\Sigma^{(E_{k})}}$-group scheme with character module given as follows: First denote by $(\frac{1}{n_{k}}\mathbb{Z}/\mathbb{Z})[\Gamma_{E_{k}/F} \times \dot{\Sigma}^{(E_{k})}]_{0,0}$ the subgroup killed by both augmentation maps, and then take $\dot{M}_{k}$ the submodule of all elements such that the coefficient of $[(\sigma, w)]$ is zero unless $\sigma^{-1}(w) \in \dot{\Sigma}^{(E_{k})}$. One can define a canonical class $\xi_{E_{k},\dot{\Sigma}^{(E_{k})},n_{k}} \in \check{H}^{2}(O_{E_{k},\Sigma^{(E_{k})}}^{\tn{perf}}/O_{F,\Sigma^{(E_{k})}}, \dot{P}_{k})$ which are compatible across a projective system of transition maps $\dot{P}_{\ell} \to \dot{P}_{k}$, and one defines $P := \varprojlim \dot{P}_{k}$. Finally, one defines a canonical class $\xi$ in $H_{\tn{fppf}}^{2}(F, P)$ whose image in $\varprojlim_{k} H_{\tn{fppf}}^{2}(F, \dot{P}_{k})$ is the corresponding projective system of canonical classes (to pin down uniqueness, we need an additional condition which we omit for the moment---see Theorem \ref{canonclassthm} for details). Denote a representative of this gerbe by $\tn{Kal}_{F}$---this gerbe is, up to isomorphism, independent of any choices and the set of isomorphism classes of its \'{e}tale $G$-torsors for a finite-type group scheme $G$ is unique up to canonical isomorphism (by \cite[Proposition 4.8]{Dillery23b}).

One can also define (cf. \cite[\S 3.1]{Dil23a}), for a place $v \in V_{F}$, a local analogue of this gerbe, denoted by $\tn{Kal}_{v}$, which is banded by $u_{v} := \varprojlim_{k} \tn{Res}_{E/F_{v}}(\mu_{n_{k}})/\mu_{n_{k}}$, where as above $(E_{k},n_{k})$ is a cofinal system of finite Galois extensions and natural numbers. There is a non-canonical localization map $\tn{Kal}_{v} \xrightarrow{\tn{loc}_{v}} (\tn{Kal}_{F})_{F_{v}}$ which induces a canonical map on cohomology $H_{\tn{\'{e}t}}^{1}(\tn{Kal}_{F}, G) \to H_{\tn{\'{e}t}}^{1}(\tn{Kal}_{v}, G)$ for a finite-type $F$-group scheme $G$, also denoted by $\tn{loc}_{v}$, see \cite[\S 4.1]{Dillery23b} for details.

\subsection{Unramified complexes of tori}\label{sec:unram} \hfill\\
We need to generalize some constructions and duality results from \cite[\S A.3, A.4]{Dillery23b} to allow for restricted ramification---the reader interested only in the main results and ideas of this paper may want to skip this subsection on a first reading. Fix a finite set of places $\Sigma$ of $F$, $T \xrightarrow{f} U$ a complex of tori over $O_{F,\Sigma}$ split over $O_{\Sigma} := \varinjlim_{E \subseteq F_{\Sigma}} O_{E,\Sigma_{E}}$, where $F_{\Sigma}/F$ is the maximal Galois extension of $F$ unramified outside of $\Sigma$, and also for every place $v$ of $F$ an embedding $\ov{F} \to \ov{F_{v}}$, with corresponding place $\dot{v} \in V_{\ov{F}}$, determining embeddings $F_{\Sigma} \to (F_{\Sigma})_{\dot{v}}$. By abuse of notation, we will denote $\dot{v}|_{F_{\Sigma}}$ also by $\dot{v}$. Set $\A_{\Sigma} = \A_{F_{\Sigma},(\Sigma)_{F_{\Sigma}}}^{\tn{perf}}$, and define $H^{i}(\A_{\Sigma}/\A_{F,\Sigma}, T \xrightarrow{f} U)$ to be the hypercohomology of the double complex associated to the \v{C}ech cover $\A_{\Sigma}/\A_{F,\Sigma}$.

Define $\bar{H}^{i}(\A_{\Sigma}/\A_{F,\Sigma}, T \xrightarrow{f} U)$ to be the hypercohomology of the following double complex
\[ 
\begin{tikzcd}
\frac{T(\A_{\Sigma})}{T(O_{\Sigma}^{\tn{perf}})} \arrow{r} \arrow{d} & \frac{T(\A_{\Sigma} \otimes_{\A_{F,\Sigma}} \A_{\Sigma})}{T(O_{\Sigma}^{\tn{perf}} \otimes_{O_{F,\Sigma}} O_{\Sigma}^{\tn{perf}})} \arrow{r} \arrow{d} & \frac{T(\A_{\Sigma} \otimes_{\A_{F,\Sigma}} \A_{\Sigma} \otimes_{\A_{F,\Sigma}} \A_{\Sigma})}{T(O_{\Sigma}^{\tn{perf}} \otimes_{O_{F,\Sigma}} O_{\Sigma}^{\tn{perf}} \otimes_{O_{F,\Sigma}} O_{\Sigma}^{\tn{perf}})} \arrow{r} \arrow{d} & \dots \\
\frac{U(\A_{\Sigma})}{U(O_{\Sigma}^{\tn{perf}})} \arrow{r} & \frac{U(\A_{\Sigma} \otimes_{\A_{F,\Sigma}} \A_{\Sigma})}{U(O_{\Sigma}^{\tn{perf}} \otimes_{O_{F,\Sigma}} O_{\Sigma}^{\tn{perf}})} \arrow{r} & \frac{U(\A_{\Sigma} \otimes_{\A_{F,\Sigma}} \A_{\Sigma} \otimes_{\A_{F,\Sigma}} \A_{\Sigma})}{U(O_{\Sigma}^{\tn{perf}} \otimes_{O_{F,\Sigma}} O_{\Sigma}^{\tn{perf}} \otimes_{O_{F,\Sigma}} O_{\Sigma}^{\tn{perf}})} \arrow{r}  & \dots,
\end{tikzcd}
\]

As in \cite[Lemma A.6]{Dillery23b}, we are free to replace $O_{\Sigma}^{\tn{perf}}$ by $O_{\Sigma}$ and $\A_{\Sigma}$ by $\A_{F_{\Sigma},(\Sigma)_{F_{\Sigma}}}$ without changing any of these hypercohomology groups (up to canonical isomorphism). The following decomposition is from Shapiro's Lemma:

\begin{Proposition}\label{prop:resprod}
We have a canonical isomorphism
\begin{equation*}
H^{i}(\A_{\Sigma}/\mathbb{A}_{F,\Sigma}, T \xrightarrow{f} U) \xrightarrow{\sim} \prod_{v \in \Sigma} H^{i}((F_{\Sigma})_{\dot{v}}/F_{v}, T \xrightarrow{f} U) \times \prod_{v \notin \Sigma} H^{i}(O_{v}^{\tn{nr}}/O_{v}, T \xrightarrow{f} U).
\end{equation*}
In particular, when $i \geq 2$ the above product is really a direct sum.
\end{Proposition}

All of these groups can be topologized using completely analogous arguments as in \cite[Appendix A]{Dillery23b}. The following result will be crucial to our arguments:

\begin{Lemma}\label{divlemma}
When $|\Sigma| \geq 2$, the group $(F_{\Sigma})_{\dot{v}}^{\tn{perf},\times}$ is divisible for $v \in \Sigma$.
\end{Lemma}

\begin{proof}
First, note that $O_{(F_{\Sigma})_{\dot{v}}}^{\times}$ is $n$-divisible for $(n,p)=1$ by the structure of unit groups of rings of integers over local function fields, and thus $O_{(F_{\Sigma})_{\dot{v}}}^{\tn{perf},\times}$ is divisible. The claimed divisibility will then follow if we can show that the value group of $(F_{\Sigma})_{\dot{v}}$ is $n$-divisible for $(n,p)=1$. To see this value group claim, we note that in the ring $O_{F,\Sigma \setminus \{v\}}$ the ideal $\mathfrak{p}_{v}$ becomes principal over some finite unramified extension $E/F$ according to the proof of \cite[Proposition 8.3.6]{NSW}; now $E \subseteq F_{\Sigma}$, and taking a generator $\varpi$ of $\mathfrak{p}_{v} \cdot O_{E, \Sigma \setminus \{v\}}$ allows us to take $n$th roots without introducing ramification outside of $\Sigma$.
\end{proof}

\begin{Proposition}\label{dualitypropS}
We have a canonical pairing
\begin{equation}\label{barH1pairglob}
\bar{H}^{r}(\A_{\Sigma}/\A_{F,\Sigma}, T \xrightarrow{f} U) \times H^{3-r}(\Gamma_{F_{\Sigma}/F}, X^{*}(U) \to X^{*}(T)) \to \mathbb{Q}/\mathbb{Z},
\end{equation}
which is a perfect pairing for $r \geq 2$.  
\end{Proposition}

\begin{proof}
The pairing is constructed by identifying
\begin{equation*}
\mathcal{H}^{2}(\mathbb{G}_{m}[(\A_{\Sigma}^{\tn{sep}})^{\bigotimes_{\A_{\Sigma}}(\bullet+1)}]/\mathbb{G}_{m}[(O_{\Sigma})^{\bigotimes_{O_{F,\Sigma}}(\bullet+1)}])
\end{equation*}
with $H^{2}(\Gamma, C_{\Sigma})$ (where $C_{\Sigma} = \varinjlim_{F \subseteq K \subseteq F_{\Sigma}} C_{K}(\Sigma)$ is the universal $\Sigma$-id\`{e}le class group) and then taking the global $\Sigma$-invariant map. The proof that it's a perfect pairing for $r \geq 2$ is the same as the proof of \cite[Lemma A.2.A]{KS99}.
\end{proof}

We can also define a pairing 
\begin{equation}\label{barH1pairglobbis}
\bar{H}^{1}(\A_{\Sigma}/\A_{F,\Sigma}, T \xrightarrow{f} U) \times H^{1}(W_{F_{\Sigma}/F}, \widehat{U} \xrightarrow{\hat{f}} \widehat{T}) \to \mathbb{C}^{\times},
\end{equation}
where $W_{F_{\Sigma}/F}$ is the  $F_{\Sigma}$-relative Weil group of $F$ (built using the inverse limit of $W_{K/F}$ for all finite Galois $F \subseteq K \subseteq F_{\Sigma}$) and the right-hand group is as defined in \cite[\S A.4]{Dillery23b}.

The pairing is defined via the composition of the canonical map 
\begin{equation*}
\bar{H}^{1}(\A_{\Sigma}/\A_{F,\Sigma}, T \xrightarrow{f} U) \to \bar{H}^{1}(\A_{F_{\Sigma}}/\mathbb{A}, T \xrightarrow{f} U)
\end{equation*}
with the pairing from \cite[\S A.4]{Dillery23b}, which makes sense, since it's defined using a direct limit of finite-level pairings (via inflation). The same argument gives an analogous local pairing 
\begin{equation}\label{barH1pairloc}
H^{1}((F_{\Sigma})^{\tn{perf}}_{\dot{v}}/F_{v}, T \xrightarrow{f} U) \times H^{1}(W_{(F_{\Sigma})_{\dot{v}}/F}, \widehat{U} \xrightarrow{\hat{f}} \widehat{T}) \to \mathbb{C}^{\times}
\end{equation}
which will also be used.

First, a basic but important result is:

\begin{Lemma}\label{keytechlem}
For any $i$ and $T$ as above the natural map
\begin{equation*}
\varinjlim_{F \subseteq K \subseteq F_{\Sigma}} \bar{H}^{i}(\A_{K,\Sigma}/\A_{F,\Sigma}, T) \to \varinjlim_{F \subseteq K \subseteq F_{\Sigma}} \bar{H}^{i}(\A_{K}/\A_{F}, T)
\end{equation*}
is an isomorphism.
\end{Lemma}

\begin{proof}
This follows from the fact that the map
\begin{equation*}
\varinjlim_{F \subseteq K \subseteq F_{\Sigma}} T(\A_{K,\Sigma})/T(O_{K,\Sigma}) \to \varinjlim_{F \subseteq K \subseteq F_{\Sigma}} T(\A_{K})/T(K)
\end{equation*}
is an isomorphism (by \cite[Proposition 8.3.6]{NSW}).
\end{proof}

\begin{Proposition}\label{appH1pairing}
The pairing \eqref{barH1pairglobbis} gives a surjection $H^{1}(W_{F_{\Sigma}/F}, \widehat{U} \xrightarrow{\hat{f}} \widehat{T})\to \bar{H}^{1}(\A_{\Sigma}/\A_{F,\Sigma}, T \xrightarrow{f} U)^{D}$ with kernel the image of $(\widehat{T}^{\Gamma_{F_{\Sigma}/F}})^{\circ}$ in $H^{1}(W_{F_{\Sigma}/F}, \widehat{U} \xrightarrow{\hat{f}} \widehat{T})$. Moreover, its local analogue \eqref{barH1pairloc} gives a surjective map $H^{1}(W_{(F_{\Sigma})_{\dot{v}}/F}, \widehat{U} \xrightarrow{\hat{f}} \widehat{T})\to \bar{H}^{1}((F_{\Sigma})_{\dot{v}}^{\tn{perf}}/F_{v}, T \xrightarrow{f} U)^{D}$ with kernel 
\begin{equation*}
 \tn{Im}[(\widehat{T}^{\Gamma_{(F_{\Sigma})_{\dot{v}}/F}})^{\circ} \to H^{1}(W_{(F_{\Sigma})_{\dot{v}}/F}, \widehat{U} \xrightarrow{\hat{f}} \widehat{T})].
 \end{equation*}
\end{Proposition}

\begin{proof}
As explained in \cite[\S A.3]{KS99} (and the analogous part of \cite[\S C]{KS99}), it suffices to show that, for a torus $T$ as above, that there is a natural map $H^{1}(W_{F_{\Sigma}/F}, \widehat{T}) \to \tn{Hom}_{\tn{cts}}(\bar{H}^{0}(\A_{\Sigma}/\A_{F,\Sigma}, T), \mathbb{C}^{\times})$ which is an isomorphism and a surjective homomorphism 
\begin{equation*}
H^{0}(W_{F_{\Sigma}/F}, \widehat{T}) \to \tn{Hom}_{\tn{cts}}(\bar{H}^{1}(\A_{\Sigma}/\A_{F,\Sigma}, T), \mathbb{C}^{\times})
\end{equation*}
with kernel the image of  $(\widehat{T}^{\Gamma_{F_{\Sigma}/F}})^{\circ}$, and that both of these maps are compatible with the above pairing---compatibility with the pairing is from compatibility with the analogous pairing from \cite[\S A.4]{Dillery23b}, and the other two facts are obtained as follows.

For a finite Galois extension $F \subseteq K \subseteq F_{\Sigma}$, Langlands constructed in \cite{Lan97} maps 
\begin{equation*}
H^{i}(W_{K/F}, \widehat{T}) \to \tn{Hom}_{\tn{cts}}(\bar{H}^{j}(\A_{K}/\A_{F}, T), \mathbb{C}^{\times})
\end{equation*}
for $(i,j) = (1,0)$ and $(0,1)$ and showed that in the limit they satisfy the analogues of our claimed properties. The maps are then obtained by the taking the direct limit over all such $K$ and then applying the isomorphisms (in the co-domain) obtained by dualizing those from Lemma \ref{keytechlem}. The claimed properties then evidently hold. The local case comes from a completely analogous argument, using that $T$ and $U$ are split over $(F_{\Sigma})_{\dot{v}}$.
\end{proof}

Since the statement of Proposition \ref{appH1pairing} holds if the right-hand side is replaced with the group $\bar{H}^{1}(\A_{F_{\Sigma}}/\A, T \xrightarrow{f} U)$ by \cite[\S A.4]{Dillery23b}, we can deduce:

\begin{Corollary}\label{keycor2}
The natural map $\bar{H}^{1}(\A_{\Sigma}/\A_{F,\Sigma}, T \xrightarrow{f} U) \to \bar{H}^{1}(\A_{F_{\Sigma}}/\A, T \xrightarrow{f} U)$ is an isomorphism.
\end{Corollary}

In particular, the result \cite[Corollary A.16]{Dillery23b} holds in our context as well:

\begin{Lemma}
When $f$ is an isogeny, the group $\bar{H}^{1}(\A_{\Sigma}/\A_{F,\Sigma}, T \xrightarrow{f} U)$ is compact.
\end{Lemma}

In the setting of the previous Lemma, we also have the following analogue of \cite[Proposition A.13]{Dillery23b}:

\begin{Lemma}\label{lem:isog}
When $f$ is an isogeny, the natural map $H^{1}(\Gamma_{F_{\Sigma}/F}, \widehat{U} \xrightarrow{\hat{f}} \widehat{T}) \to H^{1}(W_{F_{\Sigma}/F}, \widehat{U} \xrightarrow{\hat{f}} \widehat{T})$ is an isomorphism.
\end{Lemma}

\begin{proof}
The argument from the proof of \cite[Proposition A.13]{Dillery23b} works here, using that the universal norm group of the $\Sigma$-id\`{e}le class group of a global function field is trivial, by \cite[Theorem 8.3.15]{NSW}.
\end{proof}

We introduce some more objects that are used in this paper: We have 
\begin{equation*}
H^{1}(W_{F_{\Sigma}/F}, \widehat{U} \xrightarrow{\hat{f}} \widehat{T}) _{\tn{red}} := H^{1}(W_{F_{\Sigma}/F}, \widehat{U} \xrightarrow{\hat{f}} \widehat{T})/([\widehat{T}^{\Gamma_{F_{\Sigma}/F}}]^{\circ}),
\end{equation*}
along with
\begin{equation*}
\tn{ker}^{1}(W_{F_{\Sigma}/F}, \widehat{U} \to \widehat{T}):= \tn{ker}[H^{1}(W_{F_{\Sigma}/F}, \widehat{U} \xrightarrow{\hat{f}} \widehat{T}) \to \prod_{v \in V_{F}} H^{1}(W_{(F_{\Sigma})_{\dot{v}}/F_{v}}, \widehat{U} \to \widehat{T})],
\end{equation*}
 and define $\tn{ker}^{1}(W_{F_{\Sigma}/F}, \widehat{U} \to \widehat{T})_{\tn{red}}$ in the obvious way, and similarly for $W_{(F_{\Sigma})_{\dot{v}}/F_{v}}$ replaced by $\Gamma_{(F_{\Sigma})_{\dot{v}}/F_{v}}$. For a multiplicative group scheme $M$ defined over $O_{F,\Sigma}$ and split over $O_{\Sigma}$, define $\tn{ker}^{i}(O_{F,\Sigma}, M)$ as the kernel of the natural map
\begin{equation*}
\check{H}^{i}(O_{\Sigma}^{\tn{perf}}/O_{F,\Sigma}, M) \to \prod_{v \in V_{F}} \check{H}^{i}((F_{\Sigma})_{\dot{v}}/F_{v}, M_{F_{v}}),
\end{equation*}
similarly with $\tn{ker}^{i}(O_{F,\Sigma}, T \to U)$, and we define $\tn{cok}^{i}(O_{F,\Sigma}, T \to U)$ as the cokernel of the map 
\begin{equation*}
H^{i}(O_{\Sigma}^{\tn{perf}}/O_{F,\Sigma}, T \to U) \to H^{i}(\A_{\Sigma}/\A_{F,\Sigma}, T \to U).
\end{equation*}

The duality results for these groups which are relevant to our applications are summarized here:

\begin{Proposition}\label{prop:cokdual}
The pairings \eqref{barH1pairglob}, \eqref{barH1pairglobbis}, and \eqref{barH1pairloc} induce the following duality isomorphisms:
\begin{enumerate}
\item $\tn{cok}^{1}(O_{F,\Sigma}, T \to U) \xrightarrow{\sim} [H^{1}(W_{F_{\Sigma}/F}, \widehat{U} \to \widehat{T})_{\tn{red}}/\tn{ker}^{1}(W_{F_{\Sigma}/F}, \widehat{U} \to \widehat{T})_{\tn{red}}]^{D}$;
\item $\tn{ker}^{1}(O_{F,\Sigma}, T \to U) \xrightarrow{\sim} \tn{ker}^{1}(W_{F_{\Sigma}/F}, \widehat{U} \to \widehat{T})^{D}$.
\end{enumerate}
\end{Proposition}

\begin{proof}
Both of these isomorphisms follow from combining Proposition \ref{appH1pairing} (the local and global cases) with Proposition \ref{prop:resprod} with the exact sequence 
\begin{equation*}
0 \to \tn{cok}^{i}(O_{F,\Sigma}, T \to U) \to \bar{H}^{i}(\A_{\Sigma}/\A_{F,\Sigma}, T \to U) \to \tn{ker}^{i+1}(O_{F,\Sigma}, T \to U) \to 0
\end{equation*}
(cf. the proof of \cite[Proposition A.19]{Dillery23b} for the first isomorphism and the proof of \cite[Lemma C.3.D]{KS99} for the second).
\end{proof}

We observe that when $T \xrightarrow{f} U$ is any isogeny, the identification from Lemma \ref{lem:isog} induces identifications $H^{1}(\Gamma_{F_{\Sigma}/F}, \widehat{U} \xrightarrow{\hat{f}} \widehat{T})_{\tn{red}} \xrightarrow{\sim} H^{1}(W_{F_{\Sigma}/F}, \widehat{U} \xrightarrow{\hat{f}} \widehat{T})_{\tn{red}}$ and $\tn{ker}^{1}(\Gamma_{F_{\Sigma}/F}, \widehat{U} \xrightarrow{\hat{f}} \widehat{T})_{\tn{red}} \xrightarrow{\sim} \tn{ker}^{1}(W_{F_{\Sigma}/F}, \widehat{U} \xrightarrow{\hat{f}} \widehat{T})_{\tn{red}}$. The last result needed is:

\begin{Lemma}\label{lem:Pisog}
When $f$ is an isogeny with kernel $P$, we have canonical identifications 
\begin{equation*}
H^{1}(\Gamma_{F_{\Sigma}/F}, \widehat{U} \xrightarrow{\hat{f}} \widehat{T})_{\tn{red}} \xrightarrow{\sim} H^{1}(\Gamma_{F_{\Sigma}/F}, X^{*}(P)), 
\end{equation*}
\begin{equation*}
\tn{ker}^{1}(\Gamma_{F_{\Sigma}/F}, \widehat{U} \xrightarrow{\hat{f}} \widehat{T})_{\tn{red}} \xrightarrow{\sim} \tn{ker}^{1}(\Gamma_{F_{\Sigma}/F}, X^{*}(P)).
\end{equation*}
\end{Lemma}

\begin{proof}
This is a consequence of the identifications 
\begin{equation*}
H^{1}(\Gamma_{F_{\Sigma}/F}, \widehat{U} \xrightarrow{\hat{f}} \widehat{T})_{\tn{red}} \xrightarrow{\sim} H^{2}(\Gamma_{F_{\Sigma}/F}, X^{*}(U) \to X^{*}(T))
\end{equation*}
and  $\tn{ker}^{1}(\Gamma_{F_{\Sigma}/F}, \widehat{U} \xrightarrow{\hat{f}} \widehat{T})_{\tn{red}} \xrightarrow{\sim} \tn{ker}^{2}(\Gamma_{F_{\Sigma}/F}, X^{*}(U) \to X^{*}(T))$, which follow from the proof of \cite[Lemma C.3.C]{KS99}. The former group is canonically identified with $H^{1}(\Gamma_{F_{\Sigma}/F}, X^{*}(P))$ and the latter with $\tn{ker}^{1}(\Gamma_{F_{\Sigma}/F}, X^{*}(P))$.
\end{proof}

\subsection{A Kaletha gerbe with restricted ramification}\hfill\\
The goal of this subsection is to define a variant of the gerbe $\tn{Kal}_{F}$ which is defined over $\tn{Spec}(O_{F,\Sigma})$ for a fixed finite set of places $\Sigma \subseteq V_{F}$---for technical reasons, we initially assume $|\Sigma| \geq 2$. Fix a co-final system $(E_{i},n_{i})$ of natural numbers $n_{i}$ and finite Galois extensions $E_{i}/F$ which are unramified outside of $\Sigma$, and a set of lifts $\dot{\Sigma}$ of $\Sigma$ to $V_{F_{\Sigma}}$, where we recall that $F_{\Sigma} = \bigcup E_{i}$ is the maximal Galois extension of $F$ unramified outside of $\Sigma$.

Set $P_{\dot{\Sigma},i}$ to be the $O_{F,\Sigma}$-group scheme Cartier dual to  $(\frac{1}{n_{i}}\mathbb{Z}/\mathbb{Z})[\Gamma_{E_{i}/F} \times \Sigma_{E_{i}}]_{0,0}$ consisting of all elements whose $[(\sigma,w)]$-coefficient is zero unless $\sigma^{-1}(w) \in \dot{\Sigma}^{(E_{i})}$. The first step is to define a ``level $i$'' canonical class $\xi_{i} \in H_{\tn{fppf}}^{2}(O_{F,\Sigma}, P_{\dot{\Sigma},i})$---as such, fix $i \in \mathbb{N}$.

By \cite[Lemma 1.12]{DLH26} we can find a cover $\{U_{i,1}, U_{i,2}\}$ of $\tn{Spec}(O_{E_{i}, \Sigma_{E_{i}}})$ with $\tn{Pic}(U_{i,j})=\{1\}$, where $U_{i,j} = \tn{Spec}(O_{E_{i}, \Sigma_{i,j,E_{i}}})$ for a finite sets of places $\Sigma_{i,j,E_{i}} \subset V_{E_{i}}$---note that $\tn{Spec}(O_{F,\Sigma}) = \bar{U}_{i,1} \cup \bar{U}_{i,2}$, where $\bar{U}_{i,j}$ denotes the image of $U_{i,j}$ in $C$. Denote the image of $\Sigma_{i,j,E_{i}}$ in $V_{F}$ by $\Sigma_{i,j}$. For a fixed $j \in \{1,2\}$, consider the map (8) from \cite{Dillery23b} (noting that $\Sigma_{i,j}$ satisfies \cite[Conditions 3.1]{Dillery23b} with respect to $E_{i}/F$)
\begin{equation}\label{JIMJ8}
\Hom_{O_{F,\Sigma_{i,j}}}(P_{E_{i}, \dot{(\Sigma_{i,j})_{E_{i}}, n_{i}}}, P_{\dot{\Sigma},i}) \to \check{H}^{2}(O_{\Sigma_{i,j}}^{\tn{perf}}/O_{F,\Sigma_{i,j}}, P_{\dot{\Sigma},i})
\end{equation}
and take the image of the homomorphism $P_{E_{i}, \dot{(\Sigma_{i,j})_{E_{i}}, n_{i}}} \to P_{\dot{\Sigma},i}$ induced by the canonical inclusion 
\begin{equation*}
(\frac{1}{n_{i}}\mathbb{Z}/\mathbb{Z})[\Gamma_{E_{i}/F} \times \Sigma_{E_{i}}]_{0,0} \to (\frac{1}{n_{i}}\mathbb{Z}/\mathbb{Z})[\Gamma_{E_{i}/F} \times (\Sigma_{i,j})_{E_{i}}]_{0,0};
\end{equation*}
denote the resulting class by $\xi_{i,j}$. Note that this does not yet determine a gerbe over $\tn{Spec}(O_{F,\Sigma})$, but rather over the open subscheme $\bar{U}_{i,j} = \tn{Spec}(O_{F,\Sigma_{i,j}})$. However, we have: 

\begin{Lemma}\label{gerbeoverlap}
The images of $\xi_{i,j}$ in $\check{H}^{2}(O_{\Sigma_{i,1} \cup \Sigma_{i,2}}^{\tn{perf}}/O_{F,\Sigma_{i,1} \cup \Sigma_{i,2}}, P_{\dot{\Sigma},i})$ coincide for $j=1,2$.
\end{Lemma}

\begin{proof}
This is a consequence of the functoriality (in the first argument) of the map \eqref{JIMJ8} (cf. \cite[\S 3.2]{Dillery23b}).
\end{proof}

\begin{Corollary}
There is a $P_{\dot{\Sigma},i}$-gerbe over $\tn{Spec}(O_{F,\Sigma})$ whose associated class in $\check{H}^{2}(O_{\Sigma}^{\tn{perf}}/O_{F,\Sigma}, P_{\dot{\Sigma},i})$, denoted by $\xi_{\Sigma,i}$, is canonical and restricts to $\xi_{i,j}$ on $\mathcal{O}(\bar{U}_{i,j})$ for $j=1,2$.
\end{Corollary}

\begin{proof}
By Mayer-Vietoris (as in \cite[0A50]{stacks-project}, whose proof works for the flat topology) one can glue the two gerbes representing $\xi_{i,j}$ along an isomorphism obtained from Lemma \ref{gerbeoverlap} to obtain the desired gerbe.
\end{proof}

Note that there is a transition map $P_{\dot{\Sigma},\ell} \to P_{\dot{\Sigma},k}$ for $\ell \geq k$, as defined in \cite[pp. 2205-2206]{Dillery23b}. The following result also comes from the functoriality of the map \eqref{JIMJ8}:
\begin{Lemma}
For $\ell \geq k$, the image of $\xi_{\Sigma,\ell}$ in $\check{H}^{2}(O_{\Sigma}^{\tn{perf}}/O_{F,\Sigma}, P_{\dot{\Sigma},k})$ is $\xi_{\Sigma,k}$.
\end{Lemma}

We thus obtain an element of $\varprojlim \check{H}^{2}(O_{\Sigma}^{\tn{perf}}/O_{F,\Sigma}, P_{\dot{\Sigma},k})$; the goal is to construct a canonical lift of this element in $\check{H}^{2}(O_{\Sigma}^{\tn{perf}}/O_{F,\Sigma}, P_{\dot{\Sigma}})$, which we turn to now, following the approach of \cite[\S 3.4]{Dillery23b}. For a place $v \in V_{F}$, define $\A_{v}^{\Sigma} := F_{\Sigma}^{\tn{perf}} \otimes_{F} F_{v}$. We have the following analogue of \cite[Lemma 3.19]{Dillery23b}, whose proof is the same:

\begin{Lemma}\label{adelicprojlim}
\begin{enumerate}
\item The natural map $\check{H}^{j}(\A_{v}^{\Sigma}/F_{v}, P_{\dot{\Sigma}}) \to \varprojlim \check{H}^{j}(\A_{v}^{\Sigma}/F_{v} P_{\dot{\Sigma},k})$ is an isomorphism for all $j$.
\item There is a canonical isomorphism $\check{H}^{1}(\A_{v}^{\Sigma}/F_{v}, P_{\dot{\Sigma}}) \xrightarrow{\sim} \check{H}^{1}((F_{\Sigma})_{\dot{v}}/F_{v}, P_{\dot{\Sigma}})$.
\end{enumerate}
\end{Lemma}

\begin{Definition}
For $\Sigma$ a finite set of places of $F$ and $v \in \Sigma$, define $u_{v}^{\Sigma} := \varprojlim \tn{Res}_{E/F}(\mu)/\mu$, where the limit is over all finite Galois $F_{v} \subseteq E \subseteq (F_{\Sigma})_{\dot{v}}$; we have a canonical class $\xi_{v}^{\Sigma}$ in $H_{\tn{fppf}}^{2}(F_{v}, u^{\Sigma}_{v})$ given by taking the image of the canonical class in $H_{\tn{fppf}}^{2}(F_{v}, u_{v})$ defining $\tn{Kal}_{v}$ under the map $u_{v} \to u^{\Sigma}_{v}$ induced by a projective limit of field norms; the class in $H_{\tn{fppf}}^{2}(F_{v}, u_{v}^{\Sigma})$ does not depend on the choice of $\dot{v}$. Denote the corresponding gerbe over $F_{v}$ for a choice of cocycle representative by $\tn{Kal}_{v}^{(\Sigma)}$. 
\end{Definition}

Observe that the composition $u_{v} \to P_{F_{v}} \to (P_{\dot{\Sigma}})_{F_{v}}$ (the first map is the ``localization map'' from \cite[p. 27]{Dillery23b}) factors through the projection $u_{v} \to u_{v}^{\Sigma}$. Denote by $\xi_{v}^{(\Sigma)} \in \check{H}^{1}(\A_{v}^{\Sigma}/F_{v}, P_{\dot{\Sigma}})$ the preimage under the isomorphism (2) from Lemma \ref{adelicprojlim} of the class $\xi^{\Sigma,\tn{mid}}_{v} \in H_{\tn{fppf}}^{2}(F_{v}, P_{\dot{\Sigma}})$ obtained by taking the image of $\xi_{v}^{\Sigma} \in H_{\tn{fppf}}^{2}(F_{v}, u_{v})$ via the induced morphism $u^{\Sigma}_{v} \to (P_{\dot{\Sigma}})_{F_{v}}$.

The decomposition 
\begin{equation*}
\check{H}^{2}(\mathbb{A}_{\Sigma}/\mathbb{A}_{F,\Sigma}, P_{\dot{\Sigma}}) = \prod_{v \in \Sigma} \check{H}^{2}(\mathbb{A}_{v}^{\Sigma}/F_{v}, P_{\dot{\Sigma}})
\end{equation*}
shows (cf. \cite[pp. 2213-2214]{Dillery23b}) that there is a canonical class $x_{\Sigma} \in \check{H}^{2}(\mathbb{A}_{\Sigma}/\mathbb{A}_{F,\Sigma}, P_{\dot{\Sigma}})$ which localizes to each $\xi_{v}^{(\Sigma)}$ via this product decomposition. Our desired canonical class in $\check{H}^{2}(O_{\Sigma}^{\tn{perf}}/O_{F,\Sigma}, P_{\dot{\Sigma}})$ will be obtained by finding a simultaneous lift of $\varprojlim \xi_{\Sigma,k}$ and $x_{\Sigma}$.

First, some setup is needed: The projective system $\{P_{\dot{\Sigma},i}\}$ fits into a short exact sequence of projective systems of tori $\{T_{i}\}$, $\{U_{i}\}$
\begin{equation*}
1 \to \{P_{\dot{\Sigma},i}\} \to \{T_{i}\} \to \{U_{i}\} \to 1
\end{equation*}
such that each $T_{i+1} \to T_{i}$ is surjective with kernel a torus (this also holds for each $U_{i+1} \to U_{i}$); set $T = \varprojlim T_{i}$ and $U = \varprojlim U_{i}$. For the details of this construction, see \cite[Lemma 3.5.1]{Kaletha18}.

\begin{Theorem}\label{canonclassthm}
There is a unique class in $\check{H}^{2}(O_{\Sigma}^{\tn{perf}}/O_{F,\Sigma}, P_{\dot{\Sigma}})$ lifting $\varprojlim \xi_{\Sigma,k}$ whose image in $H^{2}(\A_{\Sigma}/\A_{F,\Sigma}, T \to U)$ coincides with the image of $x_{\Sigma}$ in the group $H^{2}(\A_{\Sigma}/\A_{F,\Sigma}, T \to U)$. 
\end{Theorem}

\begin{proof}
This argument will follow the proof of \cite[Proposition 3.21]{Dillery23b}, which itself follows the proof of the analogous result in \cite[Proposition 3.5.2]{Kaletha18}---the purpose of this proof will be to summarize the arguments loc. cit. and explain why they still apply in our situation.

The map $\check{H}^{2}(O_{\Sigma}^{\tn{perf}}/O_{F,\Sigma}, P_{\dot{\Sigma}}) \to \varprojlim \check{H}^{2}(O_{\Sigma}^{\tn{perf}}/O_{F,\Sigma}, P_{\dot{\Sigma},i})$ is always surjective, so we can take $\tilde{\xi}$ a lift of $\varprojlim \xi_{\Sigma,k}$. By construction, the images of $x_{\Sigma}$ and $\tilde{\xi}$ in $\varprojlim H^{2}(\A_{\Sigma}/\A_{F,\Sigma}, T_{i} \to U_{i})$ coincide, and there is an identification
\begin{equation}\label{mainpfeq}
H^{2}(\A_{\Sigma}/\A_{F,\Sigma}, T_{i} \to U_{i}) = \prod_{v \in \Sigma} H_{\tn{fppf}}^{2}(F_{v}, (P_{\dot{\Sigma},i})_{F_{v}})
\end{equation}
by \cite[Theorem C.1.B]{KS99} (the same arguments work in our setting) and the surjectivity of $T_{i}((F_{\Sigma})_{\dot{v}}^{\tn{perf}}) \to U_{i}((F_{\Sigma})_{\dot{v}}^{\tn{perf}})$. This latter surjectivity is deduced from the vanishing of $H_{\tn{fppf}}^{1}((F_{\Sigma})_{\dot{v}}^{\tn{perf}}, P_{\dot{\Sigma},i})$, which reduces to the case of $\mu_{n}$, where it follows from the divisibility of the group $(F_{\Sigma})_{\dot{v}}^{\tn{perf},\times}$ as in Lemma \ref{divlemma}.

With the equality \eqref{mainpfeq} established, the identical arguments in \cite{Kaletha18} show that the images of $\tilde{\xi}$ and $x_{\Sigma}$ coincide in $\varprojlim H^{2}(\A_{\Sigma}/\A_{F,\Sigma}, T_{i} \to U_{i})$. To upgrade this equality, we need to prove that the natural map
\begin{equation}\label{keyisom1}
\varprojlim{^{(1)}}\check{H}^{1}(O_{\Sigma}^{\tn{perf}}/O_{F,\Sigma}, P_{\dot{\Sigma},i}) \to \varprojlim{^{(1)}}H^{1}(\A_{\Sigma}/\A_{F,\Sigma}, T_{i} \to U_{i})
\end{equation}
is an isomorphism.

Replacing the results from \cite[Appendix C]{KS99} in \cite{Kaletha18} with their analogues in our \S \ref{sec:unram} shows that the above isomorphism holds if we can prove the vanishing of 
\begin{equation*}
\varprojlim \tn{cok}^{1}(O_{F,\Sigma}, T_{i} \to U_{i}),
\end{equation*}
cf. \S \ref{sec:unram} for the definition of $\tn{cok}^{1}$. Then by Lemma \ref{lem:isog} and Proposition \ref{prop:cokdual} we have that $\tn{cok}^{1}(O_{F,\Sigma}, T_{i} \to U_{i})$ is Pontryagin dual to 
\begin{equation*}
H^{1}(\Gamma_{F_{\Sigma}/F}, \widehat{U} \to \widehat{T})_{\tn{red}}/\tn{ker}^{1}(\Gamma_{F_{\Sigma}/F}, \widehat{U} \to \widehat{T})_{\tn{red}},
\end{equation*}
and by Proposition \ref{prop:cokdual} and Lemma \ref{lem:Pisog} there are canonical identifications
\begin{equation*}
H^{1}(\Gamma_{F_{\Sigma}/F}, \widehat{U} \to \widehat{T})_{\tn{red}} = H^{1}(\Gamma_{F_{\Sigma}/F}, X^{*}(P_{\dot{\Sigma},i})), 
\end{equation*}
\begin{equation*}
\tn{ker}^{1}(\Gamma_{F_{\Sigma}/F}, \widehat{U} \to \widehat{T})_{\tn{red}} = \tn{ker}^{1}(\Gamma_{F_{\Sigma}/F}, X^{*}(P_{\dot{\Sigma},i}));
\end{equation*}
again, see \S \ref{sec:unram} for the notation ``$\tn{ker}^{1}$''.

It thus suffices to show that $\varinjlim H^{1}(\Gamma_{F_{\Sigma}/F}, X^{*}(P_{\dot{\Sigma},i})) = 0$. First, fit $X^{*}(P_{\dot{\Sigma},i})$ into the exact sequence (see the proof of \cite[Corollary 3.12]{Dillery23b})
\begin{equation}\label{JIMJeq1}
0 \to X^{*}(P_{\dot{\Sigma},i})\to \mathbb{Q}/\mathbb{Z}[\Gamma_{E_{i}/F} \times \Sigma]_{0} \to \mathbb{Q}/\mathbb{Z}[\Sigma_{E_{i}}]_{0} \to 0.
\end{equation}
For a fixed $i$, we claim that map $\mathbb{Q}/\mathbb{Z}[\Sigma_{E_{i}}]_{0} \to \varinjlim_{j \geq i} \mathbb{Q}/\mathbb{Z}[\Sigma_{E_{j}}]_{0}$ is zero; to see this, use equation (10) and the sentences proceeding it in \cite{Dillery23b}, noting that at each $v \in \Sigma$, the extension $(F_{\Sigma})_{\dot{v}}/F_{v}$ can (because one can always enlarge the constant field) be written as a union of finite extensions whose degrees over $F_{v}$ are a cofinal system in $\mathbb{N}^{\times}$.

The long exact sequence associated to \eqref{JIMJeq1} and the above observation thus reduces us to showing that $\varinjlim H^{1}(\Gamma_{F_{\Sigma}/F}, \mathbb{Q}/\mathbb{Z}[\Gamma_{E_{i}/F} \times \Sigma]_{0}) = 0$, which follows from an identical argument as in \cite[Lemma 3.4.4]{Kaletha18}---this establishes the isomorphism \eqref{keyisom1}. From here, the argument is totally unchanged from the one loc. cit., but we summarize it anyway for expository completeness.

The isomorphism \eqref{keyisom1} means that we may modify $\tilde{\xi}$ by an element of 
\begin{equation*}
\varprojlim {^{(1)}} \check{H}^{1}(O_{\Sigma}^{\tn{perf}}/O_{F,\Sigma}, P_{\dot{\Sigma},i})
\end{equation*}
so that the images of $\tilde{\xi}$ and $x_{\Sigma}$ in $H^{2}(\A_{\Sigma}/\A_{F,\Sigma}, T \to U)$ are equal, proving existence. Uniqueness is a consequence of the injectivity of the composition 
\begin{equation*}
 \varprojlim {^{(1)}}\check{H}^{1}(O_{\Sigma}^{\tn{perf}}/O_{F,\Sigma}, P_{\dot{\Sigma},i}) \to \varprojlim {^{(1)}}H^{1}(\A_{\Sigma}/\A_{F,\Sigma}, T_{i} \to U_{i}) \to H^{2}(\A_{\Sigma}/\A_{F,\Sigma}, T \to U).
\end{equation*}
\end{proof}

Denote the resulting class by $[\dot{\xi}_{\Sigma}] \in \check{H}^{2}(O_{\Sigma}^\tn{perf}/O_{F,\Sigma}, P_{\dot{\Sigma}})$. The following is immediate from our above work:

\begin{Lemma}
For $(\Sigma, \dot{\Sigma}) \subseteq (\Sigma', \dot{\Sigma}')$ as above, the image of $[\dot{\xi}_{\Sigma'}]$ in $\check{H}^{2}(O_{\Sigma}^\tn{perf}/O_{F,\Sigma}, P_{\dot{\Sigma}})$ is $[\dot{\xi}_{\Sigma}]$. Moreover, the image of the canonical class $\xi \in H_{\tn{fppf}}^{2}(F, P)$ in the limit $\varprojlim \check{H}^{2}(O_{\Sigma}^\tn{perf}/O_{F,\Sigma}, P_{\dot{\Sigma}})$ is $\varprojlim_{\Sigma} [\dot{\xi}_{\Sigma}]$.
\end{Lemma}

The following result says that $[\dot{\xi}]$ is characterized by the above property:

\begin{Proposition}\label{prop:H2invlim}
The map
\begin{equation*}
H_{\tn{fppf}}^{2}(F, P) \to \varprojlim \check{H}^{2}(O_{\Sigma}^{\tn{perf}}/O_{F,\Sigma}, P_{\dot{\Sigma}}) 
\end{equation*}
is an isomorphism; in particular, $\varprojlim^{1} \check{H}^{1}(O_{\Sigma}^{\tn{perf}}/O_{F,\Sigma}, P_{\dot{\Sigma}}) = 0$.
\end{Proposition}

\begin{proof}
This is a consequence of the injectivity part of the proof of Theorem \ref{canonclassthm}.
\end{proof}

In order to define the gerbe $\tn{Kal}_{F}$ we have picked a representative $\dot{\xi}$ for the canonical class in $H_{\tn{fppf}}^{2}(F, P)$, which determines representatives for each $[\dot{\xi}_{\Sigma}]$---denote the corresponding gerbes by $\tn{Kal}_{F,\Sigma}$. The above Proposition has the following useful consequence:
\begin{Corollary}\label{cor:transitions}
There exists a compatible system of transition maps
\begin{equation*}
\tn{Kal}_{F,\Sigma'} \to \tn{Kal}_{F,\Sigma}
\end{equation*}
across all pairs $(\Sigma, \dot{\Sigma}) \subseteq (\Sigma', \dot{\Sigma}')$. Moreover, these maps induce canonical pullback maps $H_{\tn{\'{e}t}}^{1}(\tn{Kal}_{F,\Sigma}, G) \to H_{\tn{\'{e}t}}^{1}(\tn{Kal}_{F,\Sigma'}, G)$.
\end{Corollary}

Recall that, for technical reasons (so that Lemma \ref{divlemma} holds), we only defined the gerbes $\tn{Kal}_{F,\Sigma}$ for $|\Sigma| \geq 2$; in fact, we can glue these gerbes to define $\tn{Kal}_{F,\{v\}}$ for a place $v$. Namely, set $\Sigma = \{v\} \cup \{w\}$ and $\Sigma' = \{v\} \cup \{w'\}$ for $w \neq w'$ both different from $v$. Observe that the proof of Theorem \ref{canonclassthm} works if $P_{\dot{\Sigma}}$ is replaced with $P_{\dot{\{v\}}}$ (keeping $O_{F,\Sigma}$ or $O_{F,\Sigma'}$ as the base-rings), replacing $u_{v}^{\Sigma}$ with $u_{v}^{\{\dot{v}\}}$ ---in particular, we obtain classes $[\dot{\xi}_{\{v\}}^{(\Sigma)}] \in \check{H}^{2}(O_{\Sigma}^{\tn{perf}}/O_{F,\Sigma}, P_{\dot{\{v\}}})$ and $[\dot{\xi}_{\{v\}}^{(\Sigma')}] \in \check{H}^{2}(O_{\Sigma'}^{\tn{perf}}/O_{F,\Sigma'}, P_{\dot{\{v\}}})$. The key result is then:

\begin{Lemma} 
The image of $[\dot{\xi}_{\{v\}}^{(\Sigma)}]$ and $[\dot{\xi}_{\{v\}}^{(\Sigma')}]$ in $\check{H}^{2}(O_{\{v,w,w'\}}^{\tn{perf}}/O_{F,\{v,w,w'\}}, P_{\dot{\{v\}}})$ coincide. In particular, by Mayer-Vietoris there is a canonical class 
\begin{equation*}
[\dot{\xi}_{\{v\}}] \in \check{H}^{2}(O_{\{v\}}^{\tn{perf}}/O_{F,\{v\}}, P_{\dot{\{v\}}}).
\end{equation*}
\end{Lemma}

\begin{proof}
It suffices to show that the images of both classes satisfy the hypotheses of Theorem \ref{canonclassthm} for the base ring $O_{F,\{v,w,w'\}}$. The fact that both classes lift $\varprojlim \xi_{\{v\},k} \in H_{\tn{fppf}}^{2}(O_{F,\{v,w,w'\}}, P_{\{\dot{v}\},k})$ is a consequence of the functoriality of the map \eqref{JIMJ8}, and the other condition is clear.
\end{proof}

We can thus define the gerbe $\tn{Kal}_{F,\{v\}}$ as well, and for any $v \in \Sigma$ we have compatible maps $\tn{Kal}_{F,\Sigma} \to \tn{Kal}_{F,\{v\}}$ inducing canonical pullback maps on $H^{1}$ (as above).

\begin{Remark}
The natural definition of $\tn{Kal}_{F,\emptyset}$ is just $C$ itself, since $P_{\emptyset}$ is trivial.
\end{Remark}

One can also define an analogue of $\tn{Kal}_{F,\Sigma}$ for the Kottwitz gerbe as follows. For a fixed finite set of places $\Sigma$, we can repeat the above construction with $\mathbb{D}_{\Sigma}$, the group scheme Cartier dual to $\varinjlim_{F \subseteq E \subseteq F_{\Sigma}} \Z[\Sigma_{E}]$, in place of $P_{\dot{\Sigma}}$ to define, for a fixed $E$, canonical classes $\xi_{E,1}$ and $\xi_{E,2}$ in $H^{2}(U_{E,i}, \mathbb{D}_{\Sigma})$, where $U = U_{E,1} \cup U_{E,2}$ with $\tn{Pic}(U_{E,i}) = \{1\}$ (again using \cite[Lemma 1.12]{DLH26}), which agree on the overlap, and thus give a canonical class in $H^{2}(O_{F,\Sigma}, \mathbb{D}_{\Sigma})$ whose projective limit over all $E \subseteq F_{\Sigma}$ defines the desired canonical class in $H^{2}(O_{F,\Sigma}, \mathbb{D}_{\Sigma})$. We leave the justification of these facts to the reader, since it is nearly identical to the same arguments for $\tn{Kal}_{F,\Sigma}$, but in this case does not require a localization argument as in Theorem \ref{canonclassthm}. We thus obtain:

\begin{Proposition}\label{SigmaKott}
There is a canonical class $[\tn{Kott}_{F,\Sigma}] \in H^{2}(O_{F,\Sigma}, \mathbb{D}_{\Sigma})$ such that 
\begin{equation*}
\varprojlim_{\Sigma} [\tn{Kott}_{F,\Sigma}] = [\tn{Kott}_{F}].
\end{equation*}
\end{Proposition}

We denote the gerbe from the preceding Proposition by $\tn{Kott}_{F,\Sigma}$.

\subsection{A new gerbe}\hfill\\
Here is the key definition of this paper, which is a global analogue of \cite[D\'{e}finition 4.6]{Fargues}:

\begin{Definition}
Define the gerbe $\tn{Kott}_{F}^{1/\infty} := \tn{Kott}_{F} \times_{F} \tn{Kal}_{F}$.
\end{Definition}

We now prove some new results about this gerbe, following \cite{Fargues} in the local case; let $Z$ be a finite multiplicative group scheme over $?$, where $? = F$ or $F_{v}$ for a fixed place $v$ of $F$.

\begin{Lemma}
\begin{enumerate}
\item{For every class $c \in H_{\tn{fppf}}^{2}(F, Z)$, there exists $\lambda \colon P \to Z$ (defined over $F$) such that $\lambda_{*}\xi = c$.}
\item{For every class $c_{v} \in H_{\tn{fppf}}^{2}(F_{v}, Z)$, there exists $\lambda \colon P_{F_{v}} \to Z$ such that $\lambda_{*}\xi = c_{v}$.}
\end{enumerate}
\end{Lemma}

\begin{proof}
The first statement comes from the remark proceeding \cite[(9)]{Dillery23b}; we leave the second as an exercise.
\end{proof}

For $\mathcal{E}$ a gerbe over $\mathrm{Spec}(F)_{\tn{fppf}}$, we denote by $H^{2}(\mathcal{E}, Z)$ the isomorphism classes of $Z$-gerbes over $\mathcal{E}$, where $\mathcal{E}$ is given the fpqc topology inherited from $\tn{Spec}(F)$. Recall that a gerbe over a site $\mathcal{C}$ is a category $\mathcal{E} \to \mathcal{C}$ fibered in groupoids such that for all objects $x \in \tn{obj}(\mathcal{C})$ there is a covering $\{U_{i}\}_{i \in I}$ of $x$ such that $\mathcal{E}(U_{i}) \neq \emptyset$ for all $i$ and any two objects in $\mathcal{E}$ lying over the same object in $\mathcal{C}$ are locally isomorphic. Picking a lift $\tilde{U}$ of $U \in \text{obj}(\mathcal{C})$ defines a functor on $\mathcal{C}/U$ via $(V \xrightarrow{f} U) \mapsto \mathrm{Isom}(f^{*}\tilde{U}, f^{*}\tilde{U})$---this glues to define a sheaf of groups on $\mathcal{C}$, and if it's isomorphic to $Z$ then we say that $\mathcal{E}$ is a $Z$-gerbe over $\mathcal{C}$. 

The rest of this subsection is essentially a re-phrasing of \cite[\S \S 8.1-8.3]{Fargues} in the global setting. We first collect two results from \cite{Fargues} (Lemme 8.3 and Lemma 8.4 loc. cit.) which hold for us (with identical proof):
\begin{Lemma}
\begin{enumerate}
\item{The map $H_{\tn{fppf}}^{2}(F, Z) \to H^{2}(\tn{Kal}_{F}, Z)$ is trivial.}
\item{If $Z \to Z'$ is a surjection of finite multiplicative $F$-groups, then the map $H_{\tn{fpqc}}^{1}(\tn{Kal}_{F}, Z) \to H_{\tn{fpqc}}^{1}(\tn{Kal}_{F}, Z')$ is surjective.} 
\end{enumerate}
\end{Lemma}

The main statement we want to show is:

\begin{Proposition}\label{vanprop1}
If $Z$ is a finite multiplicative group scheme over $F$, we have $H^{2}(\tn{Kal}_{F}, Z) = \{*\}$.
\end{Proposition} 

\begin{proof}
The same argument as in \cite[Th\'{e}or\`{e}me 8.2]{Fargues} reduces to showing that
\begin{equation*}
H^{2}([\mathrm{Spec}(\ov{F})/P_{\ov{F}}], Z_{\ov{F}}) = \{*\}.
\end{equation*}
Breaking up $Z_{\ov{F}} = \prod_{(n_{i},p)=1} \mu_{n_{i}} \times \prod_{m_{i}} \mu_{p^{m_{i}}}$ reduces to doing this computation for the prime-to-$p$ and $p$-power cases separately. The prime-to-$p$ case is the same as the analogous argument in the proof of \cite[Th\'{e}or\`{e}m 8.2]{Fargues} using \'{e}tale cohomology, so it suffices to handle the $p$-power case. For this case, we observe by \cite{Gir} that $H^{2}([\mathrm{Spec}(\ov{F})/P_{\ov{F}}], Z_{\ov{F}})$ classifies isomorphism classes of extensions of $P_{\ov{F}}$ by $Z_{\ov{F}}$. By passing to character modules and taking direct sums, this reduces to showing that the group $\tn{Ext}^{1}_{\mathbb{Z}}(\frac{1}{p}\Z/\Z, (\Q/\Z)[p^{\infty}]) \xrightarrow{\sim} ((\Q/\Z)[p^{\infty}])/p( (\Q/\Z)[p^{\infty}])$ vanishes, which is true.
\end{proof}

In a similar vein, we also have:

\begin{Proposition}\label{prop:H2torus}
For $T$ an $F$-torus, we have $H^{2}(\tn{Kott}_{F}, T) = \{*\}$
\end{Proposition}

\begin{proof}
The same argument as in proof of this result's analogue in the local case (\cite[Th\'{e}or\`{e}me 8.6]{Fargues}) works here.
\end{proof}

The key consequence of the above two cohomological vanishing results is:

\begin{Proposition}\label{prop:multvanish}
If $D$ is any multiplicative group scheme over $F$, then 
\begin{equation*}
H^{2}(\tn{Kott}^{1/\infty}_{F}, D) = \{*\}.
\end{equation*}
\end{Proposition}

\begin{proof}
As was the case for the previous two vanishing results, the argument is identical to the proof of its local analogue, which is \cite[Th\'{e}or\`{e}me 8.9]{Fargues}, after using the proof of Proposition \ref{vanprop1} in place of the continuous group cohomology arguments loc. cit.
\end{proof}

One has an analogue of all of these results for the gerbes $\tn{Kott}_{F}$ and $\tn{Kal}_{F}$ replaced by $\tn{Kott}_{F,\Sigma}$ and $\tn{Kal}_{F,\Sigma}$:
\begin{Proposition}\label{Sigmavan}
\begin{enumerate}
\item{For $Z$ a finite multiplicative group scheme over $O_{F,\Sigma}$ and split over $O_{\Sigma}$, we have $H^{2}(\tn{Kal}_{F,\Sigma},Z) = \{\ast\}$.}
\item{For $T$ a torus defined over $O_{F,\Sigma}$ and split over $O_{\Sigma}$, we have $H^{2}(\tn{Kott}_{F,\Sigma},T) = \{\ast\}$.}
\item{For $D$ any multiplicative group scheme defined over $O_{F,\Sigma}$ and split over $O_{\Sigma}$, we have $H^{2}(\tn{Kott}_{F,\Sigma} \times_{O_{F,\Sigma}} \tn{Kal}_{F,\Sigma},D) = \{\ast\}$.}
\end{enumerate}
\end{Proposition}

\begin{proof}
The same arguments for the above $F$-analogues work, replacing $\ov{F}$ with $O_{\Sigma}^{\tn{perf}}$ and using \cite[Lemma 2.11]{Dillery23b} for the required vanishing of Brauer groups used in the arguments of \cite{Fargues} quoted in the proofs of Proposition \ref{vanprop1} and Proposition \ref{prop:H2torus}.
\end{proof}

\section{The cohomology of $\tn{Kott}_{F}^{1/\infty}$}\label{sec:H1}
We now turn our focus to the study of $H_{\tn{\'{e}t}}^{1}(\tn{Kott}_{F}^{1/\infty}, G)$ for a connected reductive group $G$.

\subsection{A Tannakian description}\hfill\\
We give a Tannakian category corresponding to finite-level gerbes which approximate the gerbe $\tn{Kott}^{1/\infty}_{F}$ in a sense that will be explained shortly; this is inspired by the ideas of \cite[\S 6]{Fargues}---analogous constructions can be done for the gerbes $\tn{Kott}_{F,\Sigma} \times_{O_{F.\Sigma}} \tn{Kal}_{F,\Sigma}$ as well but we omit them here.

Recall that a \textit{Drinfeld isoshtuka} is a pair $(V, \varphi)$ consisting of a finite-dimensional $\breve{F} := F \otimes_{\mathbb{F}_{q}} \ov{\mathbb{F}_{q}}$-vector space $V$ along with a Frobenius-semilinear endomorphism $\varphi$. The category of Drinfeld isoshtukas is a Tannakian category, and it is shown in \cite{Iak22} that the corresponding gerbe over $F$ is $\tn{Kott}_{F}$. From this, one can deduce easily that the category of Drinfeld isoshtukas supported on $U = C \setminus |\Sigma|$ (that is, all $(V,\varphi)$ such that $V$ and $\varphi$ spread out to an $O_{F,\Sigma}$ module equipped with a Frobenius-semilinear automorphism) corresponds to the gerbe $\tn{Kott}_{F,\Sigma}$ from Proposition \ref{SigmaKott}. Moreover, it follows from \cite[\S 2.2]{HK21} that the category of Drinfeld isoshtukas is equivalent to the Tannakian category of finite-dimensional $F^{s}$-vector spaces, equipped with a discrete, semilinear $W_{F}$-action, where $W_{F}$ denotes the Weil group of $F$. We will use this latter interpretation of isoshtukas in the following.

We obtain from \cite[Proposition 3.1]{HK21} that $H_{\tn{\'{e}t}}^{1}(\tn{Kott}_{F}, \mathbb{G}_{m})$ is canonically identified with $\mathbb{Z}[V_{F}]_{0}$. Fix a set of lifts of $V_{F}$ in $V_{\ov{F}}$, denoted by $\dot{V}$. Fix also a finite Galois extension $E/F$ along with a finite set of places $\Sigma \subseteq V$ with lifts $\dot{\Sigma}^{(E)} \subseteq V_{E}$ satisfying the conditions as in the construction of $\tn{Kal}_{F}$ from \S \ref{Kalgerbedef} with respect to $E$, and consider the finite group scheme $P^{E}$ over $O_{F,\Sigma}$ (we omit the ``$\Sigma$'' from the notation because it is fixed throughout this discussion) with character group the subgroup of $\mathbb{Q}/\mathbb{Z}[\Gamma_{E/F} \times \Sigma_{E}]_{0,0}$ such that the coefficient of $[(\sigma, w)]$ is zero whenever $\sigma^{-1}w \notin \dot{\Sigma}^{(E)}$. It follows from the compatibility properties of the finite-level classes introduced \S \ref{Kalgerbedef} that there is a canonical class $\varprojlim_{n} \xi_{E,\dot{\Sigma}^{(E)}} \in \check{H}^{2}(O_{E,\Sigma^{(E)}}^{\tn{perf}}/O_{F,(V^{(E)})_{F}}, P^{E})$; we denote the gerbe corresponding to a choice of cocycle representative by $\tn{Kal}_{F,\Sigma}^{E}$, and then set $\tn{Kott}_{F,\Sigma}^{1/\infty,E} := \tn{Kott}_{F} \times_{F} \tn{Kal}_{F,\Sigma}^{E}$, denoted by just $\tn{Kott}_{F}^{1/\infty,E}$ if $\Sigma$ is understood.

Suppose we are given $c \in \mathbb{Z}[\Sigma]_{0}$; fix $x_{c,E} \in \breve{E}$ such that the image of $\tn{div}(x_{c,E}) \in \mathbb{Z}[V_{\breve{E}}]$ in $\mathbb{Z}[V_{E}]_{0}$ is $\dot{c}_{E} \in \mathbb{Z}[\dot{\Sigma}^{(E)}]_{0}$---this determines a rank-$1$ Drinfeld isoshtuka over $E$ and therefore a line bundle $\mathbb{L}_{E,c}$ on $\tn{Kott}_{E}$. For our fixed $E$ and $\Sigma$, we can and do choose the $x_{c,E}$ so that $x_{c+d,E} = x_{c,E} \cdot x_{d,E}$.

Now, given a function $\Gamma_{E/F} \xrightarrow{\alpha} \mathbb{Z}[\Sigma]_{0}$, one can associate a line bundle on $\tn{Spec}(E^{s})$ given by
\begin{equation*}
L_{\infty}^{\alpha} := x_{E}^{*}(\bigotimes_{\tau \in \Gamma_{E/F}} \tau^{*}(\mathbb{L}_{E,\alpha(\tau)})),
\end{equation*}
where $\tau$ is viewed as an automorphism of the covering $\tn{Kott}_{E} \to \tn{Kott}_{F}$ and $x_{E} \colon \tn{Spec}(E^{s}) \to \tn{Kott}_{E}$ is the standard fiber functor. Note that, by construction, $L_{\infty}^{\alpha}$ is a $1$-dimensional discrete, semilinear $W_{E}$-module over $F^{s}$.

For $\sigma \in W_{F}$, we set $(L_{\infty}^{\alpha})^{\sigma} := L_{\infty}^{\alpha} \otimes_{F^{s},\sigma} F^{s}$ and observe that the $W_{E}$-module $(L_{\infty}^{\alpha})^{\sigma}$ is naturally isomorphic to $L_{\infty}^{\prescript{\sigma}{}\alpha}$.

Recall from \cite[\S 2]{KT} that the character module of $P^{E}$ fits into a short exact sequence
\begin{equation}\label{tannakianSES}
0 \to \mathcal{C}(\Gamma_{E/F}, \mathbb{Z}[\Sigma]_{0}) \to X^{*}(\mathbb{T}^{E,\tn{mid}}) \to X^{*}(P^{E}) \to 0,
\end{equation}
where $X^{*}(\mathbb{T}^{E,\tn{mid}})$ is the subgroup of $\mathbb{Q}[\Gamma_{E/F} \times \Sigma_{E}]$ of elements which are killed by the $\Sigma_{E}$-augmentation map, land in $\mathbb{Z}[\Sigma_{E}]$ under the $\Gamma_{E/F}$-augmentation map, and such that the coefficient of a given $[(\sigma,w)]$ is zero unless $\sigma^{-1}w \in \dot{\Sigma}^{(E)}$. 

Fix a section $s \colon X^{*}(P^{E}) \to  X^{*}(\mathbb{T}^{E,\tn{mid}})$, determining a $2$-cocycle 
\begin{equation*}
c \in Z^{1}(X^{*}(P^{E}),  \mathcal{C}(\Gamma_{E/F}, \mathbb{Z}[\Sigma]_{0})).
\end{equation*}
We can now define the Tannakian category over $F$ corresponding to $\tn{Kott}_{F}^{1/\infty,E}$:

It has objects given by $X^{*}(P^{E})$-graded finite-dimensional $F^{s}$-vector spaces 
\begin{equation*}
D = \bigoplus_{\alpha \in X^{*}(P^{E})} D_{\alpha}
\end{equation*}
equipped with a discrete, semilinear $W_{F}$-action such that acting by $\tau \in W_{F}$ is an isomorphism from $D_{\alpha}$ to $D_{\prescript{\tau}{}\alpha}$, with tensor product given by 
\begin{equation*}
(D \otimes D')_{\alpha} = \bigoplus_{\beta + \gamma = \alpha} D_{\beta} \otimes_{F^{s}} D'_{\gamma} \otimes_{F^{s}} L_{\infty}^{c(\beta,\gamma)},
\end{equation*}
where $\tau \in W_{F}$ acts on $L_{\infty}^{c(\beta,\gamma)}$ by mapping it to $(L_{\infty}^{c(\beta,\gamma)})^{\tau} \simeq L_{\infty}^{\prescript{\tau}{}c(\beta,\gamma)}$ via the map $ L_{\infty}^{c(\beta,\gamma)} \to L_{\infty}^{c(\beta,\gamma)} \otimes_{F^{s},\tau} F^{s}$, $x \mapsto x \otimes 1$. Denote this category by $\tn{Isoc}_{F}^{E,e}$.

The key Tannakian result is then:

\begin{Proposition}\label{prop:tannakian}
The $\otimes$-category $\tn{Isoc}_{F}^{E,e}$ is a Tannakian category over $F$ banded by $\mathbb{D} \times P^{E}$ and corresponds to the isomorphism class of $\tn{Kott}_{F}^{1/\infty,E}$.
\end{Proposition}

\begin{proof}
The argument is essentially the same as its local analogue \cite[Th\'{e}or\`{e}me 6.4]{Fargues} which we summarize here. First, since this gerbe is by construction an inverse limit of gerbes with a finite band, it suffices to prove the analogue in that case (this amounts to replacing $\mathbb{Q}/\mathbb{Z}$ with $\frac{1}{n}\mathbb{Z}/\mathbb{Z}$ for some $n$). The fact that the Tannakian category $\tn{Rep}(\tn{Kott}_{F}^{1/\infty,E} \times_{F} F^{s})$ is isomorphic to $X^{*}(P^{E})$-graded modules with the above tensor structure is then a formal consequence of \cite[Proposition 3.4]{Fargues} (using that the first and second groups in the sequence \eqref{tannakianSES} are the character groups of tori), and then giving the descent datum from $F^{s}$ to $F$ amounts precisely to a family of semilinear isomorphisms from each $D_{\alpha}$ to $D_{\prescript{\tau}{}\alpha}$ for $\tau \in \Gamma_{F^{s}/F}$ and $\alpha \in X^{*}(P^{E})$.
\end{proof}

One consequence of this result is:

\begin{Corollary}\label{Cor:fiberfunctor}
The Tannakian category $\tn{Isoc}_{F}^{E,e}$ has a fiber functor over $F^{s}$.
\end{Corollary}

\begin{proof}
Let $\text{Isoc}_{E}^{(1)}$ be the quotient tensor category of $\text{Isoc}_{F}^{E,e}$ obtained by forgetting the $W_{F}$-action, the grading, and identifying all objects that differ by a twist by a rank-$1$ $E$-isocrystal (we leave it to the reader to check that this is a well-defined tensor category). The functor $\text{Isoc}_{F}^{E,e} \to \text{Isoc}_{E}^{(1)} \to \text{Vect}_{F^{s}}$ is an exact tensor functor, and is thus the desired fiber functor.
\end{proof}

If we have a tower $K/E/F$ of finite Galois extensions such that $(\Sigma^{(E)}, \dot{\Sigma}^{(E)}) \subseteq (\Sigma^{(K)}, \dot{\Sigma}^{(K)})$ in the obvious sense, then a choice of $x_{c,K} \in \breve{K}$ for all $c \in \mathbb{Z}[\Sigma^{(K)}]_{0}$ as above determines, via $x_{c,K} \mapsto N_{\breve{K}/\breve{E}}(x_{c,K})$, such a system for $E$ and $\Sigma^{(E)}$. Moreover, if we use these systems for both $K$ and $E$ to define $\tn{Isoc}^{E,e}$ and $\tn{Isoc}^{K,e}$ as above, then we have an inclusion of tensor categories $\tn{Isoc}_{F}^{E,e} \hookrightarrow \tn{Isoc}_{F}^{K,e}$ determined by the inclusion $X^{*}(P^{E}) \to X^{*}(P^{K})$, which corresponds to a map of gerbes $\tn{Kott}_{F}^{1/\infty,K} \to \tn{Kott}_{F}^{1/\infty,E}$.

According to \cite[\S 4.1]{Dillery23b}, due to the vanishing of $H_{\tn{fppf}}^{1}(F, \mathbb{D} \times P)$, for any fixed $E$ and finite-type group scheme $G$, the pullback $H_{\tn{\'{e}t}}^{1}(\tn{Kott}_{F}^{1/\infty,E}, G) \to H_{\tn{\'{e}t}}^{1}(\tn{Kott}_{F}^{1/\infty}, G)$ is canonical, factors through any of the finite-level pullbacks $H_{\tn{\'{e}t}}^{1}(\tn{Kott}_{F}^{1/\infty,E},G) \to H_{\tn{\'{e}t}}^{1}(\tn{Kott}_{F}^{1/\infty,K},G)$ described in the previous paragraph, and we have 
\begin{equation*}
H_{\tn{\'{e}t}}^{1}(\tn{Kott}_{F}^{1/\infty}, G) = \varinjlim_{(E, \Sigma_{E})} H_{\tn{\'{e}t}}^{1}(\tn{Kott}_{F}^{1/\infty,E}, G).
\end{equation*}

\subsection{The global extended Kottwitz set}\hfill\\
Recall that, for a connected reductive group $G$ over $F$, $B(G)$ is defined as $H^{1}_{\tn{\'{e}t}}(\tn{Kott}_{F}, G)$.
\begin{Definition}
We define the \textit{extended Kottwitz set} to be the subset of classes $[x] \in H^{1}_{\tn{\'{e}t}}(\tn{Kott}^{1/\infty}_{F}, G)$, denoted by $B_{e}(G)$, such that for any (equivalently, one) representative $x$ the associated homomorphism
\begin{equation*}
x|_{P_{\ov{F}}} \colon P_{\ov{F}} \to G_{\ov{F}}
\end{equation*}
factors through $Z(G)$; note that it is then automatically defined over $F$. We call an element $[x] \in B_{e}(F)$ \textit{basic} if $x|_{\mathbb{D}_{\ov{F}} \times P_{\ov{F}}}$ factors through $Z(G)$, and denote the corresponding subset by $B_{e}(G)_{\tn{basic}}$.
\end{Definition}

\begin{Remark}
We expect that the entire cohomology set $H^{1}_{\tn{\'{e}t}}(\tn{Kott}^{1/\infty}_{F}, G)$ also has applications in the Langlands correspondence and it will be studied in future work.
\end{Remark}

\begin{Lemma}
We have an exact sequence of pointed sets
\begin{equation*}
1 \to B(G) \to B_{e}(G) \to \mathrm{Hom}_{F}(P, Z(G)) \to 1,
\end{equation*}
where the first map is pullback from $\tn{Kott}_{F}$ to $\tn{Kott}_{F}^{1/\infty}$ and the second map is restriction to $P$.
\end{Lemma}

\begin{proof}
The map $\tn{Kott}_{F}^{1/\infty} \to \tn{Kott}_{F}$ gives $\tn{Kott}_{F}^{1/\infty}$ the structure of a gerbe, banded by $P_{\tn{Kott}_{F}}$, and inflation-restriction gives exactness at all parts of the sequence except the right-most term. For any maximal torus $T$ of $G$ we have $B_{e}(T) \twoheadrightarrow \tn{Hom}_{F}(P, T)$ again by inflation-restriction and the vanishing of $H^{2}(\tn{Kott}_{F}, T)$ by Proposition \ref{prop:H2torus}. In particular, we can find a class in $B_{e}(T)$ whose restriction to $P$ the composition $P \to Z(G) \to T$ and then take the image of this class in $B_{e}(G)$ (it maps into this subset because, by construction, the group $P$ maps into $Z(G)$) to get the desired result.
\end{proof}

On the other hand, the restriction of $[x] \in B_{e}(G)$ to $\mathbb{D}_{\ov{F}}$ gives a $G(\ov{F})$-conjugacy class in $\tn{Hom}(\mathbb{D}_{\ov{F}}, G_{\ov{F}})$, which we call the \textit{Newton point} of $[x]$. 

Recall that Kottwitz gave a canonical, functorial ``Tate-Nakayama duality'' isomorphism 
\begin{equation*}
B(T) \xrightarrow{\sim} (\mathbb{Z}[V_{F^{s}}]_{0} \otimes_{\mathbb{Z}} X_{*}(T))_{\Gamma_{F^{s}/F}}.
\end{equation*}
We now extend this to a ``Tate-Nakayama duality'' isomorphism for the cohomology of the gerbe $\tn{Kott}_{F}^{1/\infty}$. First, recall that \cite[Theorem 4.4]{Dillery23b} gives an analogous duality result for $H_{\tn{\'{e}t}}^{1}(\tn{Kal}_{F}, T)$, constructing a functorial isomorphism between this object and $(X_{*}(T)_{\mathbb{Q}})[V_{F^{s}}, \dot{V}]_{0,+,\tn{tor}}$ where $(X_{*}(T)_{\mathbb{Q}})[V_{F^{s}}, \dot{V}]_{0}$ denotes the subgroup of $(X_{*}(T)_{\mathbb{Q}})[V_{F^{s}}]_{0}$ which maps into $(X_{*}(T)_{\mathbb{Q}/\mathbb{Z}})[\dot{V}]_{0}$ via the natural map and 
\begin{equation*}
(X_{*}(T)_{\mathbb{Q}})[V_{F^{s}}, \dot{V}]_{0,+,\tn{tor}} := \frac{(X_{*}(T)_{\mathbb{Q}})[V_{F^{s}}, \dot{V}]_{0}}{I_{\Gamma_{F^{s}/F}} \cdot X_{*}(T)[V_{F^{s}}]_{0}}[\tn{tor}],
\end{equation*}
where $I_{\Gamma_{F^{s}/F}}$ is the augmentation ideal.

\begin{Theorem}\label{TNtorus}
There is a canonical, functorial isomorphism
\begin{equation*}
B_{e}(T) \xrightarrow{\sim} \frac{(X_{*}(T)_{\mathbb{Q}})[V_{F^{s}}, \dot{V}]_{0}}{I_{\Gamma_{F^{s}/F}} \cdot X_{*}(T)[V_{F^{s}}]_{0}}
\end{equation*}
which makes the following diagram commute
\[
\begin{tikzcd}
1 \arrow{r} & B(T) \arrow["\sim"]{d} \arrow{r} & B_{e}(T) \arrow{r} \arrow{d} & \tn{Hom}_{F}(P, T) \arrow["\sim"]{d} \arrow{r} & 1 \\
1 \arrow{r} & \frac{X_{*}(T)[V_{F^{s}}]_{0}}{I_{\Gamma_{F^{s}/F}} \cdot X_{*}(T)[V_{F^{s}}]_{0}} \arrow{r} & \frac{(X_{*}(T)_{\mathbb{Q}})[V_{F^{s}}, \dot{V}]_{0}}{I_{\Gamma_{F^{s}/F}} \cdot X_{*}(T)[V_{F^{s}}]_{0}} \arrow{r} & (X_{*}(T)_{\mathbb{Q}/\mathbb{Z}})[\dot{V}]_{0} \arrow{r} & 1,
\end{tikzcd}
\]
where the first isomorphism is as above and the third follows from \cite[Lemma 3.10]{Dillery23b}.
\end{Theorem}

\begin{proof}
The key step is proving this result for $T = \mathrm{Res}_{E/F}(\mathbb{G}_{m})$ an induced torus, where $E/F$ is a finite Galois extension. When $T = \mathbb{G}_{m}$, we seek an isomorphism 
\begin{equation*}in
B_{e}(\mathbb{G}_{m}) \xrightarrow{\sim}  \frac{\mathbb{Q}[V_{F^{s}}, \dot{V}]_{0}}{I_{\Gamma_{F^{s}/F}} \cdot \mathbb{Z}[V_{F^{s}}]_{0}} \xrightarrow{\sim} \mathbb{Q}[V_{F}]_{0},
\end{equation*}
where the second isomorphism comes from the surjective map $\mathbb{Q}[V_{F^{s}}, \dot{V}]_{0} \to  \mathbb{Q}[V_{F}]_{0}$ sending $w$ to $w|_{F}$; the kernel is evidently the same as the kernel of the analogous map $\mathbb{Z}[V_{F^{s}}]_{0} \to  \mathbb{Z}[V_{F}]_{0}$, which is $I_{\Gamma_{F^{s}/F}} \cdot \mathbb{Z}[V_{F^{s}}]_{0}$ (for this last claim, take the long exact sequence coming from taking $\Gamma_{F^{s}/F}$-coinvariants of the short exact sequence $0 \to \Z[V_{F^{s}}]_{0} \to \Z[V_{F^{s}}] \to \Z \to 0$ and use that $H_{1}(\Gamma_{F^{s}/F}, \Z[V_{F^{s}}]) \to H_{1}(\Gamma_{F^{s}/F}, \Z)$ is surjective).

It then follows from taking the $\Gamma_{E/F}$-invariants of the short exact sequence 
\begin{equation}\label{tannakianSES}
0 \to \mathcal{C}(\Gamma_{E/F}, \mathbb{Z}[\Sigma^{(E)}]_{0}) \to X^{*}(\mathbb{T}_{\Sigma^{(E)}}^{E,\tn{mid}}) \to X^{*}(P_{\Sigma^{(E)}}^{E}) \to 0
\end{equation}
(which is exact, since the kernel is induced---we are also now denoting the set of places to avoid confusion) and Proposition \ref{prop:tannakian} that each $H^{1}_{\tn{\'{e}t}}(\tn{Kott}^{1/\infty,E}_{F,\Sigma^{(E)}}, \mathbb{G}_{m})$ is isomorphic to $X^{*}(\mathbb{T}_{\Sigma^{(E)}}^{E,\tn{mid}})^{\Gamma_{E/F}}$, and so from the description of $H^{1}_{\tn{\'{e}t}}(\tn{Kott}^{1/\infty}_{F}, \mathbb{G}_{m}) = \varinjlim_{E} H^{1}_{\tn{\'{e}t}}(\tn{Kott}^{1/\infty,E}_{F,\Sigma^{(E)}}, \mathbb{G}_{m})$ we obtain an identification of this set with $X^{*}(\mathbb{T}^{\tn{mid}})^{\Gamma_{F^{s}/F}}$ (as defined in \cite[\S 2.4]{KT}), which is canonically isomorphic (by the proof of Proposition 3.4.2 in \cite{KT}) to $\varinjlim_{E} \mathbb{Q}[V_{E}]_{0}^{\Gamma_{E/F}} \xrightarrow{\sim} \mathbb{Q}[V_{F}]_{0}$ in a manner that makes the following diagram commute:
\begin{equation}
\begin{tikzcd}\label{Gmdiagram}
1 \arrow{r} & B(\mathbb{G}_{m}) \arrow["\sim"]{d} \arrow{r} & B_{e}(\mathbb{G}_{m}) \arrow{r} \arrow["\sim"]{d} & \tn{Hom}_{F}(P, \mathbb{G}_{m}) \arrow["\sim"]{d} \arrow{r} & 1 \\
1 \arrow{r} & \mathbb{Z}[V_{F}]_{0} \arrow{r} & \mathbb{Q}[V_{F}]_{0} \arrow{r} & \mathbb{Q}/\mathbb{Z}[\dot{V}]_{0} \arrow{r} & 1.
\end{tikzcd}
\end{equation}

We may then deduce, for any $E/F$, (cf. the proof of \cite[Proposition 9.7]{Fargues}) that the short exact sequence
\begin{equation*}
1 \to B(\tn{Res}_{E/F}(\mathbb{G}_{m})) \to B_{e}(\tn{Res}_{E/F}(\mathbb{G}_{m})) \to \tn{Hom}_{F}(P, \tn{Res}_{E/F}(\mathbb{G}_{m})) \to 1
\end{equation*} 
is the pull-back of the first line of \eqref{Gmdiagram} by the morphism 
\begin{equation*}
\tn{Hom}_{F}(P,  \tn{Res}_{E/F}(\mathbb{G}_{m})) \xrightarrow{N_{E/F}^{*}} \tn{Hom}_{F}(P,  \mathbb{G}_{m}) = \mathbb{Q}/\mathbb{Z}[\dot{V}]_{0}.
\end{equation*}
We therefore have a commutative diagram
\begin{equation*}
\begin{tikzcd}
1 \arrow{r} & B(\tn{Res}_{E/F}(\mathbb{G}_{m})) \arrow["\sim"]{d} \arrow{r} & B_{e}(\tn{Res}_{E/F}(\mathbb{G}_{m})) \arrow{r} \arrow["\sim"]{d} & \tn{Hom}_{F}(P,  \tn{Res}_{E/F}(\mathbb{G}_{m}))\arrow["\sim"]{d} \arrow{r} & 1 \\
1 \arrow{r} & \mathbb{Z}[V_{F}]_{0} \arrow{r} \arrow["\id"]{d} &\frac{\mathbb{Q}[\Gamma_{E/F} \times V_{F}]_{0}}{\mathbb{Z}[\Gamma_{E/F} \times V_{F}]_{0,0}}  \arrow["\aug"]{d} \arrow{r} & \mathbb{Q}/\mathbb{Z}[\Gamma_{E/F} \times \dot{V}]_{0} \arrow{r} \arrow["\aug"]{d} & 1 \\
1 \arrow{r} & \mathbb{Z}[V_{F}]_{0} \arrow{r} & \mathbb{Q}[V_{F}]_{0} \arrow{r} & \mathbb{Q}/\mathbb{Z}[\dot{V}]_{0} \arrow{r} & 1,
\end{tikzcd}
\end{equation*}
where the subscript ``$0,0$'' means those elements killed by both augmentation maps and the first map in the middle row sends $\sum_{v \in V_{F}} c_{v}[v]$ to $\sum_{v \in V_{F}} \sum_{\gamma \in \Gamma_{E/F}} c_{v}[\gamma \times v]$. This gives the desired result for $T = \tn{Res}_{E/F}(\mathbb{G}_{m})$ after making the identification 
\begin{equation}\label{inducedisomeq}
\frac{\mathbb{Q}[\Gamma_{E/F}][V_{F^{s}}, \dot{V}]_{0}}{I_{\Gamma_{F^{s}/F}} \cdot \mathbb{Z}[\Gamma_{E/F}][V_{F^{s}}]_{0}} \xrightarrow{\sim} \frac{\mathbb{Q}[\Gamma_{E/F} \times V_{F}]_{0}}{\mathbb{Z}[\Gamma_{E/F} \times V_{F}]_{0,0}} 
\end{equation}
induced by the map $\mathbb{Q}[\Gamma_{E/F}][V_{F^{s}}, \dot{V}]_{0} \to \mathbb{Q}[\Gamma_{E/F} \times V_{F}]_{0}$ sending $w$ to $w|_{F}$. 

We now deduce the result for general $T$. A given $x \in X_{*}(T)[V_{F_{s}}]$ yields a $V_{F_{s}}$-tuple of morphisms $\mathbb{G}_{m,E} \to T_{E}$ for some sufficiently large finite Galois $E$, giving rise to an analogous $V_{F_{s}}$-tuple $(\nu_{x,w})_{w \in V_{F^{s}}}$ of morphisms $\tn{Res}_{E/F}(\mathbb{G}_{m}) \to T$ (note that all but finitely-many of these morphisms are trivial).

We thus obtain from $x$, by the induced case (using the left-hand side of the isomorphism \eqref{inducedisomeq}), a map
 \begin{equation*}
[ \frac{\mathbb{Q}[\Gamma_{E/F}][V_{F^{s}}, \dot{V}]_{0}}{I_{\Gamma_{F^{s}/F}} \cdot \mathbb{Z}[\Gamma_{E/F}][V_{F^{s}}]_{0}}]^{V_{F^{s}}} \to B_{e}(\tn{Res}_{E/F}(\mathbb{G}_{m}))^{V_{F^{s}}} \xrightarrow{\prod_{w \in V_{F^{s}}} \nu_{x,w}} B_{e}(T),
 \end{equation*}
 denoted by $B_{e}(x)$.
 
Fix an auxiliary place $v \in V_{F}$ (the constructions will not depend on this choice). For an element $y \in X_{*}(T)_{\Q}[V_{F^{s}}, \dot{V}]_{0}$, write $y = \sum_{w \in V_{F^{s}}} (\lambda_{w} \otimes x_{w})[w]$ for $x_{w} \in  X_{*}(T)$ and $\lambda_{w} \in \Q$, where $\lambda_{w} \in \Z$ if $w \notin \dot{V}$. We define the duality map by sending $y$ to
\begin{equation} \label{TNweilres}
B_{e}(\sum_{w \in V_{F^{s}}} x_{w}[w])[(\ov{\lambda_{w} [1 \times w]- \lambda_{w} [1 \times \dot{v}]})_{w \in V_{F^{s}}}] \in B_{e}(T),
\end{equation}
noting that each $\lambda_{w} [1 \times w]- \lambda_{w} [1 \times \dot{v}]$ lies in $\mathbb{Q}[\Gamma_{E/F}][V_{F^{s}}, \dot{V}]_{0}$, where $1 := 1_{\Gamma_{E/F}}$ and we denote by $\ov{\lambda_{w} [1 \times w] - \lambda_{w} [1 \times \dot{v}]}$ the image of $\lambda_{w} [1 \times w] - \lambda_{w} [1 \times \dot{v}]$ in $\frac{\mathbb{Q}[\Gamma_{E/F}][V_{F^{s}}, \dot{V}]_{0}}{I_{\Gamma_{F^{s}/F}} \cdot \mathbb{Z}[\Gamma_{E/F}][V_{F^{s}}]_{0}}$. It is straightforward to check that the above formula does not depend on the choice of $v$ nor the choice of the $x_{w}$, $\lambda_{w}$ and that the resulting map $X_{*}(T)_{\Q}[V_{F^{s}}, \dot{V}]_{0} \to B_{e}(T)$ is a group homomorphism which is trivial on $I_{\Gamma_{F^{s}/F}} \cdot X_{*}(T)[V_{F^{s}}]_{0}$ (for this last fact, one uses that $B_{e}(\prescript{\sigma}{}x)(\bar{z}) = B_{e}(x)(r_{[\sigma]}\bar{z})$ for $x \colon \tn{Res}_{E/F}(\mathbb{G}_{m}) \to T$, $\bar{z} \in \frac{\mathbb{Q}[\Gamma_{E/F}][V_{F^{s}}, \dot{V}]_{0}}{I_{\Gamma_{F^{s}/F}} \cdot \mathbb{Z}[\Gamma_{E/F}][V_{F^{s}}]_{0}}$, and $r_{[\sigma]}$ the automorphism of  $\frac{\mathbb{Q}[\Gamma_{E/F}][V_{F^{s}}, \dot{V}]_{0}}{I_{\Gamma_{F^{s}/F}} \cdot \mathbb{Z}[\Gamma_{E/F}][V_{F^{s}}]_{0}}$ given by right-translation by $\sigma$ on the $\Gamma_{E/F}$-coefficients). This map is an isomorphism because we may resolve $T$ by induced tori and we already proved the result in the induced case.
\end{proof}

We record the following consequence of Theorem \ref{TNtorus}, which gives an alternative (substantially shorter) proof the main cohomological result (Theorem 4.4) in \cite{Dillery23b}:

\begin{Corollary}\label{TNcompat}
The composition
\begin{equation*}
H^{1}_{\tn{\'{e}t}}(\tn{Kal}_{F}, T) \to B_{e}(T) \to \frac{(X_{*}(T)_{\mathbb{Q}})[V_{F^{s}}, \dot{V}]_{0}}{I_{\Gamma_{F^{s}/F}} \cdot X_{*}(T)[V_{F^{s}}]_{0}}
\end{equation*}
induces a map 
\begin{equation}\label{TNKal}
H^{1}_{\tn{\'{e}t}}(\tn{Kal}_{F}, T) \to (X_{*}(T)_{\mathbb{Q}})[V_{F^{s}}, \dot{V}]_{0,+,\tn{tor}} 
\end{equation}
which is the duality isomorphism from \cite[Theorem 4.4]{Dillery23b}.
\end{Corollary}

\begin{proof}
The full statement of \cite[Theorem 4.4]{Dillery23b} is that the duality isomorphism loc. cit. is uniquely characterized by making the diagram
\[
\begin{tikzcd}
1 \arrow{r} & H^{1}(F, T) \arrow{r} \arrow["\sim"]{d}& H^{1}_{\tn{\'{e}t}}(\tn{Kal}_{F}, T) \arrow{r} \arrow{d} &  \tn{Hom}_{F}(P, T) \arrow["\sim"]{d} \\
1 \arrow{r} & (X_{*}(T)_{\mathbb{Q}})[V_{F^{s}}]_{0,\Gamma_{F^{s}/F}}[\tn{tor}] \arrow{r} &  (X_{*}(T)_{\mathbb{Q}})[V_{F^{s}}, \dot{V}]_{0,+,\tn{tor}} \arrow{r} &  (X_{*}(T)_{\mathbb{Q}/\mathbb{Z}})[\dot{V}]_{0}
\end{tikzcd}
\]
commute. Theorem \ref{TNtorus} says that this diagram commutes for the map \eqref{TNKal}.

Moreover, Theorem \ref{TNtorus} also implies that the image of the composition
\begin{equation*}
H^{1}_{\tn{\'{e}t}}(\tn{Kal}_{F}, T) \to B_{e}(T) \xrightarrow{\tn{Theorem \ref{TNtorus}}}  \frac{(X_{*}(T)_{\mathbb{Q}})[V_{F^{s}}, \dot{V}]_{0}}{I_{\Gamma_{F^{s}/F}} \cdot X_{*}(T)[V_{F^{s}}]_{0}} \to  (X_{*}(T)_{\mathbb{Q}/\mathbb{Z}})[\dot{V}]_{0}
\end{equation*}
coincides with the image of the composition
\begin{equation}\label{Compateq1}
H^{1}_{\tn{\'{e}t}}(\tn{Kal}_{F}, T) \to \tn{Hom}_{F}(P, T)  \xrightarrow{\sim}  (X_{*}(T)_{\mathbb{Q}/\mathbb{Z}})[\dot{V}]_{0}.
\end{equation}
By \cite[Theorem 4.4]{Dillery23b}, the image of \eqref{Compateq1} is the same as the image of the map
\begin{equation*}
 (X_{*}(T)_{\mathbb{Q}})[V_{F^{s}}, \dot{V}]_{0,+,\tn{tor}}\to  (X_{*}(T)_{\mathbb{Q}/\mathbb{Z}})[\dot{V}]_{0}.
\end{equation*}
From this observation, a diagram chase shows that \eqref{TNKal} is an isomorphism, so we deduce the result by the uniqueness part of \cite[Theorem 4.4]{Dillery23b}
\end{proof}

We conclude this subsection by discussing some localization properties of the above isomorphism. Recall that Fargues in \cite[Proposition 9.7]{Fargues} proves an analogous result for the local gerbe $\tn{Kott}_{v}^{1/\infty}:= \tn{Kott}_{v} \times_{F_{v}} \tn{Kal}_{v}$, which is a functorial isomorphism
\begin{equation}\label{locTN}
B_{e}(T_{F_{v}}) \xrightarrow{\sim} \frac{X_{*}(T)_{\mathbb{Q}}}{I_{\Gamma_{F_{v}^{s}/F_{v}}}X_{*}(T)}.
\end{equation}

\begin{Remark}\label{FarguesRem}
Strictly speaking, Fargues defines the isomorphism \eqref{locTN} for the analogues of these gerbes over $p$-adic local fields. Nevertheless, a similar argument holds in the local function field case---the only difference is that one needs to adapt certain cohomological-vanishing arguments to deal with the non-smoothness of $u_{v}$, as in the proof of Proposition \ref{vanprop1}. 
\end{Remark}

Observe that there is a canonical map
\begin{equation*}
\frac{(X_{*}(T)_{\mathbb{Q}})[V_{F^{s}}, \dot{V}]_{0}}{I_{\Gamma_{F^{s}/F}} \cdot X_{*}(T)[V_{F^{s}}]_{0}} \xrightarrow{(l_{v})_{v}} \bigoplus_{v \in V_{F}} \frac{X_{*}(T)_{\mathbb{Q}}}{I_{\Gamma_{F_{v}^{s}/F_{v}}}X_{*}(T)}
\end{equation*}
given by the formula in \cite[(15)]{Dillery23b}. One then has:

\begin{Proposition}\label{localtoglobal2}
For $T$ a torus over $F$, the following localization results hold:
\begin{enumerate}
\item{We have the following commutative diagram with exact bottom row
\[
\begin{tikzcd}
B_{e}(T) \arrow["(\tn{loc}_{v})_{v}"]{r} & \bigoplus_{v \in V_{F}} B_{e}(T_{F_{v}})  \\
 \frac{(X_{*}(T)_{\mathbb{Q}})[V_{F^{s}}, \dot{V}]_{0}}{I_{\Gamma_{F^{s}/F}} \cdot X_{*}(T)[V_{F^{s}}]_{0}} \arrow["\tn{Theorem \ref{TNtorus}}"]{u} \arrow["(l_{v})_{v}"]{r} & \bigoplus_{v \in V_{F}} \frac{X_{*}(T)_{\mathbb{Q}}}{I_{\Gamma_{F_{v}^{s}/F_{v}}}X_{*}(T)}
 \arrow["\tn{\cite[Prop. 9.7]{Fargues}}"]{u} 
\end{tikzcd}
\]}
\item{For $G$ a connected reductive group and $x \in B_{e}(T)$, $\tn{loc}_{v}(x)$ is the neutral element in $B_{e}(T_{F_{v}})$ for almost all $v \in V_{F}$.}
\end{enumerate}
\end{Proposition}

\begin{proof}
It suffices to prove the first statement. The localization map $l_{v}$ is defined at ``level $k$'' as follows: choose a representative $\dot{\tau} \in \Gamma_{E_{k}/F}$ for each right coset $\tau \in \Gamma_{E_{k}/F,\dot{v}} \backslash \Gamma_{E_{k}/F}$ such that $\dot{\tau} = 1$ for the trivial coset; for $f= \sum_{w \in (\Sigma^{(E_{k})})_{E_{k}}} c_{w}[w] \in X_{*}(T)_{\Q}[(\Sigma^{(E_{k})})_{E_{k}}, \dot{\Sigma}^{(E_{k})}]_{0}$, set 
\begin{equation*}
l^{k}_{v}(f) = \sum_{\tau \in \Gamma_{E_{k}/F,\dot{v}} \backslash \Gamma_{E_{k}/F}} \prescript{\dot{\tau}}{}c_{\prescript{\tau^{-1}}{}(\dot{v})} \in X_{*}(T)_{\Q} 
\end{equation*}
and then take the image in $\frac{X_{*}(T)_{\mathbb{Q}}}{I_{\Gamma_{F_{v}^{s}/F_{v}}}X_{*}(T)}$ we obtain the result by observing that our global duality formula \eqref{TNweilres} in the proof of Theorem \ref{TNtorus} yields the above formula for $(l_{v})_{v}$ after applying the local analogue of \eqref{TNweilres} on \cite[p. 31]{Fargues} at each $v$.
\end{proof}

\subsection{The extended Kottwitz map for reductive groups}\label{sec:Kottmap} \hfill\\
The goal is to extend the duality result for $B_{e}(T)$ given above to $B_{e}(G)_{\tn{basic}}$ for general connected reductive $G$. The next result helps us reduce to the case where $G$ is quasi-split:

\begin{Lemma}\label{lem:qsplit}
If $G^{*}$ is a quasi-split connected reductive group over $F$ and $b \in B_{e}(G^{*})_{\tn{basic}}$ with associated inner form $G^{*}_{b}$, then there is a bijection $B_{e}(G^{*}) \xrightarrow{f_{b}} B_{e}(G^{*}_{b})$ sending $b$ to $1$.
\end{Lemma}

\begin{proof}
This is immediate from Proposition \ref{prop:multvanish}.
\end{proof}

Denote by $\pi_{1}(G)$ the algebraic fundamental group of $G$. The global duality result proved by Kottwitz in \cite{Kot14} gives a map
\begin{equation*}
B(G) \xrightarrow{\kappa_{G}} (\Z[V_{F^{s}}]_{0} \otimes \pi_{1}(G))_{\Gamma_{F^{s}/F}}
\end{equation*}
restricting to an isomorphism on $B(F,G)_{\text{basic}}$, which we will generalize.

The desired linear algebraic object is defined as follows: Fix a maximal torus $T$ of $G$ with absolute root system $\Phi$ (with coroots $\check{\Phi}$) and set $\langle \Phi \rangle^{*} := \{\nu \in X_{*}(T)_{\Q} | \langle \alpha, \nu \rangle \in \Z \hspace{1mm} \forall \alpha \in \Phi\}$. Our desired object is then 
\begin{equation*}
\pi_{1}(G)^{e}_{\Gamma_{F^{s}/F}} := \frac{\langle \Phi \rangle^{*}/\langle \check{\Phi} \rangle[V_{F^{s}},\dot{V}]_{0}}{I_{\Gamma_{F^{s}/F}} \cdot \pi_{1}(G)[V_{F^{s}}]_{0}}.
\end{equation*}
Some remarks are in order: First, we observe that canonically $\langle \Phi \rangle^{*}/\langle \check{\Phi} \rangle \xrightarrow{\sim} \pi_{1}(G_{\tn{ad}}) \oplus X_{*}(Z(G))_{\Q}$. In a manner analogous to the torus case, the subgroup $\langle \Phi \rangle^{*}/\langle \check{\Phi} \rangle[V_{F^{s}},\dot{V}]$ denotes those elements $\sum_{w} c_{w}[w]$ of $\langle \Phi \rangle^{*}/\langle \check{\Phi} \rangle[V_{F^{s}},\dot{V}]_{0}$ such that $c_{w} \in \pi_{1}(G)[V_{F^{s}}]$ if $w \notin \dot{V}$. Finally, as argued in \cite[\S 9.5]{Fargues}, the (absolute) Weyl group $W$ of $T$ in $G$ acts trivially on $\langle \Phi \rangle^{*}/\langle \check{\Phi} \rangle$, and consequently on $\pi_{1}(G)^{e}_{\Gamma_{F^{s}/F}}$, which means that this definition is independent of the choice of $T$ up to canonical isomorphism.

\begin{Lemma}\label{toruslift}
Any two elements $b_{1}, b_{2} \in B_{e}(G)_{\text{basic}}$ lie in the image of $B_{e}(T)$ for some maximal torus $T$ of $G$. In particular, any element $b \in B_{e}(G)$ is in the image of $B_{e}(T)$ for such a $T$.
\end{Lemma}

\begin{proof}
Apply the proof of \cite[Lemme 9.11]{Fargues}, replacing the surjectivity of 
\begin{equation*}
H^{1}(F_{v}, T_{\tn{ad}}) \to H^{1}(F_{v}, G_{\tn{ad}})
\end{equation*}
in the local case (which does not hold globally) with the fact that any two elements of $H^{1}(F, G_{\tn{ad}})$ lie in the image $H^1(F, T_{\tn{ad}})$ for some maximal torus $T$. This claim follows from the proof of \cite[Theorem 4.11]{Dillery23b}, specifically the observation after equation (18) applied to $Z = \{\text{id}\}$ (this argument itself is essentially the same as the one from the proof of \cite[Theorem 3.8.1]{Kaletha18}).

The second statement comes from the fact that $b \in B_{e}(G)$ is in the image of $B_{e}(Z_{G}(\nu_{x}))_{\text{basic}}$, where $Z_{G}(\nu_{x})$ is a Levi subgroup of $G$, combined with the first part.
\end{proof}

Assume for the moment that $G$ is quasi-split. We can now define the ``extended Kottwitz map'' for $G$, written as
\begin{equation*}
B_{e}(G) \xrightarrow{\kappa_{G}} \pi_{1}(G)^{e}_{\Gamma_{F^{s}/F}} 
\end{equation*}
as follows. Given $b \in B_{e}(G)$, we may use Lemma \ref{toruslift} to lift it to $b' \in B_{e}(T)$ for some maximal torus $T$ and then use the isomorphism from Theorem \ref{TNtorus} to obtain an element $x_{b'} \in \frac{(X_{*}(T)_{\mathbb{Q}})[V_{F^{s}}, \dot{V}]_{0}}{I_{\Gamma_{F^{s}/F}} \cdot X_{*}(T)[V_{F^{s}}]_{0}}$. The fact that $b|_{u_{\ov{F}}}$ factors through $Z(G)$ implies that the image of $x_{b'}$ in $(X_{*}(T)_{\mathbb{Q}/\Z})[\dot{V}]_{0}$ pairs to zero with the subgroup $\langle \Phi \rangle_{\Q/\Z}[\dot{V}] \subseteq X^{*}(T)_{\Q/\Z}$, and therefore $x_{b'}$ lies in the subgroup $\frac{\langle \Phi \rangle^{*}[V_{F^{s}}, \dot{V}]_{0}}{I_{\Gamma_{F^{s}/F}} \cdot X_{*}(T)[V_{F^{s}}]_{0}}$. It then makes sense to take the image of $x_{b'}$ in $\pi_{1}(G)^{e}_{\Gamma_{F^{s}/F}}$, which we define to be $\kappa_{G}(b)$. After observing that any two elements in the same fiber of $B_{e}(T) \to B_{e}(G)$ are $W$-conjugate (by twisting, this reduces this to the case of $H^{1}(F, T) \to H^{1}(F, G)$), we deduce from the triviality of the action of $W$ on $\langle \Phi \rangle^{*}/\langle \check{\Phi} \rangle$ that $\kappa_{G}(b)$ is independent of both $T$ and $b'$.

As in \cite[\S 9.3.5]{Fargues}, this can be extended canonically to arbitrary $G$ by picking $b \in B_{e}(G)_{\text{basic}}$ such that $G_{b}$ is quasi-split, so that (by Lemma \ref{lem:qsplit}) there is a bijection $B_{e}(G) \xrightarrow{f_{b}} B_{e}(G_{b})$ and then setting $\kappa_{G} := \kappa_{G_{b}} \circ f_{b} - \kappa_{G_{b}}(f_{b}(1))$. As argued loc. cit., this map is independent of the choice of element $b$. 

Observe that, by construction, the map $\kappa_{G}$ fits into the commutative diagram
\[
\begin{tikzcd}
1 \arrow{r} & B(G) \arrow["\kappa_{G}"]{d} \arrow{r} & B_{e}(G) \arrow{r} \arrow["\kappa_{G}"]{d} & \tn{Hom}_{F}(P, Z(G)) \arrow["\sim"]{d} \arrow{r} & 1 \\
1 \arrow{r} & (\pi_{1}(G)[V_{F^{s}}]_{0})_{\Gamma_{F^{s}/F}} \arrow{r} & \pi_{1}(G)^{e}_{\Gamma_{F^{s}/F}} \arrow{r} & (X_{*}(Z(G))_{\mathbb{Q}/\mathbb{Z}})[\dot{V}]_{0} \arrow{r} & 1,
\end{tikzcd}
\]
where $\kappa_{G}$ is the global Kottwitz map for $B(F,G)$ and again the right-most isomorphism is from \cite[Lemma 3.10]{Dillery23b}. The final main result in this direction is:

\begin{Theorem}\label{TNgp}
The map $\kappa_{G}$ restricts to a bijection $B_{e}(G)_{\tn{basic}} \xrightarrow{\sim} \pi_{1}(G)^{e}_{\Gamma_{F^{s}/F}}$.
\end{Theorem}

\begin{proof}
For surjectivity, we first note that there is a canonical direct sum decomposition 
\begin{equation}\label{fundirectsum}
\langle \Phi \rangle^{*} = (\langle \Phi \rangle^{*} \cap \langle \check{\Phi} \rangle_{\Q}) \oplus \langle \Phi \rangle_{\Q}^{\perp},
\end{equation}
and observe that $\langle \check{\Phi} \rangle$ lies only in the left-hand summand. Now we claim that the map
\begin{equation*}
\langle \Phi \rangle^{*}[V_{F^{s}}, \dot{V}]_{0} \to \langle \Phi \rangle^{*}/ \langle \check{\Phi} \rangle[V_{F^{s}}, \dot{V}]_{0} = [\frac{\langle \Phi \rangle^{*} \cap \langle \check{\Phi} \rangle_{\Q}}{\langle \check{\Phi} \rangle} \oplus \langle  \Phi \rangle_{\Q}^{\perp}][V_{F^{s}}, \dot{V}]_{0}
\end{equation*}
is surjective, where the left-hand term is the subgroup of elements of $\langle \Phi \rangle^{*}[V_{F^{s}}]_{0}$ with $x_{w} \in X_{*}(T_{\text{der}})$ if $w \notin \dot{V}$. Let $\sum_{i} (\ov{a_{i}} \oplus c_{i})[w_{i}]$ be an element of the right-hand side, with lift $\sum_{i} (a_{i} \oplus c_{i})[w_{i}]$ in the left-hand side---observe that $\sum_{i} a_{i} \in \langle \check{\Phi} \rangle$, $c_{i} = 0$ if $w_{i} \not\in \dot{V}$, $\sum_{i} c_{i} = 0$, and, since $\bar{a_{i}}$ is in the image of $\pi_{1}(G)$ if $w_{i} \notin \dot{V}$, we may take $a_{i} \in X_{*}(T_{\text{der}})$ for such $w_{i}$. We then pick one auxiliary place $\dot{v} \in \dot{V}$ and our desired lift is given by $\sum_{i} (a_{i} \oplus c_{i})[w_{i}] - (\sum_{i} a_{i} \oplus 0)[\dot{v}] \in \langle \Phi \rangle^{*}[V_{F^{s}}, \dot{V}]_{0}$.

This claim shows that we may lift any given element $x$ of $\pi_{1}(G)^{e}_{\Gamma_{F^{s}/F}}$ to an element $x'$ of $\pi_{1}(T)^{e}_{\Gamma_{F^{s}/F}}$ represented by an element in $\langle \Phi \rangle^{*}[V_{F^{s}}, \dot{V}]_{0}$. Since $x'$ is of this form, its preimage under the isomorphism $B_{e}(T) \to \pi_{1}(T)^{e}_{\Gamma_{F^{s}/F}}$ from Theorem \ref{TNtorus} has image in $B_{e}(G)$ lying in $B_{e}(G)_{\text{basic}}$, and we denote this image by $b_{x}$. By the functoriality of $\kappa$ we see that $\kappa_{G}(b_{x}) = x$, giving the surjectivity of $\kappa_{G}$.

Since by Lemma \ref{toruslift} any two $b, b' \in B_{e}(G)_{\text{basic}}$ lie in the image of $B_{e}(T)$ for some maximal torus $T$ of $G$, the argument for injectivity in the proof of (\cite[Theor\`{e}me 9.12]{Fargues} reduces us to proving the surjectivity of 
\begin{equation*}
\text{ker}[B_{e}(T) \to B_{e}(T_{\tn{ad}})] \to \text{ker}[\pi_{1}(G)^{e}_{\Gamma_{F^{s}/F}} \to \pi_{1}(G_{\tn{ad}})^{e}_{\Gamma_{F^{s}/F}} = (\pi_{1}(G_{\tn{ad}})[V_{F^{s}}]_{0})_{\Gamma_{F^{s}/F}}].
\end{equation*}
Using the decomposition \eqref{fundirectsum}, this means that we need to lift any given element of $I_{\Gamma_{F^{s}/F}} \cdot \pi_{1}(G_{\tn{ad}})[V_{F^{s}}]_{0}$ lying in $\pi_{1}(G_{\tn{ad}})[V_{F^{s}}, \dot{V}]_{0}$ to an element of $X_{*}(T_{\tn{ad}})[V_{F^{s}}, \dot{V}]_{0}$ which lies in $I_{\Gamma_{F^{s}/F}} \cdot X_{*}(T_{\tn{ad}})[V_{F^{s}}]_{0}$. We can find such a lift $x \in X_{*}(T_{\tn{ad}})[V_{F^{s}}]_{0}$, and since the image of $x$ in $\pi_{1}(G_{\tn{ad}})[V_{F^{s}}]$ lands in $\pi_{1}(G_{\tn{ad}})[V_{F^{s}}, \dot{V}]_{0}$, in fact already $x \in X_{*}(T_{\tn{ad}})[V_{F^{s}},\dot{V}]_{0}$, giving the result.
\end{proof}

\section{Geometric considerations}\label{sec:geom}
We now geometrize the above framework using the constructions in the main body of the paper.

\subsection{Adic gerbes and bundles} \hfill\\
Pick a \v{C}ech $2$-cocycle $\dot{\xi}$ representing the canonical class of $\check{H}^{2}(\ov{F}/F, P) = H_{\tn{fppf}}^{2}(F,P)$ corresponding to $\tn{Kal}_{F}$---recall from the comments immediately following the proof of Proposition \ref{prop:H2invlim} that this determines a projective system of cocycles $(\dot{\xi}_{\Sigma})_{\Sigma}$ and a corresponding projective system of gerbes $(\tn{Kal}_{F,\Sigma})_{\Sigma}$ as $\Sigma$ ranges over all finite subsets of closed points of $C$---fix such a set of places $(\Sigma, \dot{\Sigma})$ with $U = \tn{Spec}(O_{F,\Sigma})$. 

The group scheme $P_{\dot{\Sigma}} \to \tn{Spec}(O_{F,\Sigma})$ gives a morphism of adic spaces $P_{\dot{\Sigma}}^{\tn{ad}} \to \tn{Spa}(O_{F,\Sigma})$, and $P_{\dot{\Sigma}}^{\tn{ad}}$ is a group sheaf on the v-site over $\tn{Spa}(O_{F,\Sigma})$, which we will henceforth denote by $P_{\dot{\Sigma}}$. Since we have $\tn{Spa}(O_{\Sigma}^{\tn{perf}} \otimes_{O_{F,\Sigma}} O_{\Sigma}^{\tn{perf}}) = \tn{Spa}(O_{\Sigma}^{\tn{perf}}) \times_{\tn{Spa}(O_{F,\Sigma})}  \tn{Spa}(O_{\Sigma}^{\tn{perf}})$ (and the same holds for fibered products with more terms), the cocycle $\dot{\xi}_{\Sigma} \in P(O_{\Sigma}^{\tn{perf}} \otimes_{O_{F,\Sigma}} O_{\Sigma}^{\tn{perf}} \otimes_{O_{F,\Sigma}} O_{\Sigma}^{\tn{perf}})$ gives an element 
\begin{equation*}
\xi_{\Sigma}^{\tn{ad}} \in P_{\dot{\Sigma}}(\tn{Spa}(O_{\Sigma}^{\tn{perf}}) \times_{\tn{Spa}(O_{F,\Sigma})}  \tn{Spa}(O_{\Sigma}^{\tn{perf}}) \times_{\tn{Spa}(O_{F,\Sigma})} \tn{Spa}(O_{\Sigma}^{\tn{perf}}))
\end{equation*}
which is a \v{C}ech $2$-cocycle with respect to the v-cover of adic spaces $\tn{Spa}(O_{\Sigma}^{\tn{perf}}) \to \tn{Spa}(O_{F,\Sigma})$.

\begin{Definition} Define $\tn{Kal}_{F,\Sigma}^{\tn{ad}}$ as the fibered category over $\tn{Spa}(O_{F,\Sigma})$ such that, for an adic space $S$ over $\tn{Spa}(O_{F,\Sigma})$, we have that $\tn{Kal}_{F,\Sigma}^{\tn{ad}}(S)$ is the category of $\xi^{\tn{ad}}_{\Sigma}$-twisted $P_{\dot{\Sigma}}^{\tn{ad}} \times_{\tn{Spa}(O_{F,\Sigma})} S$-torsors over $S$ in the v-topology (cf. \cite[Definition 2.35]{Dil23a}).
\end{Definition}

\begin{Lemma}
The fibered category $\tn{Kal}_{F,\Sigma}^{\tn{ad}} \to (\tn{Spa}(O_{F,\Sigma}))_{\tn{v}}$ is a $P_{\dot{\Sigma}}$-gerbe, split over $\tn{Spa}(O_{\Sigma}^{\tn{perf}})$.
\end{Lemma}

\begin{proof}
Given a commutative group sheaf $A$ on a site $\mathcal{C}$ with final object $X$ and a \v{C}ech $2$-cocycle $a \in \check{Z}^{2}(X'/X, A)$ for some covering $X' \to X$ in $\mathcal{C}$, the fibered category of $a$-twisted $A$-torsors over $X$ is always an $A$-gerbe over $\mathcal{C}$ which is trivialized over $X'$ (cf. \cite[Proposition 2.36]{Dil23a}).
\end{proof}

Given $S = \tn{Spa}(R,R^{+})$ an affinoid, perfectoid space over $\mathbb{F}_{q}$, we obtain an adic gerbe
\begin{equation*}
\mathfrak{U}_{S}:= U_{S}^{\tn{an}} \times_{\tn{Spa}(O_{F,\Sigma})} \tn{Kal}_{F,\Sigma}^{\tn{ad}} \to U_{S}^{\tn{an}},
\end{equation*}
where, as in \cite[\S 1.4-1.5]{DLH26}, $U_{S}^{\tn{an}}$ is defined as $\bigcup_{\varpi} U_{S,\varpi}$, where the union runs over all pseudouniformizers $\varpi$ of $R$ and, for a fixed $\varpi$, if $U \xrightarrow{\sim} \tn{Spec}(\mathbb{F}_{q}[T_{1}, \dots, T_{d}]/I)$, then we define $U_{S,\varpi}$ as $\tn{Spa}(B_{S,\varpi} := R\langle \varpi T_{1}, \dots \varpi T_{d} \rangle/I, B_{S,\varpi}^{+})$, where $B_{S,\varpi}^{+}$ is the integral closure of the image of $R^{+}\langle  \varpi T_{1}, \dots \varpi T_{d} \rangle$.

The gerbe $\mathfrak{U}_{S}\to U_{S}^{\tn{an}}$ is banded by $(P_{\dot{\Sigma}})_{S}$, and we have a canonical isomorphism of group sheaves on the v-site of $U_{S}^{\tn{an}}$ (cf. \cite[Lemma 2.27]{Dillery23b})
\begin{equation}\label{inertialaction}
(P_{\dot{\Sigma}})_{S} \xrightarrow{\sim} \text{Band}(\mathfrak{U}_{S}),
\end{equation}
where $\text{Band}(\mathfrak{U}_{S})$ is the sheaf on the v-site of $U_{S}^{\tn{an}}$ determined by the condition that for an object $X \in \mathfrak{U}_{S}(S')$ we have an isomorphism $\text{Band}(\mathfrak{U}_{S})(S') \xrightarrow{\sim} \text{Aut}_{S'}(X)$ compatible with morphisms in $\mathfrak{U}_{S}$ as $X$ and $S'$ vary.

Let $G$ be a parahoric group scheme over $C$, as defined in \cite[Definition 2.18]{Ric16}. Note that for a $G_{S}$-torsor $\mathscr{T}$ on $\mathfrak{U}_{S}$ there is a canonical morphism of group sheaves over $\mathfrak{U}_{S}$
\begin{equation}\label{centralhom}
(Z(G)_{S})_{\mathfrak{U}_{S}} \to  \underline{\text{Aut}}(\mathscr{T})
\end{equation}
given by the structure of $\mathscr{T}$ as a twisted $G_{S'}$-torsor, as in \cite[Definition 2.35]{Dillery23b}. There is also a morphism of group sheaves over $\mathfrak{U}_{S}$
\begin{equation}\label{intertialactionbis}
(P_{\dot{\Sigma}})_{S} \ \to  \underline{\text{Aut}}(\mathscr{T})
\end{equation}
induced by the isomorphism \eqref{inertialaction} and the action of $\text{Band}(\mathfrak{U}_{S})$ on $\mathscr{T}$ by automorphisms of objects of $\mathfrak{U}_{S}$.

\begin{Definition}\label{kalbasicdef}
We say that $\mathscr{T} \in Z_{\tn{\'{e}t}}^{1}(\mathfrak{U}_{S}, G)$ is $\text{Kal}_{F}$\textit{-basic} if the homomorphism \eqref{intertialactionbis} is obtained by pre-composing \eqref{centralhom} with a uniquely-determined morphism $(P_{\dot{\Sigma}})_{S} \to Z(G)_{S}$ over $U_{S}^{\tn{an}}$---note that this only depends on the isomorphism class of $\mathscr{T}$. Denote the set of all such isomorphism classes by $H^{1}(\mathfrak{U}_{S}, G_{S})_{\tn{Kal}_{F}\tn{-basic}}$, and the set of all such torsors by $Z^{1}(\mathfrak{U}_{S}, G)_{\tn{Kal}_{F}\tn{-basic}}$ (we omit the ``\'{e}t'' notation when writing down these $\tn{Kal}_{F}$-basic subsets to lighten the notational load).
\end{Definition}

Corollary \ref{cor:transitions} implies that for $U' \subseteq U$ we have a compatible family of transition maps
\begin{equation*}
\mathfrak{U}'_{S} \to \mathfrak{U}_{S},
\end{equation*}
which induce injective pullback maps $H^{1}_{\tn{\'{e}t}}(\mathfrak{U}_{S}, G_{S}) \to H^{1}_{\tn{\'{e}t}}(\mathfrak{U}'_{S}, G_{S})$ preserving $\Kal$-basic subsets. 

\begin{Definition}
We set $H^{1}(\mathfrak{X}_{S}, G_{S})_{\tn{$\Kal$-basic}} := \varinjlim_{U} H^{1}(\mathfrak{U}_{S}, G_{S})_{\tn{$\Kal$-basic}}$, where the limit is over all dense open $U \subseteq C$.
\end{Definition}

There is an auxiliary variant of $H^{1}_{\tn{\'{e}t}}(\mathfrak{U}_{S}, G_{S})$ which is needed for some proofs. For a purely inseparable extension $E/F$ of a finite Galois extension $E'/F$ inside $F_{\Sigma}$ and corresponding sets of places $(\Sigma, \dot{\Sigma}^{(E)})$, we can repeat the above construction for the image of $\dot{\xi}_{\Sigma}$ in $\check{Z}^{2}(O_{E,\dot{\Sigma}^{(E)}}/O_{F,\Sigma}, P_{E,\dot{\Sigma}^{(E)},n})$ for some $n \in \N$ to obtain an adic gerbe $\tn{Kal}_{F,E,\dot{\Sigma}^{(E)},n}^{\tn{ad}}$ over $\tn{Spa}(O_{F,\Sigma})$. Now we can form the adic gerbe
\begin{equation*}
U_{S}^{\tn{an}} \times_{\tn{Spa}(O_{F,\Sigma})} \tn{Kal}_{F,E,\dot{\Sigma}^{(E)},n}^{\tn{ad}} \to U_{S}^{\tn{an}},
\end{equation*}
which is banded by $(P_{E,\dot{\Sigma}^{(E)},n})_{S}$. We can thus write $H^{1}_{\tn{\'{e}t}}(U_{S}^{\tn{an}} \times_{\tn{Spa}(O_{F,\Sigma})} \tn{Kal}_{F,E,\dot{\Sigma}^{(E)},n}^{\tn{ad}}, G_{S})$  (and $Z^{1}_{\tn{\'{e}t}}(U_{S}^{\tn{an}} \times_{\tn{Spa}(O_{F,\Sigma})} \tn{Kal}_{F,E,\dot{\Sigma}^{(E)},n}^{\tn{ad}}, G_{S})$) for the parahoric group scheme $G$, and also the ``$\Kal$-basic'' variants (we don't use better notation because these objects aren't used enough to justify it). These definitions are useful in view of:

\begin{Lemma}\label{keyunionprop}
Pullback induces canonical identifications
\begin{enumerate}
\item{\begin{equation*}
H^{1}(\mathfrak{U}_{S}, G_{S})_{\tn{$\Kal$-basic}} = \varinjlim_{E,n} H^{1}(U_{S}^{\tn{an}} \times_{\tn{Spa}(O_{F,\Sigma})} \tn{Kal}_{F,E,\dot{\Sigma}^{(E)},n}^{\tn{ad}})_{\tn{$\Kal$-basic};}
\end{equation*}}
\item{\begin{equation*}
Z^{1}(\mathfrak{U}_{S}, G_{S})_{\tn{$\Kal$-basic}} = \varinjlim_{E,n} Z^{1}(U_{S}^{\tn{an}} \times_{\tn{Spa}(O_{F,\Sigma})} \tn{Kal}_{F,E,\dot{\Sigma}^{(E)},n}^{\tn{ad}})_{\tn{$\Kal$-basic}}.
\end{equation*}}
\end{enumerate}
Moreover, any isomorphism of $G$-torsors on $\mathfrak{U}_{S}$ is pulled back from an isomorphism of $G$-torsors on some $U_{S}^{\tn{an}} \times_{\tn{Spa}(O_{F,\Sigma})} \tn{Kal}_{F,E,\dot{\Sigma}^{(E)},n}^{\tn{ad}}$.
\end{Lemma}

\begin{proof}
Because the classifying stack $B(G_{S})$ is an Artin v-stack over $U_{S}^{\text{an}}$, any morphism $\mathfrak{U}_{S} \to B(G_{S})$ factors through one of the projection maps $\mathfrak{U}_{S} \to U_{S}^{\tn{an}} \times_{\tn{Spa}(O_{F,\Sigma})} \tn{Kal}_{F,E,\dot{\Sigma}^{(E)},n}^{\tn{ad}}$. Any isomorphism of $G_{S}$-torsors corresponds to a similar factorization.
\end{proof}

We can now make the main definition(s):

\begin{Definition}\label{maindefapp}
This is an analogue of the functors defined in \cite[\S 1.8]{DLH26}:
\begin{enumerate}
\item{Given $U$ a dense affine open of $C$, define a functor $\tn{Bun}_{G,U}^{e}$ from affinoid perfectoid spaces over $\mathbb{F}_{q}$ to groupoids by
\begin{equation*}
S \mapsto \{(\mathscr{T}, \phi) | \mathscr{T} \in Z^{1}(\mathfrak{U}_{S}, G_{S})_{\tn{$\Kal$-basic}}, \hspace{.2cm} \phi \colon \mathscr{T} \xrightarrow{\sim} (\tn{Frob}_{S} \times \tn{id})^{*}\mathscr{T}\};
\end{equation*}}
\item{Define a functor $\tn{Bun}_{G,F}^{e}$ from affinoid perfectoid spaces over $\mathbb{F}_{q}$ to groupoids by 
\begin{equation*}
S \mapsto \varinjlim_{U} \tn{Bun}_{G,U}^{e}(S);
\end{equation*}}
\item{Define a functor $\tn{Bun}^{e}_{G,F_{v}}$ from affinoid perfectoid spaces over $\mathbb{F}_{q}$ to groupoids by $S \mapsto$
\begin{equation*}
\{(\mathscr{T}, \phi) | \mathscr{T} \in Z^{1}((\tn{Spa}(F_{v}) \times S) \times_{\tn{Spa}(F_{v})} \tn{Kal}^{\tn{ad}}_{v}, G_{S})_{\tn{$\Kal$-basic}}, \hspace{.2cm} \phi \colon \mathscr{T} \xrightarrow{\sim} (\tn{Frob}_{S} \times \tn{id})^{*}\mathscr{T}\},
\end{equation*}
where recall that $\tn{Kal}_{v}$ is the local Kaletha gerbe at $v$ with adic analogue defined in \cite[\S 12.3]{Fargues}.\footnote{As pointed out in Remark \ref{FarguesRem}, \cite{Fargues} works with $\tn{Kal}_{v}$ over a $p$-adic local field, but the identical constructions carry over in the function field case.}
 }
 \item{Define the functor $\tn{Bun}^{e}_{G,O_{v}}$ as just $\tn{Bun}_{G,O_{v}}$; because there is no difference in this case, we exclusively use the notation  $\tn{Bun}_{G,O_{v}}$.
 }
 \end{enumerate}
\end{Definition}

\begin{Lemma}\label{smallvstacklem}
All of the above functors are small v-stacks over $\mathbb{F}_{q}$.
\end{Lemma}

\begin{proof}
For the first two functors it suffices to show this for $\tn{Bun}_{G,U}^{e}$. For this statement, we use Lemma \ref{keyunionprop} to deduce that any descent datum is pulled back from a finite level, and is therefore effective if we can show that the analogue of $\tn{Bun}_{G,U}^{e}$ with $\tn{Kal}_{F,\Sigma}$ replaced by $\tn{Kal}_{F,E,\dot{\Sigma}^{(E)},n}$, denoted $\tn{Bun}_{G,U,E,n}^{e}$, is a small v-stack over $\mathbb{F}_{q}$. First, by the argument in \cite[\S 1.8]{DLH26} we can reduce to the case of $G = \mathrm{GL}_{m}$. Then we may, by Corollary \ref{Cor:fiberfunctor}, pick a finite Galois extension $E''/F$ such that we have a presentation
\begin{equation*}
K_{1} := \mathrm{Spec}(E'') \times_{ \tn{Kal}_{F,E,\dot{\Sigma}^{(E)},n}}  \mathrm{Spec}(E'') \to  \mathrm{Spec}(E'') \to \tn{Kal}_{F,E,\dot{\Sigma}^{(E)},n},
 \end{equation*}
 where each projection $K_{1} \to \mathrm{Spec}(E'')$ is an $A$-torsor with $A/E''$ finite flat. Write $K_{1}$ as $\coprod_{i} K_{1,i}$, where $E_{i}''$ is a finite Galois extension of $E''$ corresponding to the covering $C_{E_{i}''} \to C$ of curves, and $K_{1,i} = \tn{Spec}(\tilde{E}_{i}'' \otimes_{E_{i}''} \tilde{E}_{i}'')$ for a purely-inseparable extension $\tilde{E}_{i}''/E_{i}''$. 
 
Observe that each map $U_{K_{1},i} \to U$ is finite and flat---here we write $U_{E''_{i}} = \mathrm{Spec}(B_{i})$ for the preimage in $C_{E_{i}''}$ of the fixed affine open $U$ of $C$ and then $U_{K_{1},i} := \mathrm{Spec}(B_{i}^{(1/p^{m})} \otimes_{B_{i}} B_{i}^{(1/p^{m})})$ for $m$ the degree of $\tilde{E}_{i}''/E_{i}''$---as is $U_{E''} \to U$. We now have, as in the proof of \cite[Proposition 12.5]{Fargues} a morphism
 \begin{equation*}
\tn{Bun}_{G,U,E,n}^{e} \hookrightarrow \tn{Bun}_{G_{E''}} \times_{\tn{Bun}_{G_{K_{1}}} \times \tn{Bun}_{G_{K_{1}}}, \Delta} \tn{Bun}_{G_{K_{1}}}
 \end{equation*}
which is representable by a v-sheaf. We then obtain the result from the fact that both $\tn{Bun}_{G_{E''}}$ and $\tn{Bun}_{G_{K_{1}}}$ are v-stacks, by \cite[Theorem 1.6]{DLH26}, using that $U_{K_{1}}$ and $U_{E''}$ are finite and flat over $U$, as are $(U_{E''})_{S}^{\text{an}}$ and $(U_{K_{1}})_{S}^{\text{an}}$ over $U_{S}^{\tn{an}}$, and so vector bundles over these objects yield vector bundles over $U_{S}^{\tn{an}}$. The same reasoning shows that it is a small v-stack over $\mathbb{F}_{q}$.

The argument for $\tn{Bun}^{e}_{G,F_{v}}$ is completely analogous.
\end{proof}

\subsection{The structure of $\tn{Bun}_{G,F}^{e}$}\label{subsec:Bunstruc} \hfill\\
We temporarily return to the algebro-geometric setting. Observe that there is a morphism of stacks $U/\tn{Frob}^{\Z} \to \tn{Kott}_{F,\Sigma}$ induced by the tensor functor from Drinfeld isoshtukas supported on $U$ to vector bundles on $U$ equipped with a Frobenius-semilinear isomorphism. We thus obtain a map
\begin{equation}\label{vbequation}
Z^{1}_{\tn{\'{e}t}}(\tn{Kott}_{F,\Sigma} \times_{O_{F,\Sigma}} \tn{Kal}_{F,\Sigma}, G)\to \tn{Bun}_{G,U}^{e,\tn{alg}}(\ov{\mathbb{F}_{q}})
 \end{equation}
 via pull-back along $(U/\tn{Frob}^{\Z}) \times_{U} \tn{Kal}_{F,\Sigma} \to \tn{Kott}_{F,\Sigma} \times_{U} \tn{Kal}_{F,\Sigma}$, where $\tn{Bun}_{G,U}^{e,\tn{alg}}$ denotes the algebraic analogue of Definition \ref{maindefapp}. Observe that the above map is compatible with pull-back along the inclusion $U' \subseteq U$ (replacing $\Sigma$ with $\Sigma'$ such that $U' = \mathrm{Spec}(O_{F,\Sigma'})$).
 
 \begin{Proposition}\label{AlgBij}
The above map \eqref{vbequation} is a bijection on isomorphism classes. Moreover, taking its direct limit over all $U$ induces a bijection $B_{e}(G) \to \tn{Bun}_{G,F}^{e,\tn{alg}}(\ov{\mathbb{F}_{q}})$.
 \end{Proposition}
 
 \begin{proof}
 The identical argument from the proof of \cite[Th\'{e}or\`{e}me 11.4]{Fargues} applies here.
 \end{proof}

The above algebraic result lets us define substacks of the analytic object constructed above.
\begin{Definition}
\begin{enumerate}
\item{Define $\tn{Bun}^{e,1}_{G,F}$ to be the substack characterized by 
\begin{equation*}
\tn{Bun}^{e,1}_{G,F}(S) = \{x \in \tn{Bun}^{e}_{G,F}(S) | \bar{x} = \tn{triv} \in  \tn{Bun}_{G,F}(\ov{s})\},
\end{equation*}
where $\bar{s}$ runs over all geometric points of $S$ and $\bar{x}$ denotes the image of $x$ at $\bar{s}$.}
\item{Given $b$ an element of $B_{e}(G)_{\tn{basic}}$, define $\tn{Bun}^{e,b}_{G,F}$  as the substack characterized by $\tn{Bun}^{e,b}_{G,F}(S) = \{x \in \tn{Bun}^{e}_{G,F}(S) | \bar{x} = b \in  \tn{Bun}_{G,F}(\ov{s})\}$, where $\bar{s}$ runs over all geometric points of $S$ and $\bar{x}$ denotes the image of $x$ at $\bar{s}$.} 
\end{enumerate}
\end{Definition}

The next result comes from combining the above definition with the proof of \cite[Proposition 2.9]{DLH26}:
\begin{Proposition}\label{prop:bundletwist}
For $b \in B_{e}(G)_{\tn{basic}}$ pulled back from $\tn{Kott}_{F,\Sigma} \times_{O_{F,\Sigma}} \tn{Kal}_{F,\Sigma}$, there is a natural isomorphism $\tn{Bun}^{e}_{G,U} \xrightarrow{\sim} \tn{Bun}^{e}_{G_{b},U}$ sending $\tn{Bun}^{e,b}_{G,U}$ to $\tn{Bun}^{e,1}_{G_{b},U}$.
\end{Proposition}

For an object $x \in \tn{Bun}_{G,U}^{e}(S)$ we have by Definition \ref{kalbasicdef} a $\tn{Frob}_{S}$-equivariant morphism $(P_{\dot{\Sigma}})_{S} \to Z(G)_{S}$ over $U_{S}^{\tn{an}}$, which we will denote by $\lambda_{x}$. We now introduce a substack which is similar to $\tn{Bun}_{G,e}^{1}$:
\begin{Definition}
\begin{enumerate}
\item{Define $\tn{Bun}^{e,0_{P}}_{G,U}$ to be the substack characterized by 
\begin{equation*}
\tn{Bun}^{e,0_{P}}_{G,U}(S) = \{x \in \tn{Bun}^{e}_{G,U}(S) | \lambda_{\bar{x}} = 0 \in \Hom_{U_{\bar{s}}^{\tn{an}}}((P_{\dot{\Sigma}})_{\bar{s}}, Z(G)_{\bar{s}})\}
\end{equation*}
where $\bar{s}$ runs over all geometric points of $S$ and $\bar{x}$ denotes the image of $x$ at $\bar{s}$.}
\item{Observe that we have a canonical isomorphism
\begin{equation*}
\Hom_{U}(P_{\dot{\Sigma}}, Z(G)) \xrightarrow{\sim} \Hom_{U_{\bar{s}}^{\tn{an}}}((P_{\dot{\Sigma}})_{\bar{s}}, Z(G)_{\bar{s}})^{\tn{Frob}_{\ov{s}}}.
\end{equation*}
In view of this, given $\lambda \in \Hom(P_{\dot{\Sigma}}, Z(G))$, define $\tn{Bun}^{e,\lambda}_{G,U}$  as the substack characterized by 
\begin{equation*}
\tn{Bun}^{e,\lambda}_{G,U}(S) = \{x \in \tn{Bun}^{e}_{G,U}(S) | \lambda_{\bar{x}} = \lambda \in  \Hom_{U}(P_{\dot{\Sigma}}, Z(G))\},
\end{equation*}
where $\bar{s}$ runs over all geometric points of $S$ and $\bar{x}$ denotes the image of $x$ at $\bar{s}$.}
\end{enumerate}
\end{Definition}

A key structural result is:

\begin{Proposition}\label{Nclopenprop} 
\begin{enumerate}
\item{The substack $\tn{Bun}^{e,0_{P}}_{G,U} \hookrightarrow \tn{Bun}_{G,U}^{e}$ is open and closed.}
\item{The substack $\tn{Bun}^{e,\lambda}_{G,U} \hookrightarrow \tn{Bun}_{G,U}^{e}$ is open and closed for all $\lambda \in \Hom(P_{\dot{\Sigma}}, Z(G))$.}
\end{enumerate}
\end{Proposition}

\begin{proof}
Lemma \ref{keyunionprop} allows us to replace $P_{\dot{\Sigma}}$ with $A:= P_{E,\dot{\Sigma}^{(E)}, n}$ for some $E$ and $n$. Then the inertial action gives a canonical morphism of stacks
\begin{equation}
\tn{Bun}_{G,U}^{e} \xrightarrow{\Lambda} \underline{\tn{Hom}}_{U}(A, Z(G)),
\end{equation}
where $\underline{\tn{Hom}}_{U}(A, Z(G))$ is the functor sending $S$ to $\tn{Hom}_{U_{S}^{\tn{an}}}(A_{S}, Z(G_{S}))$. It has the closed subfunctor $S \mapsto 0 \in \tn{Hom}_{U_{S}^{\tn{an}}}(A_{S}, Z(G_{S}))$. We claim that a morphism $f \in \tn{Hom}_{U_{S}^{\tn{an}}}(A_{S}, Z(G_{S}))$ is zero in $\Hom_{U_{\bar{s}}^{\tn{an}}}(A_{\bar{s}}, Z(G)_{\bar{s}})$ for every geometric point $\bar{s}$ of $S$ if and only if it is zero. Indeed, since $U^{\tn{an}}_{\bar{s}}$ is connected, every connected component $V$ of $U_{S}^{\tn{an}}$ contains some $U^{\tn{an}}_{\bar{s}}$, which forces $f|_{V} = 0$, by the diagonalizability of the two group schemes. This claim shows that $\tn{Bun}^{e,0_{P}}_{G,U}$ is precisely the preimage of this closed substack under $\Lambda$.

We claim that, in fact, the above $0$-subfunctor is open as well. This may be checked after replacing $U$ by a finite \'{e}tale cover, which lets assume that $A$ is a product of $\mu_{n}$'s and $Z(G)$ is a product of $\mu_{n}$'s and $\mathbb{G}_{m}$'s, and thus that $\underline{\tn{Hom}}_{U}(A, Z(G))$ is a finite product of the functors sending $S$ to $\underline{\Z/n\Z}(U_{S}^{\tn{an}})$ for some $n$, for which it suffices to show that the fiber over the trivial element in $\underline{\Z/n\Z}(U_{S}^{\tn{an}})$ in $|S|$ is open. This map is constant on connected components of $|U_{S}^{\tn{an}}|$, and since $\pi_{0}(|U_{S}^{\tn{an}}|) = \pi_{0}(|S|)$ it is also constant on connected components of $|S|$ by \cite[Tag 0378]{stacks-project}, as desired. 

To obtain the second statement from the first, we begin by lifting the element $\lambda \in \underline{\tn{Hom}}_{U}(P_{\dot{\Sigma}}, Z(G))$ to $b \in H^{1}(\text{Kott}_{F,\Sigma} \times \text{Kal}_{F,\Sigma})_{\text{basic}}$ (as we can by Proposition \ref{Sigmavan}). Then Proposition \ref{prop:bundletwist} gives an isomorphism $\tn{Bun}_{G,U}^{e} \xrightarrow{\sim} \tn{Bun}_{G_{b},U}^{e}$ and identifying $\tn{Bun}_{G,U}^{e,\lambda}$ with $\tn{Bun}_{G_{b},U}^{e,0_{P}}$, giving the claim.

\end{proof}

Observe that the natural map $\tn{Bun}_{G,U}^{e} \xrightarrow{\Lambda} \underline{\tn{Hom}}_{U}(P_{\dot{\Sigma}}, Z(G))$ from the previous proof is compatible with pullback to smaller $U'$, thus giving a map $\tn{Bun}_{G,F}^{e} \xrightarrow{\Lambda} \underline{\tn{Hom}}_{F}(P, Z(G))$. The most important result for understanding the geometry of $\tn{Bun}_{G,U}^{e}$ and $\tn{Bun}_{G,F}^{e}$ is:

\begin{Corollary}\label{prop:newtondecomp}
\begin{enumerate}
\item{For any $b \in H^{1}_{\tn{\'{e}t}}(\tn{Kott}_{F,\Sigma} \times \tn{Kal}_{F,\Sigma}, G)$ lifting $\lambda$, we have an identification of stacks
\begin{equation*}
 \tn{Bun}_{G,U}^{e,\lambda} \xrightarrow{\sim} \tn{Bun}_{G_{b},U}.
\end{equation*}}
\item{The map $\Lambda$ induces a decomposition (as a $v$-stack)
\begin{equation*}
\tn{Bun}_{G,U}^{e} \xrightarrow{\sim} \bigsqcup_{\lambda \in \tn{Hom}_{U}(P_{\dot{\Sigma}}, Z(G))}  \tn{Bun}_{G,U}^{e,\lambda}.
\end{equation*}}
\end{enumerate}
The above two results also hold for $\tn{Bun}_{G,F}^{e}$.
\end{Corollary}

\noindent In particular, the above result in conjunction with \cite[Proposition 4.5]{DLH26} shows that each $\tn{Bun}_{G,U}^{e}$ is an Artin v-stack over $\mathbb{F}_{q}$. We now have the analogue of \cite[Theorem 2.5]{DLH26}, which is \cite[Theorem C.8.(1)]{DLH26}: 

\begin{Proposition}
The substack $\tn{Bun}^{e,1}_{G,U} \subseteq \tn{Bun}^{e}_{G,U}$ is open, and is isomorphic to $*/\underline{G(O_{F,\Sigma})}$. 
\end{Proposition}

We can also deduce the following analogue of \cite[Theorem 2.11]{DLH26}, which identifies the set $\tn{Bun}_{G,U}^{e}(\tn{Spd}(\ov{ \mathbb{F}_{q}}))$ with $B_{e}(G)$. Since it is a straightforward consequence of the result loc. cit., it is proved in \cite[Theorem C.6]{DLH26}:

\begin{Corollary}\label{cor:spd}
  The perfect v-stack $\tn{Bun}_{G,U}^{e,\tn{red}}$ over $\mathbb{F}_q$ is the v-sheafification of
  \begin{align*}
 \Spec{B}\mapsto\left\{
    \begin{tabular}{c}
      $G$-torsors $\mathscr{T}$ on $U_B \times_{U} \tn{Kal}_{F,\Sigma}$ equipped with\\
      an isomorphism $\phi:\mathscr{T} \xrightarrow{\sim} (\tn{Frob}_B \times \tn{id})^*\mathscr{T}$
    \end{tabular}
    \right\}.
  \end{align*}
Moreover, when each connected component of $\Spec{B}$ is a valuation ring, no sheafification is needed.
\end{Corollary}

Define the functor $\tn{Bun}^{e}_{G,\mathbb{A}}$ by 
\begin{equation*}
S \mapsto \varinjlim_{\Sigma} \prod_{v \notin \Sigma} \tn{Bun}_{G,O_{v}}(S) \times \prod_{v \in \Sigma} \tn{Bun}^{e}_{G,F_{v}}(S),
\end{equation*}
where the limit is over all finite subsets of places $\Sigma$. This defines a small v-stack over $\mathbb{F}_{q}$, and we have a morphism
\begin{equation*}
\tn{Bun}^{e}_{G,F} \to \tn{Bun}^{e}_{G,\mathbb{A}}
\end{equation*}
induced by a localization map $\tn{Kal}_{v} \to \tn{Kal}_{F}$ at each place as discussed in \S \ref{Kalgerbedef}.

We can now define a variant of part (3) of Definition \ref{maindefapp} that will be used in the following. Define the functor $\tn{Bun}^{e,(\Sigma)}_{G,F_{v}}$ from affinoid perfectoid spaces over $\mathbb{F}_{q}$ to groupoids by $S \mapsto$
\begin{equation*}
\{(\mathscr{T}, \phi) | \mathscr{T} \in Z^{1}((\tn{Spa}(F_{v}) \times S) \times_{\tn{Spa}(F_{v})} (\tn{Kal}^{(\Sigma)}_{v})^{\tn{ad}}, G_{S})_{\tn{$\Kal$-basic}}, \hspace{.2cm} \phi \colon \mathscr{T} \xrightarrow{\sim} (\tn{Frob}_{S} \times \tn{id})^{*}\mathscr{T}\},
\end{equation*}
which is also a small v-stack over $\mathbb{F}_{q}$. Note that there is a canonical inclusion $\tn{Bun}^{e,(\Sigma)}_{G,F_{v}} \to \tn{Bun}^{e}_{G,F_{v}}$.

\begin{Proposition}
There is a morphism 
\begin{equation*}
\tn{Bun}^{e}_{G,U} \to \prod_{v \notin \Sigma} \tn{Bun}_{G,O_{v}} \times \prod_{v \in \Sigma} \tn{Bun}^{e,(\Sigma)}_{G,F_{v}} \subseteq \prod_{v \notin \Sigma} \tn{Bun}_{G,O_{v}} \times \prod_{v \in \Sigma} \tn{Bun}^{e}_{G,F_{v}}
\end{equation*}
which makes the square
\[
\begin{tikzcd}
\tn{Bun}^{e}_{G,U} \arrow{r} \arrow{d} &  \prod_{v \notin \Sigma} \tn{Bun}_{G,O_{v}} \times \prod_{v \in \Sigma} \tn{Bun}^{e,(\Sigma)}_{G,F_{v}}  \arrow{d} \\
\tn{Bun}^{e}_{G,F} \arrow{r} &  \tn{Bun}^{e}_{G,\mathbb{A}}
\end{tikzcd}
\]
Cartesian.
\end{Proposition}

\begin{proof}
As in the proof of \cite[Lemma 2.7]{DLH26}, it suffices to prove the result with $\tn{Spec}(F)$ replaced by an affine open $U' \subseteq U$. The argument proceeds by constructing an intermediate cartesian square. Namely, for $A_{v} = F_{v}$ or $O_{v}$ and a fixed $U = C \setminus |\Sigma|$, define the functor $\widetilde{\tn{Bun}}^{e,(\Sigma)}_{G,A_{v}}$ from affinoid perfectoid spaces over $\mathbb{F}_{q}$ to groupoids by $S \mapsto$
\begin{equation*}
\{(\mathscr{T}, \phi) | \mathscr{T} \in Z^{1}((\tn{Spa}(A_{v}) \times S) \times_{\tn{Spa}(A_{v})} \tn{Kal}_{F,\Sigma}^{\tn{ad}}, G_{S})_{\tn{$\Kal$-basic}}, \hspace{.2cm} \phi \colon \mathscr{T} \xrightarrow{\sim} (\tn{Frob}_{S} \times \tn{id})^{*}\mathscr{T}\},
\end{equation*}
and the functor $\widetilde{\tn{Bun}}^{e,(\Sigma)}_{G,\mathbb{A}}(S) :=  \prod_{v \notin \Sigma} \widetilde{\tn{Bun}}^{e,(\Sigma)}_{G,O_{v}}(S) \times \prod_{v \in \Sigma} \widetilde{\tn{Bun}}^{e,(\Sigma)}_{G,F_{v}}(S)$. The proof of \cite[Lemma 2.7]{DLH26} shows that we have a Cartesian diagram
\[
\begin{tikzcd}
\tn{Bun}^{e}_{G,U} \arrow{r} \arrow{d} &  \prod_{v \notin \Sigma} \widetilde{\tn{Bun}}^{e,(\Sigma)}_{G,O_{v}} \times \prod_{v \in \Sigma} \widetilde{\tn{Bun}}^{e,(\Sigma)}_{G,F_{v}}  \arrow{d} \\
\tn{Bun}^{e}_{G,U'} \arrow{r} & \prod_{v \notin \Sigma'} \widetilde{\tn{Bun}}^{e,(\Sigma')}_{G,O_{v}} \times \prod_{v \in \Sigma'} \widetilde{\tn{Bun}}^{e,(\Sigma')}_{G,F_{v}},
\end{tikzcd}
\]
and then we deduce the result from the fact that the diagram
\[
\begin{tikzcd}
\prod_{v \notin \Sigma} \widetilde{\tn{Bun}}^{e,(\Sigma)}_{G,O_{v}} \times \prod_{v \in \Sigma} \widetilde{\tn{Bun}}^{e,(\Sigma)}_{G,F_{v}}  \arrow{r} \arrow{d} & \prod_{v \notin \Sigma} \tn{Bun}_{G,O_{v}} \times \prod_{v \in \Sigma} \tn{Bun}^{e,(\Sigma)}_{G,F_{v}}  \arrow{d} \\
\prod_{v \notin \Sigma'} \widetilde{\tn{Bun}}^{e,(\Sigma')}_{G,O_{v}} \times \prod_{v \in \Sigma'} \widetilde{\tn{Bun}}^{e,(\Sigma')}_{G,F_{v}} \arrow{r} &  \prod_{v \notin \Sigma'} \tn{Bun}_{G,O_{v}} \times \prod_{v \in \Sigma'} \tn{Bun}^{e,(\Sigma')}_{G,F_{v}},
\end{tikzcd}
\]
is Cartesian, where each of the maps $\widetilde{\tn{Bun}}^{e,(\Sigma)}_{G,A_{v}} \to \tn{Bun}^{e,(\Sigma)}_{G,F_{v}}$ or $\tn{Bun}_{G,O_{v}}$ is obtained by pulling back along the morphism of gerbes $\tn{Kal}^{(\Sigma)}_{v} \xrightarrow{\tn{loc}_{v}} \tn{Kal}_{F,\Sigma}$.
\end{proof}

We can also define an analogue of the semi-stable locus:

\begin{Definition}
Set $\tn{Bun}_{G,F}^{e,\tn{ss}}$ to be the preimage of 
\begin{equation*}
\varinjlim \prod_{v \notin \Sigma} \tn{Bun}_{G,O_{v}} \times \prod_{v \in \Sigma} \tn{Bun}^{e,\tn{ss}}_{G,F_{v}}
\end{equation*}
in $\tn{Bun}_{F,G}^{e}$, where $\tn{Bun}^{e,\tn{ss}}_{G,F_{v}}$ is the local semistable locus as defined in \cite[\S 11]{Fargues}. This does not depend on the choice of localization maps, is a small v-stack over $\mathbb{F}_{q}$, and is an open substack of $\tn{Bun}_{G,F}^{e,\tn{ss}}$.
\end{Definition}

By combining \cite[Theorem 4.11]{DLH26} with Corollary \ref{prop:newtondecomp} we can deduce the following result, which is \cite[Theorem C.8]{DLH26}:

\begin{Proposition}\label{prop:ss}
For $b \in B_{e}(G)_{\tn{basic}}$, the substack $\tn{Bun}_{G,F}^{e,b} \subseteq \tn{Bun}_{G,F}^{e}$ is open and isomorphic to $\ast/\underline{G_{b}(F)}$. Moreover, $\tn{Bun}_{G,F}^{e,\tn{ss}} = \bigsqcup_{b \in B_{e}(G)_{\tn{basic}}} \tn{Bun}_{G,F}^{e,b}$.
\end{Proposition}

\section{A global multiplicity formula}\label{sec:mult}
We conclude this article by giving a conjectural formula for the multiplicity of an automorphic representation of $G(\A_{F})$ in the discrete automorphic spectrum using the (local and global) set $B_{e}(G)_{\tn{basic}}$. As before, $G$ is a connected reductive group over a global function field $F$.

\subsection{The $B_{e}(G)_{\text{basic}}$-local Langlands correspondence} \hfill\\
One first needs a version of the refined local correspondence in which the underlying inner forms are parametrized by the cohomology set $B_{e}(G_{F_{v}})$ for each place $v$. Recall from \cite{Fargues} that the set $B_{e}(G_{F_{v}})$ is defined as those elements of $Z_{\tn{\'{e}t}}^{1}(\tn{Kott}^{1/\infty}_{v}, G)$ whose restriction to $u_{v}$ factors through $Z(G)$. This parametrization is inspired by an earlier version of \cite{Fargues}. We set $\Gamma_{v} := \Gamma_{F_{v}^{s}/F_{v}}$.

Let $\phi_{v} \colon W_{F_{v}} \times \mathrm{SL}_{2}(\mathbb{C}) \to \prescript{L}{}G_{v}$ be an $L$-parameter. Set $S_{\phi_{v}} = Z_{\widehat{G}}(\phi_{v})$, $S_{\phi_{v}}^{+}$ its preimage in $\widehat{\bar{G}}$, where
\begin{equation*}
\widehat{\bar{G}} := \widehat{G}_{\tn{sc}} \times \varprojlim_{z \mapsto z^{n},n \in \mathbb{N}} Z(\widehat{G})^{\circ}.
\end{equation*}
The group that parametrizes the fibers of the refined $B_{e}(G_{F_{v}})_{\tn{basic}}$-local correspondence will be
\begin{equation*}
S_{\phi_{v}}^{+,\natural} := \frac{S_{\phi_{v}}^{+}}{(S_{\phi_{v}}^{+} \cap \widehat{G}_{\text{sc}})^{\circ}}.
\end{equation*}

The version of this refined local Langlands correspondence is as follows. First, we need to define the notion of a representation of a $B_{e}(G_{F_{v}})_{\tn{basic}}$-enhanced inner form of $G$, following \cite[\S 7.1]{Dil23a}. A $B_{e}(G_{F_{v}})_{\tn{basic}}$-enhanced inner form of $G$ is a tuple $(G', \psi,\mathscr{T},h)$ consisting of an inner form $(G',\psi)$ of $G$, a basic $G$-torsor $\mathscr{T}$ on $\tn{Kott}_{v}^{1/\infty}$, and an isomorphism $h$ from the $G_{\tn{ad}}$-torsor $\bar{\mathscr{T}} := \mathscr{T} \times^{G} G_{\tn{ad}}$ to the $G_{\tn{ad}}$ torsor given by the fiber over $\psi \in (\underline{\tn{Isom}}(G,G')/G_{\tn{ad}})(F)$ in $\underline{\tn{Isom}}(G,G')$. We say that two such enhanced inner forms are isomorphic if there is an isomorphism between the two $G$-torsors on $\tn{Kott}_{v}^{1/\infty}$---such an isomorphism induces an $F$-rational isomorphism between the underlying inner forms. A representation of $(G', \psi,z,h)$ is a tuple $(G', \psi,z,h, \pi)$ where $\pi$ is a representation of $G'(F)$, and two representations are isomorphic if the underlying inner twists are isomorphic and the corresponding $F$-rational isomorphism of groups identifies the two underlying representations.

\begin{Conjecture}\label{LLCconj1}
Let $\phi_{v}$ be an $L$-parameter for $G^{*}$ a connected, reductive quasi-split group with a choice of Whittaker datum $\mathfrak{w}_{v}$. Then, if we denote by $\Pi_{\phi_{v}}^{\tn{Kott}_{v}^{1/\infty}}$ the set of all isomorphism classes of representations of $B_{e}(G^{*}_{F_{v}})_{\tn{basic}}$-enhanced inner forms of $G^{*}$ whose image under the local Langlands correspondence is $\phi_{v}$, there is a commutative diagram, where both horizontal maps are bijections
\[
\begin{tikzcd}
    \Pi_{\phi_{v}}^{\tn{Kott}_{v}^{1/\infty}} \arrow{d} \arrow["\iota_{\phi_{v},\mathfrak{w}_{v}}"]{r} & \tn{Irr}(S_{\phi_{v}}^{+,\natural}) \arrow{d} \\
    B_{e}(G^{*}_{F_{v}})_{\tn{basic}} \arrow{r} & X^{*}(Z(\widehat{G^{*}})^{\Gamma_{v},+}).
    \end{tikzcd}
\]
The bottom map is the composition of the map from the local analogue \cite[Th\'{e}or\`{e}me 9.12]{Fargues} of Theorem \ref{TNgp}, with the isomorphism $\pi_{1}(G^{*})^{e}_{\Gamma_{v}} \xrightarrow{\sim} X^{*}(Z(\widehat{G^{*}})^{\Gamma_{v},+}$ from \cite[Proposition 12.1]{Fargues}.

Moreover, the bijection $\iota_{\phi_{v},\mathfrak{w}_{v}}$ recovers the bijections
\begin{equation*}
\Pi_{\phi_{v}}^{\tn{Kal}_{v}} \to \tn{Irr}(\pi_{0}(S_{\phi_{v}}^{+})), \hspace{1mm} \Pi_{\phi_{v}}^{\tn{Kott}_{v}} \to \tn{Irr}(S_{\phi_{v}}^{\natural})
\end{equation*}
(recall that $S_{\phi_{v}}^{\natural} := S_{\phi_{v}}/(S_{\phi_{v}} \cap \widehat{G}_{\tn{der}})^{\circ}$) from \cite[Conjecture 6.13.(3)]{Taibi22} and \cite[Conjecture 6.13.(2)]{Taibi22}, respectively, via the maps on irreducible representations induced by the surjective homomorphisms
\begin{equation*}
S_{\phi_{v}}^{+,\natural} \to \pi_{0}(S_{\phi_{v}}^{+}), \hspace{1mm} S_{\phi_{v}}^{+,\natural} \to S_{\phi_{v}}^{\natural},
\end{equation*} 
respectively.

\end{Conjecture}
One can formulate a version of the endoscopic character identities (as in \cite[\S 7.3]{Dil23a}) but we do not do so here.

\subsection{The global Langlands conjectures} \hfill\\
We conclude by defining a multiplicity formula and making some conjectures in the global setting. Let $G^{*}$ be a quasi-split connected reductive group over $F$ with Whittaker datum $\mathfrak{w}$, and fix a discrete $L$-homomorphism $\phi \colon L_{F} \to \prescript{L}{}G^{*}$ with bounded image, where $L_{F}$ is the hypothetical Langlands group of $F$ (as in \cite{Art02}). For each $v \in V_{F}$ we obtain a local parameter $\phi_{v} \colon L_{F_{v}} \to \prescript{L}{}G^{*}$. Conjecture \ref{LLCconj1} gives an $L$-packet $\Pi_{\phi_{v}}$ of tempered representations of $B_{e}(G_{F_{v}})_{\tn{basic}}$-inner twists of $G^{*}_{F_{v}}$ together with a parametrization of this packet in terms of $\phi_{v}$. 

Let $(G,\psi)$ be an inner form of $G^{*}$, which by the same argument as \cite[Lemma 4.1]{Dillery23b} we may enrich to a $B_{e}(G^{*})_{\tn{basic}}$-enhanced inner form $(G,\psi, \mathscr{T},h)$, where $\mathscr{T} \in Z^{1}(\tn{Kott}_{F}^{1/\infty}, G^{*})_{\tn{basic}}$ is the image of $\mathscr{T}_{\tn{sc}} \in Z^{1}(\tn{Kott}_{F}^{1/\infty}, G^{*}_{\tn{sc}})_{\tn{basic}}$. We observe that the canonical inclusion map 
\begin{equation}\label{sceq}
Z^{1}(\tn{Kal}_{F}, G^{*}_{\tn{sc}})_{\tn{basic}}  \hookrightarrow Z^{1}(\tn{Kott}_{F}^{1/\infty}, G^{*}_{\tn{sc}})_{\tn{basic}}
\end{equation}
is an equality, because, viewing $\tn{Kott}_{F}^{1/\infty}$ as a $\mathbb{D}$-gerbe over $\tn{Kal}_{F}$, we have $\tn{Hom}_{\tn{Kal}_{F}}(\mathbb{D}, Z(G^{*}_{\tn{sc}})) = \tn{Hom}_{F}(\mathbb{D}, Z(G^{*}_{\tn{sc}})) = 0$ since $Z(G^{*}_{\tn{sc}})^{\circ}$ is trivial and $\mathbb{D}$ is connected.

Pulling back along the localization map
\begin{equation*}
\tn{Kott}_{v}^{1/\infty} \to \tn{Kott}_{F}^{1/\infty}
\end{equation*}
defined as the product of the localization maps for the two gerbes $\tn{Kal}_{F}$ and $\tn{Kott}_{F}$ gives a family of $B_{e}(G^{*}_{F_{v}})_{\tn{basic}}$-enhanced inner forms $\{(G_{F_{v}},\psi_{v}, \mathscr{T}_{v},h_{v})\}_{v \in V_{F}}$. 

One has an exact sequence of $L_{F}$-modules (acting via $\tn{ad} \circ \phi$) 
\begin{equation*}
1 \to Z(\widehat{G^{*}}) \to \widehat{G^{*}} \to  (\widehat{G^{*}})_{\tn{ad}} \to 1;
\end{equation*} 
Kottwitz \cite[\S 10.2]{Kott84} defines 
\begin{equation*}
S^{\tn{ad}}_{\phi} := \text{ker}[Z_{(\widehat{G^{*}})_{\tn{ad}}}(\phi) \to H^{1}(L_{F}, Z(\widehat{G^{*}})) \to \prod_{v \in V_{F}} H^{1}(L_{F_{v}}, Z(\widehat{G^{*}}))]
\end{equation*}
and sets $\mathcal{S}_{\phi} := \pi_{0}(S^{\tn{ad}}_{\phi})$. The goal is to define an $L$-packet $\Pi_{\phi}$ which is a global analogue of the local packets appearing in Conjecture \ref{LLCconj1} along with pairing 
\begin{equation*}
\mathcal{S}_{\phi} \times \Pi_{\phi} \to \mathbb{C}
\end{equation*}
such that, for a fixed discrete automorphic representation $\pi$ of $G(\mathbb{A}_{F})$, we can combine this pairing with Conjecture \ref{LLCconj1} to write down a conjectural formula for the multiplicity of $\pi$ in the discrete automorphic spectrum of $G$ (we will write the formula explicitly below). 

We define the global $L$-packet 
\begin{equation}\label{eq:globalpi}
\Pi_{\phi} := \{ \pi = \otimes_{v}' \pi_{v} \mid (G_{F_{v}}, \psi, \mathscr{T}_{v}, \bar{h}_{v}, \pi_{v}) \in \Pi_{\phi_{v}}, \iota_{\phi_{v},\mathfrak{w}_{v}}((G_{F_{v}}, \psi, (\mathscr{T}_{v}, \bar{h}_{v}), \pi_{v})) = 1\text{ for a.a. } v\},
\end{equation}
and have the following result thanks to combining \cite[Lemma 5.8]{Dillery23b} with the equality of the inclusion \eqref{sceq} (and Conjecture \ref{LLCconj1}):

\begin{Lemma}
The set $\Pi_{\phi}$ consists of irreducible, admissible, and tempered representations of $G(\mathbb{A}_{F})$.
\end{Lemma}

We conjecture a strengthening of the above Lemma, which is not needed for our purposes:

\begin{Conjecture}
Let $\mathscr{T} \in Z^{1}(\tn{Kott}_{F}^{1/\infty}, G)_{\tn{basic}}$ be an enrichment of an inner form $(G, \psi)$ of $G^{*}$, with associated family $\{(G_{F_{v}},\psi_{v}, \mathscr{T}_{v},h_{v})\}_{v \in V_{F}}$ of $B_{e}(G^{*}_{F_{v}})_{\tn{basic}}$-enhanced inner forms. Then if we define $\Pi_{\phi}$ as in \eqref{eq:globalpi}, it always consists of irreducible, admissible, and tempered representations of $G(\mathbb{A}_{F})$.
\end{Conjecture}

\begin{Remark}
As in the proof of \cite[Lemma 5.8]{Dillery23b}, the only non-immediate statement is that each $\pi \in \Pi_{\phi}$ is unramified at almost all places. This would follow if we knew that every $\mathscr{T} \in Z^{1}(\tn{Kott}^{1/\infty}_{F}, G)_{\tn{basic}}$ has localization $\mathscr{T}_{v}$ which descends to an element of $Z^{1}_{\tn{\'{e}t}}(O_{F_{v}}, \mathscr{G}_{F_{v}})$ for almost all places of $v$, where $\mathscr{G}$ is an integral model of $G$ defined away from a finite set of places. This unramified localization result is known for both $\tn{Kott}_{F}$ by \cite[Lemma 14.2]{Kot14} and $\tn{Kal}_{F}$ by \cite[Proposition 4.14]{Dillery23b} and uses the explicit construction of the canonical classes defining both gerbes. In particular, this holds when $\mathscr{T}$ is obtained by pulling back a torsor on $\tn{Kal}_{F}$ or $\tn{Kott}_{F}$ via the projection maps.
\end{Remark}

We now define the pairing. For a given $s_{\tn{ad}} \in S^{\tn{ad}}_{\phi}$, lift it to an element $s_{\tn{sc}} \in S^{\tn{sc}}_{\phi}$, the preimage of $S_{\phi}^{\tn{ad}}$ in $(\widehat{G^{*}})_{\tn{sc}}$. We can then find, for each $v \in V_{F}$, elements $\dot{y}_{v}' \in Z((\widehat{G^{*}})_{\tn{sc}})$ and $\dot{y}_{v} \in Z(\widehat{G^{*}})^{\circ}$ such that $\dot{s}_{v} := s_\tn{sc} \cdot \dot{y}_{v}' \cdot y_{v}''' \in S_{\phi_{v}}^{+}$. Denote by $\mathfrak{w}_{v}$ the localization of the global datum $\mathfrak{w}$.

The following result is immediate from \cite[Proposition 5.9]{Dillery23b}, the compatibility of Tate-Nakayama duality from Corollary \ref{TNcompat} (which was proved for tori, but immediately gives compatibility for general reductive $G$ by these isomorphisms' functorial properties), and the $\tn{Kal}_{v}$-compatibility part of Conjecture \ref{LLCconj1}:
\begin{Proposition}
\begin{enumerate}
\item{For a fixed $\pi \in \Pi_{\phi}$ and $s_{\tn{ad}} \in S_{\phi}^{\tn{ad}}$, the value 
\begin{equation*}
\langle \mathscr{T}_{\tn{sc}}, \dot{y}_{v}' \rangle^{-1} \cdot \tn{tr}[\iota_{\phi_{v}, \mathfrak{w}_{v}}((G_{F_{v}}, \psi, \mathscr{T}_{v}, \bar{h}_{v}, \pi_{v}))](\dot{s}_{v})
\end{equation*}
is $1$ for all but finitely many $v \in V_{F}$.}
\item{The function 
\begin{equation*}
s_{\tn{ad}} \mapsto \langle s_{\tn{ad}}, \pi \rangle := \prod_{v \in V_{F}} \langle \mathscr{T}_{\tn{sc}}, \dot{y}_{v}' \rangle^{-1} \cdot \tn{tr}[\iota_{\phi_{v}, \mathfrak{w}_{v}}((G_{F_{v}}, \psi, \mathscr{T}_{v}, \bar{h}_{v}, \pi_{v}))](\dot{s}_{v})
\end{equation*}
descends to a character of $\mathcal{S}_{\phi}$ and is the character of a finite-dimensional representation which does not depend on the choice of $\mathscr{T}_{\tn{sc}}$, the choice of $s_{\tn{sc}}$, $\dot{y}_{v}'$, $y_{v}''$, nor the Whittaker datum $\mathfrak{w}$.}
\end{enumerate}
\end{Proposition}

The multiplicity conjecture is then:

\begin{Conjecture}\label{multform}
For $\pi$ an irreducible tempered discrete automorphic representation of $G(\A_{F})$, its multiplicity in the discrete automorphic spectrum of $G$ is given by the integer
\begin{equation*}
\sum_{\phi} |\mathcal{S}_{\phi}|^{-1} \sum_{s_{\tn{ad}} \in \mathcal{S}_{\phi}} \langle s_{\tn{ad}}, \pi \rangle,
\end{equation*}
where the outer sum runs over all discrete parameters $\phi$ with bounded image such that $\pi \in \Pi_{\phi}$, where $\Pi_{\phi}$ is as defined above.
\end{Conjecture}

We observe that if one assumes Conjecture \ref{LLCconj1}, including the $\tn{Kal}_{v}$-refined local Langlands compatibility part (one must assume the first part of this conjecture in any case to define the pairing $\langle - , - \rangle$), then Conjecture \ref{multform} is evidently equivalent to \cite[Conjecture 1.2]{Dillery23b}, its analogue using the gerbe $\tn{Kal}_{F}$.

\bibliography{GlobalKF.bib}

@unpublished{Fargues,
  TITLE = {{Sur la gerbe de Kaletha, la courbe et l'ensemble de Kottwitz}},
  AUTHOR = {Fargues, Laurent},
  URL = {https://hal.science/hal-03608672},
  NOTE = {https://hal.science/hal-04234891/document},
  YEAR = {2022},
  MONTH = Mar,
  HAL_ID = {hal-03608672},
  HAL_VERSION = {v1},
}

@misc{stacks-project,
  author       = {The {Stacks project authors}},
  title        = {The {S}tacks project},
  howpublished = {\url{https://stacks.math.columbia.edu}},
  year         = {2025},
}

@misc{DK25,
  author       = {Dillery, Peter and Kedlaya, Kiran},
  title        = {Vector bundles on cones over a {F}argues–{F}ontaine curve},
  howpublished = {\url{https://kskedlaya.org/papers/affine-cone.pdf}},
  year         = {2025},
}

@article {HP04,
    AUTHOR = {Hartl, Urs and Pink, Richard},
     TITLE = {Vector bundles with a {F}robenius structure on the punctured
              unit disc},
   JOURNAL = {Compos. Math.},
  FJOURNAL = {Compositio Mathematica},
    VOLUME = {140},
      YEAR = {2004},
    NUMBER = {3},
     PAGES = {689--716},
      ISSN = {0010-437X,1570-5846},
   MRCLASS = {13A35 (14G22 14H60)},
  MRNUMBER = {2041777},
MRREVIEWER = {Alessandra\ Bertapelle},
       DOI = {10.1112/S0010437X03000216},
       URL = {https://doi.org/10.1112/S0010437X03000216},
}

@article {Ric16,
    AUTHOR = {Richarz, Timo},
     TITLE = {Affine {G}rassmannians and geometric {S}atake equivalences},
   JOURNAL = {Int. Math. Res. Not. IMRN},
  FJOURNAL = {International Mathematics Research Notices. IMRN},
      YEAR = {2016},
    NUMBER = {12},
     PAGES = {3717--3767},
      ISSN = {1073-7928,1687-0247},
   MRCLASS = {14L17 (14F05 14M15 20G25 22E67)},
  MRNUMBER = {3544618},
MRREVIEWER = {Guy\ Rousseau},
       DOI = {10.1093/imrn/rnv226},
       URL = {https://doi.org/10.1093/imrn/rnv226},
}

@article {Kott84,
    AUTHOR = {Kottwitz, Robert E.},
     TITLE = {Stable trace formula: cuspidal tempered terms},
   JOURNAL = {Duke Math. J.},
  FJOURNAL = {Duke Mathematical Journal},
    VOLUME = {51},
      YEAR = {1984},
    NUMBER = {3},
     PAGES = {611--650},
      ISSN = {0012-7094,1547-7398},
   MRCLASS = {11R39 (11F70 11F72 22E55)},
  MRNUMBER = {757954},
MRREVIEWER = {Jean-Pierre\ Labesse},
       DOI = {10.1215/S0012-7094-84-05129-9},
       URL = {https://doi.org/10.1215/S0012-7094-84-05129-9},
}

@article {Dillery23b,
    AUTHOR = {Dillery, Peter},
     TITLE = {R{igid} {inner} {forms} {over} {global} {function} {fields}},
   JOURNAL = {J. Inst. Math. Jussieu},
  FJOURNAL = {Journal of the Institute of Mathematics of Jussieu. JIMJ.
              Journal de l'Institut de Math\'ematiques de Jussieu},
    VOLUME = {24},
      YEAR = {2025},
    NUMBER = {6},
     PAGES = {2181--2256},
      ISSN = {1474-7480,1475-3030},
   MRCLASS = {11F70 (11E72 11R58 11S37 14F20)},
  MRNUMBER = {4976572},
       DOI = {10.1017/S1474748025100972},
       URL = {https://doi.org/10.1017/S1474748025100972},
}

@book {Gir,
    AUTHOR = {Giraud, Jean},
     TITLE = {Cohomologie non ab\'elienne},
    SERIES = {Die Grundlehren der mathematischen Wissenschaften},
    VOLUME = {Band 179},
 PUBLISHER = {Springer-Verlag, Berlin-New York},
      YEAR = {1971},
     PAGES = {ix+467},
   MRCLASS = {14F20 (14L99 18D30 18F20 55B99 55F65)},
  MRNUMBER = {344253},
MRREVIEWER = {R.\ T.\ Hoobler},
}

@misc{Kot14,
  author       = {Kottwitz, Robert},
  title        = {{$B(G)$} for all local and global fields},
  year         = {2014},
  month        = Jan,
  note         = {arXiv:1401.5728 [math.RT]},
  url          = {https://arxiv.org/abs/1401.5728},
}

@misc{Taibi22,
  author       = {Ta{\"i}bi, Olivier},
  title        = {The local {L}anglands conjecture},
  year         = {2025},
  month        = Oct,
  note         = {arXiv:2510.00632 [math.NT]},
  url          = {https://arxiv.org/abs/2510.00632},
}

@article {Lafforgue18,
    AUTHOR = {Lafforgue, Vincent},
     TITLE = {Chtoucas pour les groupes r\'eductifs et param\'etrisation de
              {L}anglands globale},
   JOURNAL = {J. Amer. Math. Soc.},
  FJOURNAL = {Journal of the American Mathematical Society},
    VOLUME = {31},
      YEAR = {2018},
    NUMBER = {3},
     PAGES = {719--891},
      ISSN = {0894-0347,1088-6834},
   MRCLASS = {14G35 (11F70 14D24 14H60)},
  MRNUMBER = {3787407},
MRREVIEWER = {James\ H.\ Stankewicz},
       DOI = {10.1090/jams/897},
       URL = {https://doi.org/10.1090/jams/897},
}

@misc{GL18,
  author       = {Genestier, Alain and Lafforgue, Vincent},
  title        = {Chtoucas restreints pour les groupes r\'eductifs et
                  param\'etrisation de {L}anglands locale},
  year         = {2018},
  month        = Sep,
  note         = {arXiv:1709.00978 [math.AG]},
  url          = {https://arxiv.org/abs/1709.00978},
}

@article {Dil23a,
    AUTHOR = {Dillery, Peter},
     TITLE = {Rigid inner forms over local function fields},
   JOURNAL = {Adv. Math.},
  FJOURNAL = {Advances in Mathematics},
    VOLUME = {430},
      YEAR = {2023},
     PAGES = {Paper No. 109204, 100},
      ISSN = {0001-8708,1090-2082},
   MRCLASS = {11S37 (22E50)},
  MRNUMBER = {4617942},
MRREVIEWER = {Amiya\ Kumar\ Mondal},
       DOI = {10.1016/j.aim.2023.109204},
       URL = {https://doi.org/10.1016/j.aim.2023.109204},
}

@incollection {Lan97,
    AUTHOR = {Langlands, R. P.},
     TITLE = {Representations of abelian algebraic groups},
      NOTE = {Olga Taussky-Todd: in memoriam},
   JOURNAL = {Pacific J. Math.},
  FJOURNAL = {Pacific Journal of Mathematics},
      YEAR = {1997},
     PAGES = {231--250},
      ISSN = {0030-8730,1945-5844},
   MRCLASS = {11R39 (22E55)},
  MRNUMBER = {1610871},
MRREVIEWER = {Volker\ J.\ Heiermann},
       DOI = {10.2140/pjm.1997.181.231},
       URL = {https://doi.org/10.2140/pjm.1997.181.231},
}

@article {HKW,
    AUTHOR = {Hansen, David and Kaletha, Tasho and Weinstein, Jared},
     TITLE = {On the {K}ottwitz conjecture for local shtuka spaces},
   JOURNAL = {Forum Math. Pi},
  FJOURNAL = {Forum of Mathematics. Pi},
    VOLUME = {10},
      YEAR = {2022},
     PAGES = {Paper No. e13, 79},
      ISSN = {2050-5086},
   MRCLASS = {14G45 (11S37)},
  MRNUMBER = {4430954},
MRREVIEWER = {Kazuhiro\ Ito},
       DOI = {10.1017/fmp.2022.7},
       URL = {https://doi.org/10.1017/fmp.2022.7},
}

@incollection {Art02,
    AUTHOR = {Arthur, James},
     TITLE = {A note on the automorphic {L}anglands group},
      NOTE = {Dedicated to Robert V.\ Moody},
   JOURNAL = {Canad. Math. Bull.},
  FJOURNAL = {Canadian Mathematical Bulletin. Bulletin Canadien de
              Math\'ematiques},
    VOLUME = {45},
      YEAR = {2002},
    NUMBER = {4},
     PAGES = {466--482},
      ISSN = {0008-4395,1496-4287},
   MRCLASS = {11R39 (22E55)},
  MRNUMBER = {1941222},
MRREVIEWER = {Mahdi\ Asgari},
       DOI = {10.4153/CMB-2002-049-1},
       URL = {https://doi.org/10.4153/CMB-2002-049-1},
}

@article {Kal18,
    AUTHOR = {Kaletha, Tasho},
     TITLE = {Rigid inner forms vs isocrystals},
   JOURNAL = {J. Eur. Math. Soc. (JEMS)},
  FJOURNAL = {Journal of the European Mathematical Society (JEMS)},
    VOLUME = {20},
      YEAR = {2018},
    NUMBER = {1},
     PAGES = {61--101},
      ISSN = {1435-9855,1435-9863},
   MRCLASS = {11E72 (11S37 22E50)},
  MRNUMBER = {3743236},
MRREVIEWER = {Boris\ \`E.\ Kunyavski\u i},
       DOI = {10.4171/JEMS/759},
       URL = {https://doi.org/10.4171/JEMS/759},
}

@article {KT,
    AUTHOR = {Kaletha, Tasho and Ta\"ibi, Olivier},
     TITLE = {Global rigid inner forms vs isocrystals},
   JOURNAL = {Doc. Math.},
  FJOURNAL = {Documenta Mathematica},
    VOLUME = {28},
      YEAR = {2023},
    NUMBER = {4},
     PAGES = {765--826},
      ISSN = {1431-0635,1431-0643},
   MRCLASS = {11R34 (11F70)},
  MRNUMBER = {4705600},
MRREVIEWER = {\Dbar\cftil o{} Ng\d oc Di\cfudot ep},
       DOI = {10.4171/dm/916},
       URL = {https://doi.org/10.4171/dm/916},
}

@article {KS99,
    AUTHOR = {Kottwitz, Robert E. and Shelstad, Diana},
     TITLE = {Foundations of twisted endoscopy},
   JOURNAL = {Ast\'erisque},
  FJOURNAL = {Ast\'erisque},
    NUMBER = {255},
      YEAR = {1999},
     PAGES = {vi+190},
      ISSN = {0303-1179,2492-5926},
   MRCLASS = {22E55 (11F70 11R34 22-02 22E50)},
  MRNUMBER = {1687096},
MRREVIEWER = {Volker\ J.\ Heiermann},
}

@book {Iak22,
    AUTHOR = {Iakovenko, Sergei},
     TITLE = {Representations of the {K}ottwitz {G}erbes},
      NOTE = {Thesis (Ph.D.)--Rheinische Friedrich-Wilhelms-Universitaet
              Bonn (Germany)},
 PUBLISHER = {ProQuest LLC, Ann Arbor, MI},
      YEAR = {2021},
     PAGES = {50},
      ISBN = {979-8304-99947-2},
   MRCLASS = {99-05},
  MRNUMBER = {4890347},
       URL =
              {https://gateway.proquest.com/openurl?url_ver=Z39.88-2004&rft_val_fmt=info:ofi/fmt:kev:mtx:dissertation&res_dat=xri:pqm&rft_dat=xri:pqdiss:31909536},
}

@misc{FS21,
  author       = {Fargues, Laurent and Scholze, Peter},
  title        = {Geometrization of the local {L}anglands correspondence},
  year         = {2024},
  month        = Feb,
  note         = {arXiv:2102.13459 [math.RT]},
  url          = {https://arxiv.org/abs/2102.13459},
}

@article {HK21,
    AUTHOR = {Hamacher, Paul and Kim, Wansu},
     TITLE = {On {${G}$}-isoshtukas over function fields},
   JOURNAL = {Selecta Math. (N.S.)},
  FJOURNAL = {Selecta Mathematica. New Series},
    VOLUME = {27},
      YEAR = {2021},
    NUMBER = {4},
     PAGES = {Paper No. 75, 34},
      ISSN = {1022-1824,1420-9020},
   MRCLASS = {20G30 (11G09 11S25)},
  MRNUMBER = {4292785},
       DOI = {10.1007/s00029-021-00683-w},
       URL = {https://doi.org/10.1007/s00029-021-00683-w},
}

@book {NSW,
    AUTHOR = {Neukirch, J\"urgen and Schmidt, Alexander and Wingberg, Kay},
     TITLE = {Cohomology of number fields},
    SERIES = {Grundlehren der mathematischen Wissenschaften [Fundamental
              Principles of Mathematical Sciences]},
    VOLUME = {323},
   EDITION = {Second},
 PUBLISHER = {Springer-Verlag, Berlin},
      YEAR = {2008},
     PAGES = {xvi+825},
      ISBN = {978-3-540-37888-4},
   MRCLASS = {11R34 (11-02 11G45 11R23 11S20 11S25 11S31 12G05)},
  MRNUMBER = {2392026},
       DOI = {10.1007/978-3-540-37889-1},
       URL = {https://doi.org/10.1007/978-3-540-37889-1},
}

@article {Kaletha18,
    AUTHOR = {Kaletha, Tasho},
     TITLE = {Global rigid inner forms and multiplicities of discrete
              automorphic representations},
   JOURNAL = {Invent. Math.},
  FJOURNAL = {Inventiones Mathematicae},
    VOLUME = {213},
      YEAR = {2018},
    NUMBER = {1},
     PAGES = {271--369},
      ISSN = {0020-9910,1432-1297},
   MRCLASS = {11R34},
  MRNUMBER = {3815567},
MRREVIEWER = {Meng\ Fai\ Lim},
       DOI = {10.1007/s00222-018-0791-3},
       URL = {https://doi.org/10.1007/s00222-018-0791-3},
}

@unpublished{DLH26,
  author = {Li-Huerta, Siyan Daniel},
  title  = {Courbes et {F}ibr{\'e}s vectoriels en th{\'e}orie de {H}odge {$z$}-adique globale},
  year   = {2026},
  month  = Feb,
  note   = {with an Appendix by Peter Dillery. arXiv:2602.04978 [math.NT]},
}
\bibliographystyle{amsalpha}

\end{document}